\documentclass[10pt,reqno]{amsart}

\usepackage[affil-it]{authblk}
\usepackage{titling}
\usepackage[latin1]{inputenc}
\usepackage[english]{babel}
\usepackage{calligra}
\usepackage[OT1]{fontenc}
\usepackage{amsfonts}
\usepackage{amsmath}
\usepackage{amssymb}
\usepackage{amsthm}
\usepackage{faktor}
\usepackage{float}
\usepackage{enumerate}
\usepackage{color}
\usepackage{pdfpages}
\usepackage{esint}
\usepackage{hyperref}
\usepackage{cite}
\usepackage{fancyhdr}
\usepackage{ulem}
\usepackage{mathtools}
\usepackage{mathrsfs}

\usepackage{etoolbox}
\patchcmd{\section}{\scshape}{\bfseries}{}{}
\makeatletter
\renewcommand{\@secnumfont}{\bfseries}
\makeatother
\newcommand\testname{\textbf{ Abstract}}
\makeatletter
\newenvironment{abs}{%
    \small
    \begin{center}%
        {\textsc \testname\vspace{-.2em}\vspace{\z@}}%
    \end{center}%
    \quote
    }
   {\endquote}
\makeatother

\advance\hoffset by 5mm

\allowdisplaybreaks[4]

\usepackage{verbatim}
\usepackage[letterspace=150]{microtype}
\usepackage{multicol}
\usepackage[affil-it]{authblk}
\usepackage{titling}
\usepackage[latin1]{inputenc}
\usepackage{calligra}
\usepackage[OT1]{fontenc}
\usepackage[english]{babel}
\usepackage{amsfonts}
\usepackage{amsmath}
\usepackage{amssymb}
\usepackage{amsthm}
\usepackage{faktor}
\usepackage{float}
\usepackage{enumerate}
\usepackage{color}
\usepackage{pdfpages}
\usepackage{esint}
\usepackage{cite}
\usepackage{fancyhdr}
\usepackage{mathtools}
\usepackage{mathrsfs}
\usepackage{comment}

\DeclareMathOperator*{\Id}{Id}

\newcommand{\vphi}{\varphi}

\newcommand{\Div}{\mathrm{div}}

\newcommand{\ee}{\varepsilon}

\newcommand{\Cc}{\mathcal{C}}

\newcommand{\dd}{d}

\newcommand{\pre}{ p}
\newcommand{\uu}{u}

\newcommand{\s}{ \mathrm{s}}

\newcommand{\TT}{\mathbb{T}}

\newcommand{\I}{\mathcal{I}}

\newcommand{\EE}{\mathcal{E}}

\newcommand{\DD}{\mathfrak{D}}
\newcommand{\RR}{\mathbb{R}}
\newcommand{\ZZ}{\mathbb{Z}}
\newcommand{\R}{\mathbb{R}}
\newcommand{\NN}{\mathbb{N}}
\newcommand{\N}{\mathbb{N}}

\newcommand{\BB}{\dot{B}}
\newcommand{\Hh}{\dot{H}}
\newcommand{\Dd}{\dot{\Delta}}
\newcommand{\Sd}{\dot{S}}

\newcommand{\no}[2]{ \left\| #1 \right\|_{#2} } 
\newcommand{\ext}{\mathrm{ext}}
\newcommand{\eff}{\mathrm{eff}}

\newcommand{\To}{{\mathbb{T}^2}}
\newcommand{\esssup}{{\operatorname{ess~sup}}}
\newcommand{\Sp}{{\mathbb{S}^2}}

\newcommand{\eps}{\epsilon}

\newtheorem{theorem}{Theorem}[section]
\newtheorem{cor}[theorem]{Corollary}
\newtheorem{prop}[theorem]{Proposition}
\newtheorem{lemma}[theorem]{Lemma}
\newtheorem{definition}[theorem]{Definition}
\newtheorem{remark}[theorem]{Remark}

\DeclareFontFamily{OT1}{rsfs}{}
\DeclareFontShape{OT1}{rsfs}{m}{n}{ <-7> rsfs5 <7-10> rsfs7 <10-> rsfs10}{}
\DeclareMathAlphabet{\mycal}{OT1}{rsfs}{m}{n}

\definecolor{grey}{rgb}{0.85,0.85,0.85}
\date{}

\title{\Large 
	\textbf{\uppercase{Struwe-like solutions for an evolutionary model of magnetoviscoelastic fluids}}}

\author{Francesco De Anna$\,^1$,\quad Joshua Kortum$\,^2$,\quad Anja Schl\"omerkemper$\,^3$}

\affil{	\small{$\,^1$Institute of Mathematics, University of W\"urzburg, Germany}\\
		\small\textnormal{francesco.deanna@mathematik.uni-wuerzburg.de}
		\vspace{0.2cm}}
		
\affil{	\small{$\,^2$Institute of Mathematics, University of W\"urzburg, Germany}\\
		\small\textnormal{joshua.kortum@mathematik.uni-wuerzburg.de}
		\vspace{0.2cm}}

\affil{	\small{$\,^3$Institute of Mathematics, University of W\"urzburg, Germany}\\
		\small\textnormal{anja.schloemerkemper@mathematik.uni-wuerzburg.de}
		\vspace{0.2cm}}

\date{\today}		
 

\begin{document}
\maketitle 
 
\begin{abs}
	 In this work we investigate the existence and uniqueness of Struwe-like solutions for a system of partial differential equations modeling the dynamics of magnetoviscoelastic fluids. The considered system couples a Navier-Stokes type equation with a 
     dissipative equation for the deformation tensor and a Landau--Lifshitz--Gilbert type equation for the magnetization field. 
      
     \noindent 
      The main purpose is to establish a well-posedness theory in a two-dimensional periodic domain under standard assumption of critical regularity for the (possibly large) initial data. We prove that the considered weak solutions are everywhere smooth, except for a discrete set of time values.
      
      \noindent The proof of the uniqueness is based on suitable energy estimates for the solutions within a functional framework which is less regular than the one of the Struwe energy level. These estimates rely on several techniques of harmonic analysis and paradifferential calculus.
      
\end{abs}
 

 
\section{Introduction}
\noindent In this paper we are concerned with the analysis of the Cauchy problem associated to the following nonlinear system of  parabolic-type equations, modeling the motion of a magnetoviscoelastic fluid: 
\begin{equation} \label{main_equation}
\begin{cases}
	\, \partial_t \uu  + \uu \cdot \nabla \uu - \nu \Delta \uu + \nabla \pre =-\, \Div (\nabla M\odot \nabla M - W'(F) F^\top)+\mu_0 \nabla^\top H_{\rm ext}M  \hspace{1.5cm }&(0,T)\times \mathbb{T}^2 ,\\
	 \,\Div \, \uu = 0&(0,T)\times \mathbb{T}^2,\\
	 \,\partial_t F  + \uu \cdot \nabla F - \nabla \uu F  = \kappa \Delta F&(0,T)\times \mathbb{T}^2,\\ 
	\, \Div \, F^\top=0,\quad |M|^2 = 1& (0,T)\times \mathbb{T}^2, \\
	\, \partial_t M + (\uu \cdot \nabla ) M =  - M \wedge H_{\text{eff}} - M\wedge M\wedge H_{\text{eff}}  
	& (0,T)\times \mathbb{T}^2,\\
\end{cases}
\end{equation}
where $\TT^2= \Pi_{i=1}^2\RR/2\pi\ZZ$ is the periodic torus on $\RR^2$, $T>0$ is a positive time, $\uu = \uu(t,x)$ stands for the velocity field of the fluid on $\RR^2$, $\pre=\pre(t,x)$ is the scalar-valued function that represents the  pressure of the system, $F= F(t,x)$ is the deformation tensor of the fluid in $\RR^{2\times 2}$. Moreover $W(F)$ is the elastic energy density and $M= M(t,x)\in\RR^3$ stands for the magnetization attending values in  $\mathbb{S}^2$.  The effective magnetic field $H_{\text{eff}}$ satisfies
\begin{equation*}
    H_{\text{eff}} = \Delta M +\mu_0 H_{\rm ext}-\psi'(M),
\end{equation*}
for a suitable even polynomial $\psi(M)$ depending on $M$ and for  an external applied magnetic field $H_{\rm ext} = H_{\rm ext}(t,x)$.  The entire system expresses a strong coupling and competition between the three state variables $(\uu,\,F,\,M)$, where $\uu$ satisfies a forced Navier-Stokes equation, $F$ evolves along with the motion given by a dissipative transport equation with stretching effects for the deformation tensor, while $M$ fulfills the  Landau-Lifshitz-Gilbert (LLG) equation, i.e.\ $\eqref{main_equation}_5$. We supplement  system \eqref{main_equation} with the initial data
\begin{equation*}
	\uu_{|t=0} =\uu_0,\qquad F_{|t=0}=F_0,\qquad M_{|t = 0}=M_0,
\end{equation*}
for suitable periodic functions $(\uu_0,\,F_0,\,M_0)$ defined over $\TT^2$ (cf.\ Theorem~\ref{main-thm1}). In addition, we recall the definition of the conservative force related to the nonlinearity $(\nabla M\odot\nabla M)_{ij}=\partial_i M \cdot \partial_j M$. 

\noindent As main results of this paper, we will prove the existence and uniqueness of suitable weak solutions of system \eqref{main_equation} for a large class of initial data $(u_0,\,F_0,\,M_0)$ in Theorem~\ref{struweSolutions} and Theorem~\ref{thm:uniqueness}, respectively, see also Subsection~\ref{subsect-main-results}. We address the case of a periodic domain $\mathbb{T}^2$, in which the Fourier transform is available. 

\subsection{\bf Derivation of the model and related works}$\,$

\noindent  The mathematical study of system \eqref{main_equation} has a quite recent history, which we briefly describe in this paragraph. For a derivation of system \eqref{main_equation} we refer the reader to   \cite{Forster,MR3763923}. Indeed, the model can be achieved through an energetic variational approach, stating that the entire system of partial differential equations can be determined by the free energy of the material together with the overall kinematics and mechanisms of dissipation  (cf.\ \cite{MR3916774}). The Helmholtz free energy in magnetoelasticity can be expressed by 
\begin{equation*}
	\psi_{\rm free}(F,M) = \underbrace{\int_{\TT^2} \frac{|\nabla M|^2}{2}}_{\text{exchange energy}} + \underbrace{\int_{\TT^2} W(F)}_{\text{elastic energy}}- \underbrace{\mu_0\int_{\TT^2}M\cdot H_{\rm ext}}_{\text{Zeeman energy}}
	+
	\underbrace{ \int_{\TT^2} \psi(M) }_{\substack{\text{anisotropy and}\\  \text{ stray field energy}}}.
\end{equation*}
On the right-hand side, the function $\psi(M)$ stands for a polynomial depending on the magnetization  $M$ which reflects the anisotropy and, in two dimensions, the stray field energy, see \cite{gioiajames}.
The LLG equation represents the kinematics of the magnetization field $M$.
The particular structure of the $M$-equation automatically leads to the unitary constraint
\begin{equation}\label{unit-constr}
	| M(t,x) |= 1\qquad\text{for almost any}\quad (t,x)\in (0,T) \times \TT^2.
\end{equation}
It is well-known that the above condition generates one of the major difficulties for the analysis of the considered equations. 

\noindent When the magnetization field $M$ is assumed overall constant, system \eqref{main_equation} reduces to the so-called viscoelastic system, which has been highly treated in literature (cf.\ \cite{MR3320133, MR2165379, MR3277178, MR3048595}). In such a framework, it is common to assume that the deformation tensor $F$ satisfies the following convection equation:
\begin{equation*}
	\partial_t F+\uu\cdot \nabla F =\nabla u F .
\end{equation*}
This equation retains important difficulties associated to the stretching term due to the right-hand side of the above identity. To overcome these challenges, it is usual in literature to require   a smallness condition on the initial data, or to introduce an additional dissipative mechanism to the system. 
In \cite{kalousekAnja}, for instance, the authors dealt with System \eqref{main_equation} under the condition $\kappa = 0$, investigating the existence of so-called dissipative solutions. The dissipative mechanism was in these solutions reflected by the appearance of some dissipation defect, perturbing the overall equations. In this manuscript we stick to the system with $\kappa>0$, considering an additional diffusion operator in $F$, as it was introduced in \cite{MR3816746, MR3763923} in the context of magnetoelasticity.

\medskip\noindent 
An existence result concerning weak solutions for a system similar to  \eqref{main_equation} was obtained  in \cite{Forster}, while the uniqueness of these solutions was investigated  in \cite{MR3816746}. 
The considered system involves a Ginzburg-Landau penalization, allowing a simplification of the nonlinearities concerning the unitary constraint. More precisely, the Helmholtz free energy  is recasted by
\begin{equation}\label{GL-pen}
	\psi_{\rm free}(F, M) +\frac{\left(|M|^2-1\right)^2}{4\ee^2},
\end{equation}
for a suitable small parameter $\ee>0$. This formulation, allows somehow to unlock a global-in-time estimate of $\Delta M$, since the main nonlinearity $|\nabla M|^2 M$ in the LLG equation reduces to $-(|M|^2-1)M/\varepsilon^2$, which depends on lower regularity of $M$.
The authors proved that any finite energy initial data (meaning square-integrable data) generates a global-in-time weak solution 
on a bounded domain. These solutions satisfy suitable energy inequalities and are unique in dimension two. In \cite{MR3824719}, local well-posedness and blow-up criteria of classical solutions were obtained, still in the framework of the Ginzburg-Landau approximation \eqref{GL-pen}. 

\medskip\noindent
In this manuscript, however, we do not consider the Ginzburg-Landau relaxation and we preserve the genuine unitary constraint on $M$ together with its analytical challenges.

\noindent The main result in \cite{MR3763923} gives a partial answer to the construction of a global weak solution for system \eqref{main_equation}, when the unitary constraint on $M$ is preserved. Indeed, the authors showed that such a solutions exists whenever the initial data satisfies a smallness condition related to the energy norms and the initial magnetization has a slightly additional regularity than the one proposed by the energy inequality. Recently, local-in-time existence of classical solutions for system \eqref{main_equation} on a bounded domain as well as  a weak-strong uniqueness result was obtained in \cite{kalousek}.

\smallskip\noindent To the best of our knowledge, the well-posedness of global solutions for large data within the function space given by the energy has not been treated yet in literature. We hence provide in this manuscript a complete new analysis related to this problem. Nevertheless, we mention that system \eqref{main_equation} shares similarities with the Ericksen-Leslie equations for the evolution of the so-called nematic liquid crystals (cf.\ \cite{LIU2019108294,MR1329830}). In particular, the magnetization equation for $M$ satisfies a similar equation as the one of the director field driving the evolution of the molecular orientation of the liquid crystal. The main differences and difficulties in system \eqref{main_equation} result from the presence of an external force given through $H_\ext$, an  additional convection term in the LLG equation due to the gyromagnetic factor $M\wedge H_{\rm ext}$ and finally the appearance of the deformation tensor $F$ with energy density $W$, not being coercive in $\nabla F$. The presence of these top-order nonlinearities additionally complicates the analysis of our system. 

\subsection{\bf Main results}\label{subsect-main-results}$\,$

\noindent
Let us start by considering the mathematical framework, and more precisely by giving the definition of weak solution which is relevant to us. Our analysis relies basically on energy estimates. More specifically, we  look for weak solutions of Struwe type originally constructend by Struwe in \cite{struwe} for the harmonic map heat flow on Riemannian surfaces. Struwe-like solutions are weak solutions that are smooth away from a finite number of singularities (here in time)  and their behaviour at singular points can be characterized by  concentration of the local energy. We present here their formal definition and we refer to Section~\ref{app:LP} for a definition of the considered functional spaces.
\begin{definition}[Struwe-like solutions]\label{def:weak-sol} A triple $(\uu,\, F,\,M)$ is called a Struwe-like solution to system \eqref{main_equation} in $(0,T)\times \TT^2$ with initial data $(\uu_0,\,F_0,\,M_0)\in L^2(\TT^2)\times L^2(\TT^2)\times H^1(\TT^2)$ satisfying $\Div \; \uu_0 =0, \Div\; F_0^\top =0 $ and $ |M_0|=1$ if it fulfills
\begin{equation}\label{def:weak-sol-energy-space}
\begin{alignedat}{16}
	\uu 	&\in L^\infty(0,T,\, L^2(\TT^2))		&&\cap 	L^2(0,T;\,H^1(\TT^2))  		\quad\quad&&&&\text{with}\quad\quad  &&&&&&&&\partial_t \uu &&&&&&&&\in L^2(T_i, \widetilde T;\,H^{-1}(\TT^2)),\\
	F 	&\in L^\infty(0,T,\, L^2(\TT^2))				&&\cap 	L^2(0,T;\,\Hh^1(\TT^2))  		\quad\quad&&&&\text{with}\quad\quad  &&&&&&&&\partial_t F 	&&&&&&&&\in L^2(0,T;\,\Hh^{-1}(\TT^2)),\\
	M	&\in L^\infty(0,T,\, H^1(\TT^2))			&&\cap L^2(T_i, \widetilde T; \,H^2(\TT^2))			\quad\quad&&&&\text{with}\quad\quad  &&&&&&&&\partial_t  M  &&&&&&&&\in L^2(T_i, \widetilde T;\,L^2(\TT^2))	,
\end{alignedat}
\end{equation}
for any $\widetilde T\in [T_i, T_{i+1})$ and for a suitable ordered finite family of times $\{T_1,\dots , T_N\}$, with $T_0 = 0$ and $T_N\leq T$ . Furthermore the following identities are satisfied,
\begin{equation*}
	\int_{\TT^2} \uu(t,\,x)\cdot \nabla \phi(x)\dd x \,=\,0 \quad\quad\text{for a.e.}\quad t\in (0,T),
\end{equation*}
for any function $\phi$ in $\dot H^1(\TT^2)$, as well as
\begin{equation*}
\begin{aligned}
	-\int_0^T\int_{\TT^2} & \uu\cdot \partial_t \varphi  -\int_0^T\int_{\TT^2} \uu \otimes \uu :\nabla \varphi +\nu \int_0^T\int_{\TT^2} \nabla\uu:\nabla\varphi \\ &=\,
	-\int_{\TT^2} \uu_0(x)\cdot \varphi(0,x)
	+\int_0^T\int_{\TT^2} \nabla M \odot \nabla M : \nabla  \varphi  -\int_{\TT^2} W'(F) F^\top : \nabla  \varphi
	+\int_0^T\int_{\TT^2} \mu_0\nabla H_{\rm ext} M \cdot  \varphi 
\end{aligned}
\end{equation*}
for any smooth vector function $\varphi \in \mathfrak{D}([0,T)\times  \TT^2)$, with $\Div\,\varphi = 0$. Similarly, for any matrix function $\Xi \in \mathfrak{D}([0,T)\times \TT^2)$,
\begin{equation*}
\begin{aligned}
	-\int_0^T\int_{\TT^2}  F\cdot \partial_t \Xi  -\int_0^T\int_{\TT^2}  \uu \otimes F:\nabla \Xi -\sum_{i,j,k=1}^2\int_0^T\int_{\TT^2}  \uu_i F_{jk}\partial_j \Xi_{ik} +\kappa \int_0^T\int_{\TT^2}  \nabla F:\nabla\Xi \,=\,
		-\int_{\TT^2}  F_0\cdot  \partial_t \Xi.
\end{aligned}
\end{equation*}
Finally, the following identity must be satisfied pointwise almost everywhere in $ (0,T)\times \TT^2$,
\begin{equation*}
	\partial_t M + (\uu \cdot \nabla ) M =  - M \wedge H_\eff - M \wedge (M \wedge H_\eff), 
\end{equation*}
together with the initial condition $M(t,\cdot) \to M_0$ as $t \to 0^+$  in $L^2(\TT^2)$.

\noindent If $\{T_1,\dots , T_N\}$ is the minimum set for which the above relations holds true, then we say that $\{T_1,\dots , T_N\}$ is the set of singular times of the Struwe-like solution.
\end{definition}
\noindent Before describing our main result, we first state the assumptions we impose on the elastic energy density $W$ and anisotropic/stray-field energy $\psi$ of our system \eqref{main_equation}. We assume $W(\mathcal{R}F) = W(F)$ for any $\mathcal{R}\in SO(2)$ (thus $W'(\mathcal{R}\Xi) =\mathcal{R} W'(\Xi)$, for any $\Xi\in \R^{2\times 2}$) which reflects the commonly required frame-indifference in elasticity. Furthermore, we assume that $W\in C^2(\R^{2\times 2})$ is of second-order growth, more precisely there exists a constant $C_1>0$ such that
\begin{equation*}
    C_1|A|^2 \leq W(A) \leq \frac{1}{C_1} (|A|^2+1),\qquad \text{for any}\quad A \in \RR^{2\times 2}.
\end{equation*}
In addition, we further assume that $W'(0) = 0$ and that $W'(F)$ is of first-order growth, more precisely there exist two constants $\chi>0$ and $C_2>0$ such that
\begin{equation*}
    |W'(A)-\chi A |\leq C_2,\qquad \text{for all}\quad A \in \RR^{2\times 2},
\end{equation*}
and likewise $W''(A)$ is bounded, i.e.
\begin{equation*}
    |W''(A)|\leq C_3.
\end{equation*}
Finally, we assume that $W(F)$ is  a convex function. 
It is well known that this assumption reduces the application of system \eqref{main_equation} in real materials. However, the convexity of $W$ is one of the key ingredients for the existence of the Struwe-like solutions (cf.\ Section~\ref{sec:existence}).
Our assumptions on $\psi: \RR^3 \to \RR$ consist of 
$\psi$ being a quadratic non-negative polynomial which is even, i.e.\ 
$$ \psi(M) = \psi(-M)$$
for every $M \in \Sp$. Therefore we set
\begin{equation*}
    \sup_{M \in \Sp} |\psi'(M)| =: \sigma<+\infty,
\end{equation*}
which exists due to the compactness of $\Sp$.
Further we assume that $H_\ext \in C^1([0,T]\times \mathbb{T}^2)$.


\noindent The main results of this paper (i.e.\ Theorem~\ref{struweSolutions} and Theorem~\ref{thm:uniqueness}) can be summarised in the following theorem about the existence and uniqueness of Struwe-like solutions for system \eqref{main_equation}.
\begin{theorem}\label{main-thm1}
	Let $(u_0,\,F_0,\,M_0)\in L^2(\TT^2)\times L^2(\TT^2) \times H^1(\TT^2)$ be given initial data satisfying $\Div\, u_0 =0$ and $\Div F_0^\top = 0$ in the sense of distributions. Assume that $|M_0|\equiv 1$ almost everywhere in $\TT^2$. Then there exists a unique global Struwe-like solution for system \eqref{main_equation} as depicted in Definition~\ref{def:weak-sol}. Furthermore, there are two constants $\ee_0>0$ and $R_0>0$ such that for any singular time $T_i$, there is at least one singular point $y_i\in \TT^2$, characterized by the condition 
	$$ 	\limsup_{t\nearrow T_i} \int_{B_R(y_i)}|\nabla M(t,x)|^2\dd x\geq \ee_0, $$
	for any $R>0$ with $R\leq R_0$.
\end{theorem}

\noindent
\begin{remark}
    The above statement reflects an important aspect of each possible singularity of the solutions: for any singular time $T_i$ the irregularity is caused by the behavior of the  magnetization $M$. If $F$ and $M$ are ignored, the system reduces to the 2D-version of the Navier-Stokes equations, which produces no singularities in  the velocity field $\uu$. However, in system \eqref{main_equation}, the singularity formation in $M$ might also cause a blow-up in $\uu$. In fact, this was conjectured recently for the related Ericksen-Leslie system in \cite{lai}. 
\end{remark}

\begin{remark}
    From the arbitrariness of the life span $T>0$ both in system \eqref{main_equation} and in the above statement, we argue that any solution built by Theorem~\ref{main-thm1} can be uniquely extended to a weak solution which is defined globally in time. Then, the amount of singularities occurring in time can actually be infinite (but still discrete), with the family of singular times converging towards $+\infty$.  This can roughly be  explained by the presence of the external force in \eqref{main_equation} which keeps on feeding up the energy of the system, allowing the appearance of new singularities. Considering a finite time interval $(0,T)$ allows to preserve the typical characteristic of Struwe-like solutions, i.e.\ the finite amount of singular times.
\end{remark}

\noindent Concerning the existence result, the concept of Struwe-like solution was first introduced to address the Cauchy problem associated to the harmonic map heat flow defined on Riemannian surfaces. Struwe \cite{struwe} proved the existence and uniqueness of such solutions, in which a finite number of singularities can occur at most. The main principle that allows to develop this theory relies on the fact that whenever a singular point appears then a fixed ``loss of energy'' occurs. It is rather standard to assert that under a sufficient smallness condition on the energy of the system, a unique global-in-time classical solution exists. Hence, Struwe proved that under a sufficiently large amount of singular points in which a constant loss of energy occurs, the fluid enters into the energy level in which the solution becomes smooth.
An extension to the Ericksen-Leslie model driving the motion of liquid crystals in 2D was performed in \cite{hong, linlinwang}. Starting from these results, we extend here the technique of Struwe to System \eqref{main_equation}.

\noindent 
The existence being achieved, we will hence address the uniqueness of Struwe-like solutions for \eqref{main_equation}. The main difficulties associated with treating the uniqueness in Theorem~\ref{main-thm1} are related to the presence of the Navier-Stokes part as well as the presence of the magnetization equation for $M$. One can essentially think of the system as a highly non-trivial perturbation of a Navier-Stokes system. 

\noindent Even in 2D, the perturbation produced by the presence of the additional stress-tensor generates significant technical difficulties related in the first place to the weak norms available for the velocity field $\uu$ and the magnetization $M$. A rather common way of dealing with this issue is by using a weak norm for estimating the difference between the two weak solutions, a norm that is  below the natural spaces in which the weak solutions are defined (cf.\ \eqref{energyDifference}). This approach was used before in the context of a tensor model for the evolution of liquid crystals (cf.\  \cite{MR3599422,MR3576270}).

\noindent In our case, for technical convenience we  use a homogeneous Sobolev space, namely $\dot H^{-\frac{1}{2}}(\TT^2)$. The fact that the initial data for the difference is zero (i.e.\ $(\delta \uu,\,\delta F,\,\delta M)_{t=0}=(0,\,0,\,0)$, where $\delta$ denotes the difference of two states) helps in controlling the difference in this space with low regularity. However, one of the main reasons for choosing the homogeneous setting is a specific product law. The mentioned law states that the product is a bounded operator from $\Hh^{s_1}(\TT^2)\times \Hh^{s_2}(\TT^2)$ into $\Hh^{s_1+s_2-1}(\TT^2)$, for any $|s_1|,\,|s_2|\leq 1$ such that  $s_1+s_2$ is positive (cf.\ Lemma~\ref{lemma_product_homogeneous_sobolev_spaces}). We mention that evaluating the difference at regularity level $s=0$, i.e.\ in $L^2$, would only allow to prove a weak-strong uniqueness result. 
Working in a negative Sobolev space, $\dot H^s$ with $s\in (-1,0)$ allows to capture the uniqueness of Struwe-like solutions.  We expect that a similar proof would work in any $\dot H^s$ with $s\in (-1,0)$ and our choice $s=-\frac{1}{2}$ is just for convenience.
 
 \noindent
 Our main goal is to obtain a specific estimate that leads to the Osgood lemma, a generalization of the Gronwall inequality. Indeed the uniqueness reduces to an estimate of the following type:
\begin{equation*}
	\delta \EE(t) \leq \delta \EE(0) +\int_0^t f(s) \mu( \delta \EE(s))\dd s ,
\end{equation*}
where $\delta \EE(t)$ is a quantity controlling the difference between two solutions (cf.\ Proposition~\ref{prop:estEE})  and $f$ is a-priori in $L^1_{\rm loc}$, while $\mu(\cdot)$ is an Osgood modulus of continuity (cf.\ \eqref{modulus_of_continuity}).

\smallskip\noindent
Furthermore, there are some difficulties that are specific to our system. These are  of two different types, being related to:

\begin{itemize}
\item controlling the high nonlinearties related to the magnetization field $M$, in particular those that interact with $\uu$, such as 
\begin{equation*}
    ``u\cdot \nabla M"\; \text{in the LLG equation and}\; 
    ``\Div(\nabla M\odot \nabla M)"\;
    \text{in the momentum equation},
\end{equation*}
\item controlling the high derivatives of $M$, related to the unitary constraint, such as $``|\nabla M|^2M"$ in the LLG equation.
\end{itemize}

\smallskip
\noindent
The first difficulty is dealt with  by taking into account the specific feature of the coupling that allows for the cancellation of the ``worst" terms,  when considering certain physically meaningful combinations. This feature is explored in the next section where we perform the existence proof of Theorem~\ref{main-thm1} (cf.\ the energy inequality \eqref{energyInequality}). In what concerns the second difficulty, this is overcome by delicate harmonic analysis arguments which requires the periodic setting of the spatial domain. This type of analysis has been extensively used in several contexts and it is typical of Navier-Stokes type problems (cf.\ \cite{MR2290141,MR2038119,MR3824719}).

\smallskip
\noindent We finally present the structure of this paper: in Section~\ref{sec:existence} we address the existence of Struwe-like solutions for system \eqref{main_equation}, with some highlights about their behavior around singular points. Some of the extensive calculations for the existence proof are performed in Section~\ref{technicalcalculations1} in order to not interrupt the main argument. Section~\ref{app:LP} introduces some tools of harmonic analysis, such as a brief overview of the Littlewood-Paley theory, that will play a major role for the uniqueness part. Hence, in Section~\ref{sec:uniqueness} we perform a suitable energy estimate of Osgood type, that unlocks the uniqueness of Struwe-like solutions. In order to improve the readability, some further technical calculations are presented in Section~\ref{sec5}.


\section{Existence of weak solutions}\label{sec:existence}
\noindent
This section is devoted to the proof of the first statement of Theorem~\ref{main-thm1}, concerning the existence of Struwe-like solutions, as presented in Definition~\ref{def:weak-sol}. 
We first recall the statement we aim to prove. To this end we introduce the following simplified notation of the function spaces in which our solutions belong to 
\begin{align*}
	H(a,b):=  &\bigg\{ v: [a,b] \times \To \to \R^2: v \text{ measurable and } 
	 \esssup_{a\leq t\leq b} \int_\To |v|^2(t) + \int_a^b \int_\To |\nabla v|^2 < \infty  \bigg\},\\
	K(a,b):=  &\bigg\{ G: [a,b] \times \To \to \RR^{2 \times 2}: G \text{ measurable and }
	 \esssup_{a\leq t\leq b} \int_\To |G|^2(t) + \int_a^b \int_\To |\nabla G|^2 < \infty  \bigg\}
	\end{align*}
	and
	\begin{align*}
	V(a,b):= \bigg\{ M: [a,b] \times \To \to \Sp: & M \text{ measurable and
	}   \\
	& \esssup_{a\leq t\leq b} \int_\To |\nabla M|^2(t) + \int_a^b \int_\To |\nabla^2 M|^2
	+|\partial_t M|^2  < \infty \bigg\}.
\end{align*}
In what follows, generic constants are denoted by $C$, and $A\lesssim B$ means that $A\leq CB$. 
We can hence state our existence result in the following theorem.
\begin{theorem} \label{struweSolutions}
Let $(\uu_0,F_0,M_0) \in  L^2(\To) \times L^2(\To) \times
H^1(\To)$ be  given initial data satisfying $\Div \, \uu_0 = 0, \Div \, F_0^\top =0$ and $|M_0|=1$. Then there exists a weak solution $(\uu, F, M) : [0,T] \times \To \to \R^2 \times \RR^{2 \times 2} \times \Sp$ of \eqref{main_equation} with the above initial condition such that $(\uu,F) \in H(0,T) \times K(0,T)$ and 
$$ M \in \bigcup_{i=1}^N \bigcap_{\eps>0} V(T_{i-1}, T_i-\eps) \quad \cup \quad V(T_N,T)$$
for finitely many $0=T_0 < T_1<...<T_N$ in $[0,T)$. Moreover there are two constants $\varepsilon_0>0$ and $R_0>0$ such that for each $T_i>0$ there exists at least one singular point $y_i \in \To$ characterized by 
$$ 	\limsup_{t\nearrow T_i} \int_{B_R(y_i)}|\nabla M(t,x)|^2\dd x\geq \ee_0, $$
for any $R>0$ with $R\leq R_0$.
\end{theorem}
\begin{remark}
    The above statement requires a minor adjustment, we avoided for the sake of presentation: if $T$ itself is a time in which a singularity of $M$ occurs, then we slightly increase its value to ensure that $M$ is smooth in $T$.
\end{remark}

\noindent We start by presenting a brief overview of the strategy we  perform throughout this section.
\begin{enumerate}[1.]
    \item As first step, we show the existence of classical solutions that are defined locally-in-time. More precisely we prove the existence of a suitable time $T^*\in (0, T)$, for which a weak solution of \eqref{main_equation} exists within the following functional framework:   \begin{equation*}
        \uu,\,F \in \mathcal{C}([0,T^*]; L^2(\TT^2))\cap L^2(0, T^*; H^1(\TT^2)) \qquad 
        M\in \mathcal{C}([0,T^*]; H^1(\TT^2)\cap L^2(0, T^*; H^2(\TT^2)).
    \end{equation*}
    In particular, no singularity occurs within the time interval $(0,T^*)$. 
    \item We secondly extend the mentioned solution until a first singularity occurs at a time $T_1>0$ in $(T^*, T)$. In addition, we show that at the time $T_1$ the system ``looses'' a fixed amount of energy.
    \item We finally extend our solution on the right-hand side of the singularity and we recursively perform as in the previous points. We conclude the recursive procedure showing that no accumulation point is reached by the set of singular times, hence our solution can be smoothly extended to the entire interval $(0,T)$, with the exception of a finite amount of times in which a singularity can occur.
\end{enumerate}
The first part of the above strategy is developed in a rather standard manner: we begin with regularizing the initial data through a sequence $(\uu^m_0,\,F^m_0,\,M^m_0)_{m\in \NN}$ in $H^1(\TT^2)\times H^1(\TT^2)\times H^2(\TT^2)$ such that 
\begin{equation*}
    \uu^m_0\to\uu_0\quad\text{in}\quad L^2(\TT^2),\qquad F^m_0\to F_0\quad\text{in}\quad L^2(\TT^2),\qquad
    M^m_0\to M_0 \quad\text{in}\quad H^1(\TT^2).
\end{equation*}
Hence, we claim that a standard application of a fixed-point theorem (cf.\ Lemma~\ref{strongSolutions} and also \cite{MR3763923}) implies that there exists a time $T^m>0$ and a unique local-in-time classical solution $(\uu^m,\,F^m,\,M^m)$ of \eqref{main_equation} with
\begin{equation*}
    (\uu^m ,\,F^m)\in \mathcal{C}([0,T^m], H^1(\TT^2))\cap L^2(0,T^m; H^2(\TT^2)),\qquad
    M^m\in \mathcal{C}([0,T^m], H^2(\TT^2))\cap L^2(0, T^m; H^3(\TT^2)).
\end{equation*}
We hence proceed to prove that there exists a time $T^*\in (0,T)$, that provides a lower bound for an unbounded family of the lifespans $T^m>0$ of the approximate solutions,i.e.\ $T^*<T^m$, for any $m\in \NN$, up to a subsequence. To this end we perform below some suitable a-priori estimates whose calculations are justified since $(\uu^m,F^m,M^m)_{m\in\NN}$ is a strong solution. 

\noindent 

\noindent
The proof of Theorem~\ref{struweSolutions} follows the strategy used in \cite{hong} and \cite{struwe}, when treating the well-posedness of the harmonic map heat flow and the Ericksen-Leslie system for liquid crystals, respectively. 

\noindent
We define two constants depending on the initial data and the external force, describing suitable behaviours of the energy of the system
$$ E_0:= \frac12\int_\To |\uu_0|^2 + 2 W(F_0) + |\nabla M_0|^2 + 2\psi(M_0), \quad K(H_\ext):=K_T(H_\ext): = 2\mu_0 ( \no{H_\ext}{L^\infty} + 
T \no{\partial_t H_\ext}{L^\infty}).$$
Next, we recall the basic energy law that the system satisfies. Indeed, multiplying the momentum equation of \eqref{main_equation} by $\uu$, the equation of the deformation tensor by $W'(F)$, the magnetization equation by $-H_\eff$ and integrating over $[0,t] \times \To$, we obtain
\begin{align*} 
\begin{split}
	\int_\To |\uu(t)|^2 &+ 2W(F)(t) + |\nabla M(t)|^2 + 2\psi(M)(t) -2 \mu_0 (M\cdot H_\ext)(t)\\ 
	&+ 2 \int_0^t \int_\To \nu  |\nabla \uu|^2 + \kappa \nabla F : \nabla W'(F)
	+ H_\eff^2 -|M \cdot H_\eff|^2  \\
	 & = \int_\To |\uu_0|^2 + 2W(F_0) + |\nabla M_0|^2 + 2\psi(M_0) - 2 \mu_0  M_0 \cdot H_\ext(0)  -2\mu_0 \int_0^t \int_\To  M \cdot
	\partial_t H_\ext .
	 \end{split}
\end{align*}
We remark the presence of dissipative terms in the second line of the above identity. For what concerns the velocity field $\uu$, the dissipation unlocks a suitable bound for $\nabla \uu$ as in the classical case of the Navier-Stokes equations, while for the magnetization $M$ a more elaborate analysis is required. For the deformation tensor $F$, the term involving $\nabla F$ just satisfies $\nabla F: \nabla W'(F)\geq 0$ since $W$ is assumed convex, leading to the following energy inequality
\begin{align} \label{energyInequality}
\begin{split}
	\int_\To |\uu(t)|^2 &+ 2W(F)(t)+ |\nabla M(t)|^2 + 2 \psi(M)(t) + 2 \int_0^t \int_\To \nu  |\nabla \uu|^2 
	 + H_\eff^2 -|M \cdot H_\eff|^2   \\
	 &\leq 2 E_0 + K(H_\ext).
	 \end{split}
\end{align}
For the forthcoming analysis we nevertheless need a bound for $\nabla F$ in a suitable Lebesgue space, a feature which is not provided by our last estimate. Hence, in order to compensate the lack of dissipation on $\nabla F$, we test the equation for the deformation tensor directly by $F$ and use Ladyzhenskaya's inequality to gather
\begin{align*}
    \frac12 \frac{\dd }{\dd t} \int_\To  |F|^2(t) + \int_\To \kappa |\nabla F|^2& \leq 
    \int_\To |\nabla \uu||F|^2 
    \leq 
    \| \nabla \uu \|_{L^2}
    \| F          \|_{L^4}^2
    \leq 
    \| \nabla \uu \|_{L^2}
    \| F          \|_{L^2}
    \|\nabla F    \|_{L^2}
    \\ 
    &\leq \frac{\kappa}{2} \no{\nabla F}{L^2}^2 + C \no{\nabla \uu}{L^2}^2 \int_\To |F|^2(t).
\end{align*}
Thus, a rather standard application of Gronwall's inequality leads to 
\begin{equation}\label{estimate-nablaF-L}
\int_0^t \int_\To |\nabla F|^2 \leq \| F_0 \|_{L^2}^2 \exp \Big\{ C\int_0^t \| \nabla \uu(s) \|_{L^2}^2\dd s\Big\}\leq 2 E_0 \exp( CE_0)=: L,    
\end{equation}
for any $0\leq t \leq T$.

\noindent
In contrast to Struwe-like solutions for harmonic map flow, we need to deal with the exterior force in terms of the external magnetic field $H_{\rm ext}$. In particular, additional energy is fed into the system and the energy functional might increase in time. This implies that estimate \eqref{energyInequality} is too weak to rule out the possibility of infinitely many singularities appearing in time. Because of that, we seek for additional bounds related in primis to the magnetization field $M$. We hence derive the following  estimate from \eqref{main_equation} and \eqref{energyInequality}:
\begin{align*}
	\no{\partial_t M }{L^2(0,T; L^\frac43)} \lesssim \no{(\uu \cdot \nabla ) M}{L^4(0,T; L^\frac43)} + 
	 \no{ -M \wedge (M \wedge H_\eff) }{L^2(0,T;L^2)} \lesssim 2E_0 + K(H_\ext),
\end{align*}
which implies the following bound of the total energy of the system:
\begin{align} \label{energyLawVariant}
\begin{split}
	\int_\To |\uu(t_2)|^2 &+ 2W(F)(t_2) + |\nabla M(t_2)|^2 + 2 \psi(M)(t_2) + 2 \int_{t_1}^{t_2} \int_\To \nu  |\nabla \uu|^2  + H_\eff^2 -|M \cdot H_\eff|^2   \\
	 & = \int_\To |\uu(t_1)|^2 + 2W(F)(t_1) + |\nabla M(t_1)|^2 + 2\psi(M)(t_2) +  2\mu_0 \int_{t_1}^{t_2} \int_\To  \partial_t M \cdot
	 H_\ext \\
	 & \leq \int_\To |\uu(t_1)|^2 + 2W(F)(t_1) + |\nabla M(t_1)|^2 + 2\psi(M)(t_1)  + K_E \sqrt{t_2-t_1}
	 \end{split}
\end{align}
for all $0\leq t_1\leq t_2\leq T$ with $K_E = K_E( E_0, H_\ext, T)$. Estimate \eqref{energyLawVariant}
plays a major role to show that only finitely many singularities of the solution appear at most. To this end, the dependence of $K_E$ on just the data and $T$ will become crucial. 

\noindent
We are now in the position to address the strategy introduced at the beginning of this section. We recall that $(\uu^m,\, F^m,\, M^m)_{m\in\NN}$ stands for a sequence of approximate solutions, whose existence is given by the following lemma. 
\begin{lemma}[\!\!\cite{kalousek}, Theorem 2.4] \label{strongSolutions}
For given initial data $(\uu_0^m,F_0^m,M_0^m) \in H^1(\To) \times H^1(\To) \times H^2(\To)$ there exists a $T^m>0$ 
and a unique strong solution  $(\uu^m,\pre^m,F^m,M^m): [0,T^m] \times \To \to \R^2 \times \R \times \R^{2\times 2} \times \Sp$ of \eqref{main_equation} such that
\begin{align*}
(\uu^m,F^m,M^m) & \in \Cc([0,T^m];H^1(\To) \times H^1(\To) \times H^2(\To)) \,\,\,\,\cap \,\,\,\, L^2(0,T^m; H^2(\To) \times H^2(\To) \times H^3(\To)) ,\\
\pre^m  & \in L^2 (0,T^m;H^1(\To)).
\end{align*}
\end{lemma}
\noindent 
We clarify that Theorem 2.4 in \cite{kalousek} is stated within the framework of a general bounded domain in $\R^2$, where boundary conditions are further considered. However this result holds also for the periodic setting $\TT^2$, since the lack of a boundary simplifies somehow the overall problem.

\noindent
Without loss of generality, we can assume that $T^m$ is the lifespan of each approximate solution. 
Regarding \eqref{main_equation}, the limit passage of some of the nonlinear terms involving $M$ is made possible by the control of the second gradient of $M$. To this end, the following lemma shows that the loss of smoothness of our approximate solutions is characterized by the blow-up of the $L^2$-norm of $\nabla^2 M$.
\begin{lemma}[Blow-up criterion] \label{blowupLemma}
Suppose that $(\uu^m,F^m,M^m)$ is a strong solution to \eqref{main_equation}. Then the following inequality holds true: 
\begin{align}	\label{blowupCriterion}
\begin{split}
 \frac{\dd}{\dd t} &\left( \no{ \nabla \uu^m}{L^2}^2 + \no{ \nabla F^m}{L^2}^2 +\no{ \Delta M^m}{L^2}^2 \right)(t)  + 
\left( \nu \no{\Delta \uu^m}{L^2}^2 +  \kappa \no{\Delta F^m}{L^2}^2 +   \no{ \nabla \Delta M^m}{L^2}^2 \right) \\
 &\leq C\left(1 + \no{\uu^m}{L^2}^2 + \no{F^m}{L^2}^2 + \no{\nabla M^m}{L^2}^2 \right) \left(1+ \no{\nabla \uu^m}{L^2}^2 + \no{\nabla F^m}{L^2}^2 +
\no{\Delta M^m}{L^2}^2 \right)^2,
\end{split}
\end{align}
for a suitable constant $C$ that does not depend on the index $m$. In particular, the loss of regularity of the solution at the lifespan $T^m$ is characterized by 
$$\lim_{t \nearrow T^m} \int_0^{t} \no{\Delta M^m}{L^2}^2(s) \dd s = + \infty.$$
\end{lemma}
\begin{proof}
    This follows from Lemma~\ref{lemmaBlowUp2} in Section~\ref{technicalcalculations1}.
\end{proof}
\noindent Because of Lemma~\ref{blowupLemma}, which we will prove in the following section, the overall goal is to gain a bound on the critical $L^2$-norm
of $\Delta M^m$. We do so first introducing the following central lemma :
\begin{lemma}[\!\!\cite{struwe}, Lemma $3.1$] \label{lemmaStruwe}
There exists a constant $C_1$ such that for any $T>0$, $f \in H(0,T)$ and any $R>0$, 
\begin{align*}
\int_{[0,T] \times \To} |f|^4 \phi \leq C_1 
\left(\underset{0\leq t \leq T, x \in \To}{\esssup} \int_{B_R(x)}
 |f(t)|^2  \right) \left( \int_{[0,T] \times \To} |\nabla f|^2 \phi  + R^{-2} \int_{[0,T] \times \To} |f|^2 \phi \right)
\end{align*}
holds true for every $ \phi \in C_0^\infty(B_R(x))$ with $\phi(y)= \widetilde{\phi}(|y-x|)$ and $\widetilde{\phi}$ nonincreasing.
\end{lemma}
\noindent 
A partition of unity argument hence entails the following corollary. 
\begin{cor} \label{corollaryStruwe}
There exists a constant $C_1$ such that for any $f \in H(0,T)$ and any $R>0$,
\begin{align*}
\int_{[0,T] \times \To} |f|^4 \leq C_1 
\left(\underset{0\leq t \leq T, x \in \To}{\esssup} \int_{B_R(x)}
 |f(t)|^2  \right) \left( \int_{[0,T] \times \To} |\nabla f|^2 + R^{-2} \int_{[0,T] \times \To} |f|^2 \right)
\end{align*}
holds true.
\end{cor}
\noindent
Corollary~\ref{corollaryStruwe} unlocks a criterion for a bound on $\no{\Delta M}{L^2}$, in terms of a smallness condition on the equi-integrability of $\nabla M$, as expressed in the following lemma. 
\begin{lemma}		\label{lemmaSmall1}
Let $\widetilde T>0$ be a general positive time and $(\uu,F,M) \in H(0,\widetilde T)\times K(0,\widetilde T) \times V(0,\widetilde T)$ be a weak solution of \eqref{main_equation}. Then there exists a constant $\varepsilon_1>0$ such that if $$ \underset{0\leq t\leq \widetilde T, x \in \To}{\esssup} \int_{B_R(x)} |\nabla M(t)|^2 < \varepsilon_1$$ for a suitable $R>0$, then  the following estimate holds true:
$$\int_{[0,\widetilde T]\times \To} |\nabla \uu|^2  + |\Delta M|^2  \leq C \left[ (1+\widetilde TR^{-2})
(2E_0 + K(H_\ext) ) + \no{H_\ext}{L^2L^2}^2 + \widetilde T\right].$$
\end{lemma}
\begin{proof} We first remark that the following identity concerning the dissipation of the magnetization field holds true:
\begin{align*}
    H_\eff^2 - (M \cdot H_\eff)^2 =&    |\Delta M |^2 - |\nabla M|^4  + \mu_0^2 H_\ext^2 - \mu_0 (M\cdot H_\ext)^2 
+ 2 \mu_0 [ \Delta M \cdot H_\ext - (M \cdot \Delta M) ( M\cdot H_\ext) ] \\
&+ \psi'(M)^2- (\psi'(M)\cdot M )^2  - 2\Delta M \cdot \psi'(M) - 2\mu_0 H_\ext \cdot \psi'(M)  \\
&+ 2(\Delta M \cdot M) \cdot (M \cdot \psi'(M) )  + 2\mu_0 (H_\ext \cdot M) ( \psi'(M) \cdot  M )\\
\geq & \frac{1}{2}|\Delta M|^2-C|\nabla M|^4 - C\mu_0^2 |H_\ext|^2 -C\sigma^2.
\end{align*}
Thus, we combine the above inequality with \eqref{energyInequality} to gather 
\begin{align*}
	\int_{[0,\widetilde T]\times \To} \nu |\nabla \uu|^2 + |\nabla^2 M|^2  
	& \leq 2 \int_{[0,\widetilde T]\times \To}  \nu |\nabla \uu|^2  + 2\Big(H_\eff^2 - (M \cdot H_\eff)^2\Big) + 2|\nabla M|^4 + C\Big(|H_\ext|^2+\sigma^2\Big)  \\
	&\leq 2( 2 E_0 + K(H_\ext) ) +  \int_{[0,\widetilde T]\times \To} \widetilde C\Big(  |H_\ext|^2 + \sigma^2\Big) +  2|\nabla M|^4 .
\end{align*}
Now we use Corollary~\ref{corollaryStruwe} for $f= \nabla M$, set $\varepsilon_1= \frac{1}{4C_1}$
and employ \eqref{energyInequality} to arrive at the  assertion.
\end{proof}
\noindent 
Lemma~\ref{lemmaSmall1} becomes useful if one is able to uniformly control the exchange energy $|\nabla M|^2(t)$ on any suitable ball $B_R(x)$ of fixed radius $R$, i.e.\ if $\{ \nabla M(t) \}_{t\geq0}$ is equi-integrable. In general, this is not the case since the best control we can achieve from weak solutions is the following energy estimate 
$$ \esssup_{t \in [0,\widetilde T]} \int_\To |\nabla M(t)|^2 \leq C,$$
which yields a bound in the entire domain $\TT^2$. 
Nevertheless, to better understand the main feature presented by this criterion, we first analyse how the local energy term evolves in time.
\begin{lemma} \label{localEnergyEstimate}
Let $\widetilde T>0$ be a general positive time and $(\uu,F,M) \in H(0,\widetilde T) \times K(0,\widetilde T) \times V(0,\widetilde T)$ be a solution to \eqref{main_equation}
with initial condition $(\uu_0,F_0,M_0) \in L^2(\To) \times L^2(\To) \times H^1(\To)$. Then there exist constants $\varepsilon_1>0$ and $R>0$ such that if $$ \underset{0\leq t\leq \widetilde T, x \in \To}{\esssup} \int_{B_{2R}(x)} |\nabla M(t)|^2 < \varepsilon_1$$
then for any $t\in (0, \widetilde T)$, for any $x_0\in \TT^2$ and for any $R>0$,
\begin{align*}
\int_{B_R(x_0)}  &\Big( |\uu|^2(t) + \chi|F|^2(t) + |\nabla M|^2(t)  \Big)
\\
\leq  &
\left[\int_{B_{2R}(x_0)} \Big( |\uu_0|^2 + \chi |F_0|^2 + |\nabla M_0|^2\Big)\right]  + C_1 (1+t)R^2 + C_2 t \Big(1+\frac{1}{R^2}\Big) + 
 C_3 \frac{t^\frac{1}{3}}{R^\frac{2}{3}} \left( 1+ t\Big( 1+\frac{1}{R^2}\Big) \right)^\frac23,
\end{align*}
where $C_1$, $C_2$ and $C_3$ are three positive constants, that depend only on $\no{H_\ext}{W^{1,\infty}}$, $E_0$ and $K(H_\ext) = K_{\widetilde T}(H_\ext)=2\mu_0 ( \no{H_\ext}{L^\infty} + 
\widetilde T \no{\partial_t H_\ext}{L^\infty})$.
\end{lemma}
\begin{proof}
Let $\phi \in C_0^\infty(B_{2R}(x_0))$ be a cut-off function with $ \phi \equiv 1$ on $B_R(x_0)$ and $|\nabla \phi|\lesssim \frac{1}{R}, |\nabla^2 \phi |\lesssim \frac{1}{R^2}$ for all $R \leq R_0$. Using the local energy inequality provided by Lemma~\ref{localEnergy} of the next section, i.e.
\begin{align*}
\int_\To &\Big(|\uu(t)|^2 + \chi |F(t)|^2 +|\nabla M(t)|^2\Big) \phi^2 + \int_0^t \int_\To \Big(\nu|\nabla \uu|^2 + \kappa\chi |\nabla F|^2 + |\Delta M +|\nabla M|^2M|^2\Big) \phi^2  \\
\leq& \int_\To \Big(|\uu_0|^2+ \chi |F_0|^2 +|\nabla M_0|^2\Big) \phi^2 
+C\bigg\{  \int_0^t \int_\To \Big( |\uu|^2 + |F|^2 +|\nabla M|^2 +|\pre|\Big) |\uu ||\phi ||\nabla \phi| \\ 
 &+ \int_0^t \int_\To \Big( |\uu|^2 +|F|^2 + |\nabla M|^2 \Big) \Big(|\nabla \phi|^2 + |\phi||\nabla^2 \phi|\Big)  +  \int_0^t \int_\To |\nabla \uu ||F||\chi  F- W'(F) | \phi^2 \\ 
&+\int_0^t \int_\To |\uu|^2\phi^2 
+\int_0^t \int_\To \Big(|H_\ext|^2 +|\nabla H_\ext|^2+ |\psi'(M)|^2\Big)\phi^2\bigg\},
\end{align*}
we have that
\begin{align}
\int_{B_R(x_0)} & \Big( |\uu|^2(t) + \chi |F|^2(t) + |\nabla M|^2(t) \Big)  + \int_0^t \int_\To \Big(
\nu |\nabla \uu |^2 +\kappa \chi |\nabla F|^2 + |\Delta M + |\nabla M|^2 M|^2  \Big) \phi^2 \notag \\
\leq &\int_{B_{2R}(x_0)} \Big( |\uu_0|^2 + \chi |F_0|^2 + |\nabla M_0|^2  \Big)  \notag 
+ C \int_0^t \int_\To \Big( |\uu|^2 + |F|^2  + |\nabla M|^2  +|\pre| \Big)|\uu||\phi||\nabla \phi|\notag \\
&+ C\frac{t}{R^2} \Big(E_0 + K(H_\ext)+1\Big) +C\int_0^t \int_\To |\nabla \uu ||F||\chi  F- W'(F) | \phi^2+
   C_1 (1+t)R^2
  , \label{aux1}
\end{align}
for a constant $C_1$ depending on $H_\ext$ and $\psi'$ and some positive constant $C$. Hence, we proceed estimating any terms on the right-hand side.

\noindent We now introduce a small parameter $\delta >0$ that will be determined at the end of the proof. First we recall that $|W'(A) - \chi A|\leq C_2$ for any $A\in \mathbb{R}^{2\times 2}$ and a positive constant $C_2$. Hence we remark that   
\begin{align*}
     \int_0^t \int_\To |\nabla \uu| |F | |\chi F - W'(F) | \phi^2
     \leq 
     C_2\int_0^t \int_\To |\nabla \uu| |F |  \phi^2
     & \leq \delta \int_0^t \int_\To |\nabla \uu|^2 \phi^2 + \frac{C}{\delta} \int_0^t \int_\To \chi  |F|^2 \phi^2.
\end{align*}
Next,
thanks to Lemma~\ref{lemmaStruwe}, we obtain that
\begin{equation}\label{aux3}
\begin{aligned}
	\int_0^t \int_\To &\Big(|\uu|^2 + |F|^2 + |\nabla M|^2 \Big) |\uu||\phi||\nabla \phi |  \\
	\leq &\delta \int_0^t \int_\To \Big(|\uu|^4 + |F|^4 + |\nabla M|^4 \Big) \phi^2 + \frac{C}{\delta} \int_0^t \int_\To |\uu|^2 |\nabla  \phi|^2    \\
	\leq & C\delta  \underset{0\leq s\leq \widetilde T, x \in \To}{\esssup}
	\Big(
	    \int_{B_R(x)} |\uu(s)|^2 + \chi|F(s)|^2 + |\nabla M(s)|^2
	\Big)
	\bigg\{
	\int_0^t \int_\To \Big(|\nabla \uu|^2 + \chi |\nabla F|^2  + |\nabla^2 M|^2\Big) \phi^2 
	\\
	&+\frac{1}{R^2}
	\int_0^t \int_\To \Big(| \uu|^2 + \chi | F|^2  + |\nabla  M|^2\Big) \phi^2 
	\bigg\}
	+ \frac{C}{\delta} \frac{t}{R^2} \Big(E_0 + K(H_\ext) +1 \Big)
	\\
	\leq & 
	C \delta(E_0 + K(H_\ext) +1) \left[ \int_0^t \int_\To 
	\Big(|\nabla \uu|^2 + \chi |\nabla F|^2\Big)\phi^2\right] +
	C\frac{\delta }{R^2} t \Big(E_0 + K(H_\ext) +1 \Big)^2  \\ 
	& +C \delta \Big(E_0 + K(H_\ext) +1 \Big)
	\left[
	    \int_0^t \int_\To   |\nabla^2 M|^2  \phi^2
	\right]+ \frac{C}{\delta} \frac{t}{R^2} \Big(E_0 + K(H_\ext) +1 \Big),
\end{aligned}
\end{equation}
where we have applied \eqref{energyInequality}, as well as the assumptions on $W(F)$ in the introduction of this manuscript.
For the pressure $\pre$ we first note that 
$$  \pre = (-\Delta)^{-1} \Div \Div (W'(F) F^\top - \nabla M \odot \nabla M - \uu \otimes \uu)$$
and therefore, thanks to Lemma~\ref{lemmaSmall1},
\begin{align*} 
        &\int_0^t\int_{\To}|\pre|^2 
        \leq C \int_0^t\int_{\To}
        \Big( |\uu|^4 +|F|^4 +|\nabla M|^4\Big) \\
        &\leq C  \underset{0\leq s\leq \widetilde T, x \in \To}{\esssup}
        \Big( \int_0^t\int_{B_{2R}(x)}|\uu(s)|^2 +|F(s)|^2 +|\nabla M(s)|^2 \Big)
        \Big( \int_0^t\int_{\To}  |\nabla \uu|^2 +|\nabla F|^2 +|\nabla^2 M|^2
        \\&\hspace{10cm}+ 
            \frac{1}{R^2}\int_{\To} |\uu|^2 +|F|^2 +|\nabla M|^2\Big) \\
        & \leq 
        C \Big(E_0 + K_{\widetilde T}(H_\ext) + 1\Big) \Big(L+ \Big(1+ \frac{1}{R^2}\Big)
\Big(2E_0 + K_{\widetilde T}(H_\ext) \Big) +\\ 
&\hspace{7.5cm}+\no{H_\ext}{L^2L^2}^2 + t +
            \frac{t}{R^2}\Big(E_0 + K_{\widetilde T}(H_\ext) + 1\Big) \Big) \\
        & \leq 
        C \Big(E_0 + K_{\widetilde T}(H_\ext) +L+ 1\Big)^2 \Big( 1 + t\Big(1+ \frac{1}{R^2}\Big)\Big),
\end{align*}
where we have used the estimate $\| \nabla F \|_{L^2((0,t)\times \To)}\leq L $ of \eqref{estimate-nablaF-L}. Thus, we deduce that
\begin{align*} \notag
\int_0^t  &\int_\To |\pre||\uu||\phi||\nabla \phi| \leq \delta \int_0^t \int_\To |\uu|^4 \phi^2 + \frac{C}{\delta} \int_0^t \int_{B_{2R}(x_0)} 
|\pre|^\frac43 |\phi|^\frac23 |\nabla \phi |^\frac43 
\\ 
 & \leq C \delta
 \underset{0\leq s\leq \widetilde T, x \in \To}{\esssup} \Big( \int_{B_{2R}(x)}|\uu|^2\Big)
 \Big(
    \int_0^t \int_\To |\nabla \uu|^2 \phi^2 +\frac{1}{R^2}\int_0^t \int_\To |\uu|^2 \phi^2  
 \Big)
 +
 \frac{C}{\delta}\int_0^t \int_{B_{2R}(x_0)} |\pre|^\frac43 |\phi|^\frac23 |\nabla \phi |^\frac43
\\
& \leq C(E_0+K_{\widetilde T}(H_\ext)+1) \delta \int_0^t \int_\To \left( |\nabla \uu|^2 |\phi|^2 + \frac{|\uu|^2}{R^2} \right)  + \frac{C t^\frac13}{\delta R^\frac23} \left(
\int_0^t \int_{B_{2R}(x_0)} |\pre|^2 \right)^\frac23\\
& \leq C \delta \int_0^t \int_\To \left( |\nabla \uu|^2 |\phi|^2 + \frac{|\uu|^2}{R^2} \right)  + 
C\frac{t^\frac13}{R^\frac23} \left( 1+ t\Big( 1+\frac{1}{R^2}\Big) \right)^\frac23 \Big( E_0 + K_{\widetilde T}(H_\ext) +L+1\Big)^\frac43.
\end{align*}
A control on $\nabla^2 M $ in terms of $\Delta M$ is given by
$$ \int_\To |\nabla^2 M|\phi^2 \leq \int_\To |\Delta M|^2 \phi^2 + 4 \int_\To |\nabla^2 M| |\nabla M| |\phi||\nabla \phi|.$$
From Lemma~\ref{lemmaStruwe} it follows that 
\begin{align*}
\int_0^t \int_\To |\nabla^2 M|^2 \phi^2 &  \leq  C \int_0^t \int_\To \Big(\Big(|\Delta M + |\nabla M|^2 M |^2  + |\nabla M|^4\Big) \phi^2 + 
|\nabla M |^2 |\nabla \phi|^2 \Big) \\
 & \leq C \left( \varepsilon_1  \int_0^t \int_\To |\nabla^2 M |\phi^2  + \frac{1}{R^2} \int_0^t \int_\To |\nabla M |^2 (1+ \phi^2) + \int_0^s \int_\To |\Delta M + |\nabla M|^2 M|^2 \phi^2 \right),
\end{align*}
i.e., 
\begin{align}
\int_0^t \int_\To |\nabla^2 M|^2 \phi^2 \leq C \int_0^t \int_\To |\Delta M + |\nabla M|^2 M|^2 + \frac{Ct}{R^2} 
\Big(2E_0 + K(H_\ext)\Big) \label{aux5} .
\end{align}
Finally, for small enough $\delta$, we combine the previous inequalities to deduce that
\begin{align*}
&\int_{B_R(x_0)} |\uu|^2(t) + \chi |F|^2(t) + |\nabla M|^2(t)   
 \leq \int_{B_{2R}(x_0)} \Big( |\uu_0|^2 + \chi |F_0|^2 + |\nabla M_0|^2  \Big) +\\  
 &+ C \Big(E_0 + K_{\widetilde T}(H_\ext) +1 \Big)^2t \Big( 1 +\frac{1}{R^2}\Big)  +
 C\frac{t^\frac13}{R^\frac23} \left( 1+ t\Big( 1+\frac{1}{R^2}\Big) \right)^\frac23 \Big( E_0 + K_{\widetilde T}(H_\ext) +L+1\Big)^\frac43\!\!\!\!+
C_1 (1+t)R^2.
\end{align*}
Defining the constants
\begin{equation*}
    C_2 = C  (E_0 + K_{\widetilde T}(H_\ext) +1 )^2,\qquad 
    C_3 = C( E_0 + K_{\widetilde T}(H_\ext) +L+1)^\frac43,
\end{equation*}
we conclude the proof of the lemma.
\end{proof}
\begin{remark}
    We assert that Lemma~\ref{localEnergyEstimate} can be extended to the case of 
    an external magnetic field $H_\ext $ in $ H^1((0,T)\times \To)$. The choice of a more regular external field $H_\ext$ in $W^{1,\infty}$ has been made for the sake of clear presentation, since in this framework we can explicitly control the local energy of the system  within the ball $B_{2R}(x_0)$ with respect to   the radius $R>0$ the time $t\in (0,T)$ and the initial data $(\uu_0,F_0,M_0)$. 
\end{remark}

\noindent 
Under the previous considerations, we are finally in the position to address the proof of Theorem~\ref{struweSolutions}.

\begin{proof}[Proof of Theorem~\ref{struweSolutions}]
Let $(\uu_0,F_0,M_0) \in L^2(\To) \times L^2(\To) \times H^1(\To)$ be an initial datum satisfying $|M_0|=1$, $\Div \, \uu_0=0$ and $\Div_0 \;F^\top = 0$. As depicted at the beginning of this section, we consider a sequence $(\uu_0^m, F_0^m, M_0^m)_m \subset  H^1(\To) \times H^1(\To) \times H^2(\To)$ satisfying the constraints
$\Div \; \uu_0^m, \Div_0 \;(F^m)^\top = 0$ and $|M^m_0|=1$ and converging strongly in $L^2(\To) \times L^2(\To) \times H^1(\To)$ to $(\uu_0, F_0, M_0)$. Every triple $(\uu_0^m, F_0^m, M_0^m)$ generates a strong solution $(\uu^m,F^m,M^m)$ to \eqref{main_equation} on a maximal time interval $[0,T^m)$ according to Lemma~\ref{strongSolutions}. 

\noindent Without loss of generality, we just analyse the case in which $T^m < T$, for a subsequence that with an abuse of notation we still denote by $(T^m)_{m\in\NN}$. Indeed, when $T^m\geq T$ for any $m\in\mathbb{N}$, the lifespans are a-priori bounded from below.
By Lemma~\ref{blowupLemma}, the solution $(\uu^m,F^m,M^m)$ satisfies $\int_0^{T_m} \no{\Delta M^m}{L^2}^2 (s) \dd s= + \infty$. In turn, Lemma~\ref{lemmaSmall1} provides a bound on $\int_0^t \no{\Delta M_m}{L^2}^2 (s) \dd s$ if 
$$ \underset{0\leq s\leq t, x \in \To}{\esssup} \int_{B_R(x)} |\nabla M^m(s)|^2 < \varepsilon_1$$
for some $\epsilon_1>0$ and $R>0$. Lemma~\ref{localEnergyEstimate} yields for any $t\in (0,T^m)$ the estimate 
\begin{align} \label{aux10}
\begin{split}
\int_{B_R(x_0)}  & |\uu^m(t)|^2 + |F^m(t)|^2 + |\nabla M^m(t)|^2  \\ 
& \leq \int_{B_{2R}(x_0)} |\uu_0^m|^2 + |F_0^m|^2 + |\nabla M_0^m|^2  
  + C_1 (1+t)R^2 + C_2 t \Big(1+\frac{1}{R^2}\Big) + 
 C_3 \frac{t^\frac{1}{3}}{R^\frac{2}{3}} \left( 1+ t\Big( 1+\frac{1}{R^2}\Big) \right)^\frac23.
 \end{split}
\end{align}
We can choose an $R>0$ such that 
\begin{align*}
 \int_{B_{2R}(x_0)} |\uu_0^m|^2 + |F_0^m|^2 + |\nabla M_0^m|^2    < \frac{\eps_1 }{4 }, \qquad 
C_1 (1+t) R^2< \frac{\eps_1 }{4 }
\end{align*}
for all $t\in (0,T^m)$, $x_0 \in \To$ and $m \in \N$. Hence, we define 
\begin{align*} 
T^*: = \min \left\{  \frac{\varepsilon_1 R^2}{4C_2(1+R^2) } , \frac{\varepsilon_1^3R^2}{4^3 C_3^3 \Big(1+ T\Big(1+\frac{1}{R^2}\Big)\Big)^2} \right\}.
\end{align*}
Combining the above relations together with \eqref{aux10}, we obtain that
$$ \underset{0\leq s\leq \min\{T^*, T^m\}}{\esssup} \int_{B_R(x_0)} |\nabla M^m(s)|^2 < \varepsilon_1$$
for all $x_0 \in \To$ and $m \in \N$. Hence Lemma~\ref{lemmaSmall1} yields that
\begin{align*}
    \int_0^{\min\{T^*,T^m\}}  \no{\Delta M^m(s)}{L^2}^2  \dd s < + \infty,
\end{align*}
which can only be the case if $T^* <T^m$ for every $m \in \NN$.
Therefore the strong solution $(\uu^m, F^m, M^m)$ exists  on $[0,T^*]$, where we highlight that $T^*$ does not depend on $m$. Finally, we pass to the limit with the a-priori bounds given in energy inequality \eqref{energyInequality} and Lemma~\ref{lemmaSmall1} which yields local existence of a solution with
$$(\uu,F,M) \in \Cc_w([0,T^*]; L^2( \To) \times L^2(\To) \times H^1(\To)) \,\,\,\, \cap \,\,\,\, L^2 (0,T^*; H^1(\To) \times H^1(\To) \times H^2(\To))$$ for given initial
data $(\uu_0,F_0,M_0) \in L^2(\To) \times L^2(\To) \times H^1(\To)$. 

\noindent Because of uniqueness of these solutions (cf.\ Theorem~\ref{thm:uniqueness}), we can extend our solution up to a singular time $T_1\in (0,T)$,  characterized by the following relation
\begin{equation} \label{lossRegularity}
    \underset{0\leq t \leq T_1, x \in \To}{\esssup}\int_{B_R(x)} |\nabla M(t)|^2 \geq \varepsilon_1
\end{equation}
for any $R>0$. Since $(\uu, F, M) \in \Cc_w([0,T_1]; L^2 (\To) \times L^2(\To) \times H^1(\To))$, the solution $(\uu,\,F,\,M)$ is well-defined at the time $T_1$, in particular $(\uu,F,M) (T_1) \in L^2(\To) \times L^2(\To) \times H^1(\To)$ with $\Div \; \uu (T_1) =0 , \Div \; F^\top(T_1) =0, |M(T_1)|=1$. 
We hence claim that the following loss of energy occurs at this first time singularity:
\begin{equation} \label{energyLoss}
 2E_1:=
 \int_{\To} |\uu(T_1)|^2 + 
 2 W(F(T_1)) + |\nabla M(T_1))|^2+ 2\psi(M(T_1))
 \leq 2E_0 + K_E \sqrt{T_1} - \eps_1.
 \end{equation}
In other words, the external force feeds energy into the system through $K_E \sqrt{T_1}$ while the singularity decreases the total energy of the fixed amount given by $\ee_1>0$. To prove this statement we proceed by contradiction. We assume that \eqref{energyLoss} is false, then it follows by \eqref{energyLawVariant} that
\begin{align*}
    0 \leq \underbrace{\underset{0\leq t \leq T_1, x \in \To}{\esssup} \left( \no{\nabla M(t)}{L^2}^2 - \no{\nabla M (T_1)}{L^2}^2 \right) }_{=:a} < \underset{0\leq t \leq T_1 }{\esssup} \left( 2 E_0 + K_E \sqrt{t} - 2E_0 -K_E
    \sqrt{T_1} + \eps_1 \right) = \eps_1.
\end{align*}
Thus, for any $x \in \To$,
\begin{align*}
   \underset{0\leq t \leq T_1, x \in \To}{\esssup}\int_{B_R(x)} |\nabla M(t)|^2  
   &= 
   \underset{0\leq t \leq T_1, x \in \To}{\esssup}
   \left\{\int_{B_R(x)}\left(  |\nabla M(t)|^2 -|\nabla M(T_1)|^2\right) + 
    \int_{B_R(x)}|\nabla M(T_1)|^2 \right\}\\
    &\leq a + \underset{ x \in \To}{\esssup} \int_{B_R(x)} |\nabla M(T_1)|^2 <\eps_1
\end{align*}
for a sufficiently small radius $R>0$ which is a contradiction to \eqref{lossRegularity}.


\noindent Since $(\uu(T_1),\,F(T_1),\,M(T_1))$ is defined in $L^2(\TT^2)\times L^2(\TT^2) \times H^1(\TT^2)$, we can restart our entire procedure with this new set of initial data, extending our solution to a new time interval $[T_1,\,T_2]$, where a new singularity appears at time $T_2$:
\begin{align*}
    2E_2:=
 \int_{\To} |\uu(T_1)|^2 + 
 2 W(F(T_2)) + |\nabla M(T_2))|^2+ 2\psi(M(T_2))
    &\leq 2 E_1 + K_E\sqrt{T_2-T_1} - \ee_1 \\
    &\leq 2E_0 + K_E(\sqrt{T_1}+\sqrt{T_2 - T_1}) - 2 \ee_1,
\end{align*}
where $E_1$ stands for the value of the energy at the time $T_1$.

\noindent Then, we continue this procedure by recursion, leading to a unique solution on the intervals $[T_1, T_2], [T_2, T_3]$ and so on.
Next we show that just a finite amount of singularities occurs before reaching the final time $T$ of system \eqref{main_equation}. This is satisfied if we prove that there is no accumulation point for the set of any singular time $T_i$. To this end, we proceed by contradiction: we assume that there exists a sequence of singular times $(T_i)_{i\in\NN}$ such that $T_i < T_{i+1} < T$ (with an abuse of notation we set $T_0 = 0$).
By \eqref{energyLoss}, every finite maximal time of existence comes with the loss of $\varepsilon_1$ for the energy, more precisely 
\begin{equation} \label{finiteSingularities}
\begin{aligned}
    0\leq 
    2E_n:=
 \int_{\To} |\uu(T_n)|^2 + 
 2 W(F(T_n)) + |\nabla M(T_n))|^2+ 2\psi(M(T_n))
    &\leq   2E_0 + \sum_{i=1}^n \left( K_E \sqrt{T_i- T_{i-1}} - \eps_1\right)\\
    &\leq 2E_0 -n\ee_1 + K_E \sum_{i=1}^n \sqrt{T_i- T_{i-1}}.
\end{aligned}
\end{equation}
Since $T_i-T_{i-1}$ converges towards $0$, eventually 
$\ee_1>K_E \sqrt{T_i-T_{i-1}}$ and thus
the right-hand side of the above inequality converges towards $-\infty$ as $n$ goes to $+\infty$, which is a contradiction. Thus we deduce that only a finite amount of time singularities occur before the solution becomes smooth until the final time $T>0$. This concludes the proof of of Theorem~\ref{struweSolutions}.
\end{proof}

\section{\bf Some technical lemmata} \label{technicalcalculations1} 

\noindent 
In this section we provide the proofs of some technical results we used in the previous section. We begin with proving Lemma~\ref{blowupLemma}, which asserts that a classical solution of \eqref{main_equation} stays smooth as long as $\Delta M$ is $L^2$-integrable in time and space. We first recall here the statement, negleting the index notation depending on $m\in \NN$.
\begin{lemma} \label{lemmaBlowUp2}
Suppose that $(\uu,F,M)$ is a strong solution to \eqref{main_equation} in the interval $[0, T]$. Then the following inequality holds true:
\begin{equation}\label{blowupCriterion2}
\begin{split}
 \frac{\dd}{\dd t} &\left( \no{ \nabla \uu}{L^2}^2 + \no{ \nabla F}{L^2}^2 +\no{ \Delta M}{L^2}^2 \right)  + 
 \left( \nu \no{\Delta \uu}{L^2}^2 +  \kappa \no{\Delta F}{L^2}^2 +   \no{ \nabla \Delta M}{L^2}^2 \right) \\
 &\leq C\left(1 + \no{\uu}{L^2}^2 + \no{F}{L^2}^2 + \no{\nabla M}{L^2}^2 \right) \left(1+ \no{\nabla \uu}{L^2}^2 + \no{\nabla F}{L^2}^2 +
\no{\Delta M}{L^2}^2 \right)^2,
\end{split}
\end{equation}
for a suitable constant $C$. In particular, the loss of regularity of the solution at the time $T$ is characterized by 
$$\lim_{t \nearrow T} \int_0^{t} \no{\Delta M}{L^2}^2(s) \dd s = + \infty.$$
\end{lemma}
\begin{proof}
We first test the momentum equation of \eqref{main_equation} by $- \Delta \uu$. Hence recalling that $\Div \uu = 0$, we gather
\begin{align*}
	\frac12 \frac{\dd}{\dd t} \no{\nabla \uu}{L^2}^2 \!+\! \nu \no{\Delta \uu }{L^2}^2 &\! =\! \underbrace{ \int_\To (\uu \cdot \nabla) \uu \cdot \Delta \uu
	 -\mu_0 \nabla^\top H_\ext M \cdot \Delta \uu -  \Div (W'(F)F^\top) \cdot  \Delta \uu }_{ \sum_{i=1}^3 I_i }   + \int_\To (\Delta \uu \cdot \nabla ) M \cdot \Delta M .
\end{align*}
Next, we multiply the deformation tensor equation of \eqref{main_equation} by $-\Delta F$ and we integrate both in time and space, to gather
\begin{align*}
\frac12  \frac{\dd}{\dd t} \no{\nabla F}{L^2}^2 + \kappa \no{\Delta F }{L^2}^2  = \int_\To (\uu \cdot \nabla ) F : \Delta 
F - \nabla \uu F : \Delta F = \sum_{i=4}^5 I_i .
 \end{align*}
Additionally, we multiply the magnetization equation of \eqref{main_equation} by $\Delta^2 M$ and we integrate the result in $\mathbb{T}^2$, 
\begin{align*}
\frac12  \frac{\dd}{\dd t} \no{\Delta M}{L^2}^2 + \no{\nabla \Delta M }{L^2}^2  = -\int_\To  (\Delta u \cdot \nabla) M \cdot \Delta M + \sum_{i=6}^{15}I_i,
\end{align*}	
where 
	\begin{align*}
		I_6&=-2\sum_{k=1}^2\int_\To (\partial_k \uu \cdot\nabla) \partial_k M \cdot\Delta  M,\\
		I_7&=-\int_\To (\uu \cdot\nabla) \Delta M \cdot \Delta M,\\
		I_8&=-\int_\To  2\left((\nabla^2 M \nabla M )\otimes M \right)\cdot\nabla\Delta M ,\\
		I_9&=-\int_\To |\nabla M|^2\nabla M \cdot \nabla \Delta M ,\\
		I_{10}&=\int_\To \left(\nabla M \wedge(\Delta M + H_\ext)\right) \cdot \nabla \Delta M,\\
		I_{11}&=\int_\To (M \wedge \nabla H_\ext)\cdot \nabla \Delta M,\\
		I_{12}&=\int_\To (M \cdot H_\ext)\cdot(\nabla M \cdot \nabla \Delta M ),\\
		I_{13}&=\int_\To \left((\nabla\Delta M)^\top M\right)\cdot\left((\nabla M)^\top  H_\ext+(\nabla H_\ext)^\top M\right),\\
		I_{14}&=\int_\To \nabla H_\ext \cdot \nabla \Delta M, \\
		I_{15}&=-\int_\To \nabla ( M \wedge \psi'(M) + M \wedge M \wedge \psi'(M) ) \cdot \nabla \Delta M. 
	\end{align*}
We hence proceed estimating the terms $I_i$ by H\"older's, Ladyzhenskaya's,  Young's and interpolation inequalities. We obtain that
\begin{align*}
I_1 & \leq \no{\uu}{L^4} \no{\nabla \uu}{L^4} \no{\Delta \uu}{L^2} \lesssim
\delta  \no{\Delta \uu}{L^2}^2  + \no{\uu}{L^2} \no{\nabla \uu}{L^2}^2 \no{\Delta \uu}{L^2} \\
	& \lesssim \delta \no{\Delta \uu}{L^2}^2 +  \no{\uu}{L^2}^2 \no{\nabla \uu}{L^2}^4 \\
	I_2& \lesssim \delta \no{ \Delta \uu}{L^2}^2 + \no{\nabla H_\ext}{L^2}^2 \\
	I_3 & \lesssim \no{F}{L^4} \no{\nabla F}{L^4} \no{\Delta \uu}{L^2}
	\lesssim \delta  \no{\Delta \uu}{L^2}^2  + \no{F}{L^2} \no{\nabla F}{L^2}^2 \no{\Delta F}{L^2} \\
	& \lesssim \delta \left(  \no{\Delta \uu}{L^2}^2  + \no{\Delta F}{L^2}^2 \right)+  \no{F}{L^2}^2 \no{\nabla F}{L^2}^4,
\end{align*}
where in the last inequality we have used that $|W''(F)|\leq C_3$ and $|W'(F)|\leq C_2(1+|F|)$, as assumed in the introduction of this manuscript. Furthermore, we get that
\begin{align*}
	I_4 & 
	\lesssim \delta \no{\Delta F}{L^2}^2 + \no{\uu}{L^2}^2 \left( \no{\nabla \uu}{L^2}^4 +  \no{\nabla F}{L^2}^4\right) \\
	I_5 &  \leq \delta \left( \no{\Delta \uu }{L^2}^2 + \no{\Delta F}{L^2}^2 \right) + \no{F}{L^2}^2 \no{ \nabla \uu}{L^2}^2 \no{ \nabla F}{L^2}^2 \\
	I_6 &  \leq 2 \int |\nabla \uu | |\nabla^2 M||  \Delta M| \lesssim \no{\nabla \uu}{L^2} \no{ \Delta M}{L^4}^2 
		\lesssim \delta \no{\nabla \Delta M}{L^2}^2 + \left(\no{\nabla \uu}{L^2}^4 +\no{\Delta M}{L^2}^4\right).
\end{align*}
An integration by parts yields that $I_7 = 0$. Similarly as before we obtain that
\begin{align*}
	I_8 & \leq \int |\nabla M ||\nabla^2 M||\nabla \Delta M|^2 \lesssim \delta \no{\nabla \Delta M}{L^2}^2
	 + \no{\nabla M }{L^2}^2 \no{\Delta M}{L^2}^4 \\
	 I_9 & \leq \int |\nabla M|^3 |\nabla \Delta M| \lesssim \delta \no{\nabla \Delta M}{L^2}^2 + 
	  \no{\nabla M}{L^2}^2 \no{\Delta M}{L^2}^4,
\end{align*}
where we employed interpolation of $L^6$ into $H^1$. Recalling the assumptions on $H_\ext$, we derive similarly to above that
\begin{align*}
	 I_{10} & \lesssim \delta \no{\nabla \Delta M}{L^2}^2
	 + \no{\nabla M }{L^2}^2 \no{\Delta M}{L^2}^4  + 1\\ 
	 I_{11} & \lesssim  \delta  \no{ \nabla \Delta M}{L^2}^2 +  1\\
	 I_{12} &\lesssim \delta \no{\nabla \Delta M}{L^2}^2
	 + \no{\nabla M }{L^2}^2 \left(\no{\Delta M}{L^2}^4 +1\right) \\
    I_{13} &  \lesssim \delta \no{\nabla \Delta M}{L^2}^2
	 + \no{\nabla M }{L^2}^2\no{\Delta M}{L^2}^4 +1  \\
    I_{14} & \lesssim \delta  \no{ \nabla \Delta M}{L^2}^2 +  1 \\
    I_{15} & \lesssim \delta  \no{ \nabla \Delta M}{L^2}^2 +   \no{ \nabla  M}{L^2}^2.
\end{align*}
Summarizing, we deduce that \eqref{blowupCriterion2} holds true. Moreover, we remark that the Gronwall inequality implies that the
$L^\infty (0,\widetilde T;H^1\times H^1 \times H^2 ) \cap L^2(0,\widetilde T; H^2 \times H^2 \times H^3)$-norm of $(\uu,F,M)$ remains bounded as long as $\int_0^{\widetilde T} \no{\Delta M(s)}{L^2}^2 \dd s$ is finite.
\end{proof}
\noindent Next, we provide an additional lemma that allows to control the local energy of the system within a specific spatial support. This lemma plays a major role in the proof of the local energy inequality of Lemma~\ref{localEnergyEstimate}.
\begin{lemma} \label{localEnergy}
Let $\widetilde{T}>0$ be a general positive time and let $(\uu,F,M) \in H(0,\widetilde T) \times K(0,\widetilde T) \times V(0,\widetilde T)$ be a weak solution of \eqref{main_equation} with initial condition $(\uu_0,\,F_0,\, M_0)$. Further, let $\phi$ be a smooth function on $\To$. Then the following identity holds true
\begin{align*}
\int_\To &\Big(|\uu(t)|^2 + \chi |F(t)|^2 +|\nabla M(t)|^2\Big) \phi^2 + \int_0^t \int_\To \Big(\nu|\nabla \uu|^2 + \kappa\chi |\nabla F|^2 + |\Delta M +|\nabla M|^2M|^2\Big) \phi^2  \\
\leq & \int_\To \Big(|\uu_0|^2+ |F_0|^2 +|\nabla M_0|^2\Big) \phi^2 
+C\bigg\{  \int_0^t \int_\To \Big( |\uu|^2 + |F|^2 +|\nabla M|^2 +|\pre|\Big) |\uu ||\phi ||\nabla \phi| \\ 
& + \int_0^t \int_\To \Big( |\uu|^2 +|F|^2 + |\nabla M|^2 \Big) \Big(|\nabla \phi|^2 + |\phi||\nabla^2 \phi|\Big)  +  \int_0^t \int_\To |\nabla \uu ||F|| F- W'(F) | \phi^2 \\ 
& +\int_0^t \int_\To |\uu|^2\phi^2 
+\int_0^t \int_\To \Big(|H_\ext|^2 +|\nabla H_\ext|^2+ |\psi'(M)|^2\Big)\phi^2\bigg\}
\end{align*}
for all $0\leq  t \leq \widetilde T$.
\end{lemma}
\begin{proof}
Multiplying the momentum equation of \eqref{main_equation} by $\uu\phi^2$ and integration over $\To$, we have that 
\begin{align*}
\int_\To \uu_t \cdot \uu  \phi^2  &+  (\uu \cdot \nabla  ) \uu \cdot \uu \phi^2 - \nu \Delta \uu \cdot \uu \phi^2  + \nabla \pre \cdot \uu \phi^2  \\ &=  \int_\To \Big(- \Div (\nabla M \odot \nabla M) \cdot \uu  + \Div (W'(F)F^\top ) + \mu_0 \nabla^\top H_\ext M \Big) \cdot \uu \phi^2.
\end{align*}
We first proceed analyzing any term on the left-hand side, since the ones on the right-hand side will eventually be  simplified under suitable combination with terms of the other equations. Thus
\begin{align*}
& \bullet \int_\To \uu_t \cdot \uu \phi^2 = \frac12 \frac{\dd}{\dd t} \int_\To |\uu|^2 \phi^2, \\
& \bullet 
\int_\To (\uu \cdot \nabla  ) \uu \cdot \uu \phi^2 =
\int_\To \uu_j \partial_j \uu_i \cdot  \uu_i \phi^2 = - \int_\To \frac{|\uu|^2}{2} \uu_j 2 \phi \partial_j \phi \leq 
C\int_\To |\uu |^3 |\phi| |\nabla \phi|, \\
& \bullet - \nu \int_\To \Delta \uu \cdot \uu \phi^2 =\nu  \int_\To \partial_j \uu_i \partial_j \uu_i \phi^2 + \partial_j \uu_i \cdot \uu_i 
2 \phi \partial_j \phi  = \nu \int_\To |\nabla \uu|^2 \phi^2 - \nu \int_\To |\uu|^2 ( \partial_j \phi \partial_j \phi + \phi \partial_j^2 \phi)  \\ 
 &\hspace{2.75cm} = \nu \int_\To
 |\nabla \uu|^2 \phi^2 - |\uu|^2 (|\nabla \phi|^2 + \phi \Delta \phi) \geq  \nu \int_\To
 |\nabla \uu|^2 \phi^2 - |\uu|^2 (|\nabla \phi|^2 + |\phi||\nabla^2 \phi|) ,\\
& \bullet \int_\To \nabla \pre \cdot \uu \phi^2 = - \int_\To 2\pre \uu \cdot \phi \nabla \phi
 \leq\int_\To 2|\pre| |\uu ||\phi| |\nabla \phi |.
\end{align*}
Furthermore, the contribution of the external magnetic field to the momentum equation is dealt with by
\begin{equation*}
    \mu_0 \int_\To \nabla^\perp H_\ext M\cdot \uu \phi^2 \leq C\Big( \int_\To |\nabla H_\ext |^2\phi^2 +  \int_\To|\uu |^2\phi^2\Big).
\end{equation*}
Next, we multiply the equation for $F$ in \eqref{main_equation} by $\chi F\phi^2$, where we recall that the positive constant $\chi>0$ is such that $|W'(A)-\chi A| \leq C_2$, for any $A\in \mathbb{R}^{2\time 2}$. Integrating over $\To$, we gather
$$ \chi\int_\To\Big( F_t : F \phi^2 + (\uu\cdot  \nabla)F : F \phi^2 - \nabla \uu F :F \phi^2 - \kappa \Delta F : F  \phi^2\Big) =0,$$
where
\begin{align*}
& \bullet  \int_\To
 F_t : F\phi^2 = \frac12 \frac{\dd }{\dd t} \int_\To |F|^2 \phi^2 ,
\\
& \bullet \int_\To (\uu\cdot  \nabla)F : F \phi^2 = - \int_\To |F|^2 \uu_j  \phi \partial_j \phi
\leq \int_{\To} |F|^2|\uu||\phi||\nabla\phi| ,\\
&  \bullet -\kappa \int_\To \Delta F : F \phi^2 =\kappa \int_\To \Big(\nabla F : \nabla F \phi^2  + \partial_j F_{ik} 
F_{ik} 2 \phi \partial_j \phi \Big)\\
& \, \, \, \, = \kappa \int_\To \Big(|\nabla F|^2 \phi^2- |F|^2  \Big(|\nabla \phi|^2 + \phi \Delta \phi\Big)\Big) 
\geq \kappa \int_\To \Big(|\nabla F|^2 \phi^2- |F|^2  \Big(|\nabla \phi|^2 + |\phi|| \nabla^2 \phi|\Big)\Big). 
\end{align*}
Additionally we have that
\begin{align*}
& -\int_\To \chi\nabla \uu F: F \phi^2 + \Div (W'(F)F^\top )\cdot \uu \phi^2 = - \int_\To \partial_j \uu_i \cdot F_{jk} \chi F_{ik} \phi^2 + 
\partial_j ( W'(F)_{ik} F_{jk} ) \uu_i \phi^2 \\
& = \int_\To \uu_i  W'(F)_{ik} F_{jk} 2 \phi \partial_j \phi - \int_\To \nabla \uu F: (\chi F - W'(F) ) \phi^2
\\
& \leq C\left\{  \int_\To |\uu ||F|^2 |\phi||\nabla\phi|+ \int_{\To} |\uu|^2 \phi^2 + \int_\To |F|^2 |\nabla \phi|^2 \right\} + \int_\To |\nabla\uu||F||\chi F-W'(F)|\phi^2.
\end{align*}
Moreover, multiplying the LLG equation in \eqref{main_equation} by $-(\Delta M+|\nabla M|^2 M )\phi^2$ and integrating over $\To$, we deduce that
\begin{align*}
-\int_\To &\Big(M_t + (\uu \cdot \nabla ) M \Big) \cdot \Big(\Delta M+|\nabla M|^2 M \Big) \phi^2 \\
 =&\int_\To  M \wedge H_\eff  \cdot \Big(\Delta M+|\nabla M|^2 M \Big)\phi^2+ \int_\To M \wedge (M \wedge H_\eff) \cdot \Big(\Delta M+|\nabla M|^2 M \Big)\phi^2.
\end{align*}
Developing the first term on the right-hand side, we deduce that
\begin{align*}
	\int_\To  M \wedge H_\eff  \cdot \Big(\Delta M+|\nabla M|^2 M \Big)\phi^2 
	&=\int_\To  M \wedge \Big(  \mu_0 H_\ext -\psi'(M)\Big)  \cdot \Big(\Delta M+|\nabla M|^2 M \Big)\phi^2 \\
	&\leq  C\int_\To \Big(| H_\ext|^2 +|\psi'(M)|^2\Big)\phi^2 + 
	\frac{1}{4}\int_\To |\Delta M+|\nabla M|^2 M |^2\phi^2,
\end{align*}
while the second term yields
\begin{align*}
	\int_\To& M \wedge (M \wedge H_\eff) \cdot \Big(\Delta M+|\nabla M|^2 M \Big)\phi^2 \\
	&= 
	\int_\To M \wedge \Big(M \wedge (\mu_0 H_\ext - \psi'(M))\Big) \cdot (\Delta M+|\nabla M|^2 M )\phi^2 - 
	\int_\To |\Delta M+|\nabla M|^2M|^2\phi^2 \\ & \leq C\int_\To (| H_\ext|^2 +|\psi'(M)|^2)\phi^2 -
	\frac{3}{4} |\Delta M+|\nabla M|^2M|^2\phi^2.
\end{align*}
We hence remark that the following estimates holds true
\begin{align*}
\int_\To& (\uu \cdot \nabla ) M \cdot \Big(\Delta M +|\nabla M|^2M \Big) \phi^2 -\Div ( \nabla M \odot \nabla M) \cdot \uu \phi^2  \\
 =&  \int_\To  \uu_i \partial_i M_k \cdot \partial_j^2 M_k \phi^2 - \partial_j  (\partial_i M_k  \partial_j M_k) \uu_i\phi^2   =   \int_\To \frac{1}{2} |\nabla M|^2 \uu \cdot 2 \phi \nabla \phi 
\leq C \int_\To |\nabla M|^2 |\uu||\phi||\nabla \phi|.
\end{align*}
Next, we have $M_t\cdot |\nabla M|^2 M = |\nabla M|^2\partial_t (|M|^2/2) = 0$ and
\begin{align*}
		&  \int_\To M_t \cdot (-\Delta M) \phi^2  = \frac12 \frac{\dd}{\dd t}\int_\To |\nabla M|^2 \phi^2 + \int_\To \partial_t M_j \partial_k 
		M_j 2 \phi \partial_k \phi.
\end{align*}
Thanks to the LLG equation,
\begin{align*}
	\int_\To \partial_t M_j \partial_k& 
		M_j 2 \phi \partial_k \phi 
		= 
	-\int_\To (\uu\cdot \nabla M_j)\partial_k 
		M_j 2 \phi \partial_k \phi - 
		\int_\To \Big(M\wedge (\Delta M + \mu_0 H_\ext - \psi'(M) \Big)_j\partial_k 
		M_j 2 \phi \partial_k \phi  \\
		& - \int_\To \Big(M\wedge M\wedge \Big(\Delta M + \mu_0 H_\ext - \psi'(M) \Big)\Big)_j\partial_k 
		M_j 2 \phi \partial_k \phi \\
		= & 
	-\int_\To (\uu\cdot \nabla M_j)\partial_k 
		M_j 2 \phi \partial_k \phi  
	-	\int_\To (M\wedge (\Delta M +|\nabla M|^2 M + \mu_0 H_\ext - \psi'(M) )_j\partial_k 
		M_j 2 \phi \partial_k \phi \\ 
		& - \int_\To \Big(M\wedge M\wedge \Big(\Delta M +|\nabla M|^2 M + \mu_0 H_\ext - \psi'(M) \Big)\Big)_j\partial_k 
		M_j 2 \phi \partial_k \phi
		\\
		\leq &
		\,C\int_\To | \uu ||\nabla M|^2|\phi ||\nabla \phi| +
		\frac{1}{4}\int_\To  | \Delta M +|\nabla M|^2 M|^2\phi^2 \\ 
		& + C\int_\To |\nabla M|^2 |\nabla \phi|^2 +
		C\int_\To (|H_\ext |^2 + |\psi'(M)|^2)\phi^2
\end{align*}
Finally,  
integration over $[0,t]$ yields the assertion.
\end{proof}

\section{Fourier analysis toolbox} \label{app:LP}

\noindent
Before dealing with the uniqueness of Theorem~\ref{main-thm1}, we recall here the main ideas of the Littlewood-Paley theory along with the definition of Besov spaces, which we will exploit in the forthcoming analysis (cf.\ Proposition~\ref{prop:estEE}). We refer the interested reader to \cite{B-C-D,Gubinelli-Perkowski,ScTr}, for further details. 


\medbreak
\noindent
First of all, let us introduce the so called ``Littlewood-Paley decomposition'', based on a homogeneous dyadic partition of unity with
respect to the Fourier variable. 
%
We fix a smooth radial function $\widetilde \chi$ supported on the ball $B_2(0)\subset \RR^d $, equal to $1$ in a neighborhood of $B_1(0)$
and such that $r\mapsto \widetilde \chi(r\,e)$ is nonincreasing over $\R_+$ for all unitary vectors $e\in\R^d$. Set
$\varphi\left(\xi\right)=\widetilde \chi\left(\xi\right)-\widetilde \chi\left(2\xi\right)$ and
$\vphi_q(\xi):=\vphi(2^{-q}\xi)$ for all $q\geq0$.
%
The dyadic blocks $(\Dd_q)_{q\in\ZZ}$ are defined by
$$
	\Dd_q u(x) \,:=\,\sum_{n\in \ZZ^d} u_n\varphi_q\left(|n|\right)e^{i n\cdot x}\qquad 
	\text{with}\quad
	{u_n} = \frac{1}{(2\pi)^d}\int_{\TT^d}u(x)e^{-i n\cdot x}\dd x.
$$ 
Observe that there are only contributions to the sum for $|n|\in [2^{q-1},\,2^q]\cap \mathbb{N}$ . 
We  also introduce the following low frequency, homogeneous cut-off operator:
\begin{equation} \label{eq:S_j}
\Sd_{q} u\,:=\,\sum_{k\leq q-1}\Dd_{k}u,
\end{equation}
for $q\in \ZZ$.  Following Remark 2.11 of \cite{B-C-D}, we infer that the operators $\Sd_q$ and $\Dd_q$ continuously map $L^p$ to itself, for all $q\in \ZZ$ and all $p\in[1,+\infty]$, with norms independent of $q$ and $p$.

\noindent
The following classical property holds true: Let $u\in L^1(\TT^d)$ with null average $\int_{\TT^\dd}u = 0$; then $u=\sum_{q\in \ZZ}\Dd_q u$.
Next we recall also mention the so-called \textit{Bernstein inequalities}, which explain the way derivatives act on spectrally localized functions.
  \begin{lemma}[\!\!\cite{B-C-D}, Lemma 2.1]  \label{l:bern}
Let  $0<r<R$.   A constant $C$ exists so that, for any nonnegative integer $k$, any couple $(p,q)$ 
in $[1,+\infty]^2$, with  $p\leq q$,  and any function $u\in L^p$,  we  have, for all $\lambda>0$,
\begin{align}\label{Bernstein-ineq}
{\rm supp}\, \widehat u \subset   B(0,\lambda R)\cap \ZZ^\dd \quad
&\Longrightarrow\quad
\|\nabla^k u\|_{L^q}\, \leq\,
 C^{k+1}\,\lambda^{k+d\left(\frac{1}{p}-\frac{1}{q}\right)}\,\|u\|_{L^p}\;;\\
{\rm supp}\, \widehat u \subset \{\xi\in\R^d\,|\, r\lambda\leq|\xi|\leq R\lambda\}\cap \ZZ^\dd 
\quad & \Longrightarrow\quad C^{-k-1}\,\lambda^k\|u\|_{L^p}\,
\leq\,
\|\nabla^k u\|_{L^p}\,
\leq\,
C^{k+1} \, \lambda^k\|u\|_{L^p}\,,
\end{align}
where $\hat u$ denotes the Fourier transform of $\uu$.
\end{lemma}   

\noindent
{Later we will choose $\lambda = \tfrac1{2^q}$.} By use of the Littlewood-Paley decomposition, we can define the class of homogeneous Besov spaces.
\begin{def} \label{d:B}
Let $s\in\R$ and $1\leq p,r\leq+\infty$. The homogeneous Besov space
$\BB^{s}_{p,r}=\BB^{s}_{p,r}(\TT^2)$ is defined as the space of any homogeneous distribution $u$ in $\TT^2$ for which
$$
\|u\|_{\BB^{s}_{p,r}}\,:=\,
\left\|\left(2^{qs}\,\|\Dd_q u\|_{L^p}\right)_{q\in\ZZ}\right\|_{\ell^r}\,<\,+\infty\,.
$$
\end{def}

\noindent
Besov spaces extend the well-known class of Sobolev and H\"older spaces.  Moreover, for all $s\in\R$ with $s<d/p$ we have the isomorphism of Banach spaces $\BB^s_{2,2}(\TT^2)\cong \Hh^s(\TT^2)$, with
\begin{equation} \label{eq:LP-Sob}
\|f\|_{\Hh^s}\,\sim\,\left(\sum_{q\in \ZZ}2^{2 q s}\,\|\Dd_q f\|^2_{L^2(\TT^2)}\right)^{1/2}\,.
\end{equation}
In fact, the previous isomorphism is an isomorphism of Hilbert spaces. Indeed, if we define the $\BB^s_{2,2}$ scalar product by the formula,
\begin{equation}\label{inner_product}
\langle f,\, g \rangle_{\BB_{2,2}^{\s}} := \sum_{q\in \ZZ} 2^{2\,q\,\s}\langle \Dd_q f\,,\, \Dd_q g\rangle_{L^2(\TT^2)}\,,
\end{equation}
by Proposition 2.10 of \cite{B-C-D} we get a scalar product which is equivalent to the classical one over $\Hh^s(\TT^2)$.
We remark that the above relation is well defined also when $f$ and $g$ are $L^2$ functions with no null average, i.e.\ they are not homogeneous. In that framework, the inner product generates only a semi norm.

\noindent
An immediate consequence of the first Bernstein inequality in \eqref{Bernstein-ineq} is the following embedding result.
\begin{prop}[\!\! \cite{B-C-D}, Proposition~2.20]\label{p:embed}
Let $1\leq p_1 \leq p_2 \leq \infty$ and $1\leq r_1\leq r_2 \leq \infty$. Then, for any real number $s$ the space $\BB^{s}_{p_1,r_1}$ is continuously embedded in the space $\BB^{s - d(1/p_1 -1/p_2)}_{p_2,r_2}$. 
\end{prop}

\noindent In the case of a negative index of regularity, homogeneous Besov spaces may be characterized in terms of operators $\Sd_q$, as follows.
\begin{prop}[\!\!\cite{B-C-D},Proposition~2.33]\label{prop:s_negative}
	Let $s<0$ and $1\leq p,r\leq \infty$. Then $u$ belongs to $\BB_{p,r}^s$ if and only if 
	\begin{equation*}
		\left\|\left(2^{qs}\,\|\Sd_q u\|_{L^p}\right)_{q\in\ZZ}\right\|_{\ell^r}\,<\,+\infty.
	\end{equation*}
	The above norm is equivalent to the standard one given by the dyadic blocks $(\Dd_q)_{q\in \ZZ}$.
\end{prop}

\noindent
Let us now introduce the dyadic decomposition (after J.-M.\ Bony, see \cite{Bony}). Constructing this decomposition relies on the observation that, 
formally, any product  of two suitable functions $u$ and $v$ with null averages may be decomposed into 
\begin{equation}\label{eq:bony}
u\cdot v\;=\; \sum_{q\in\ZZ} \Sd_{q-1} u \cdot  \Dd_q v + \sum_{q\in \ZZ} \Sd_{q+2}v\cdot \Dd_q u. 
\end{equation}
Taking into account that the Fourier coefficients of $\Sd_{q-1} u \cdot\Dd_q v$ are included into the dyadic annulus $2^q \mathcal{C}$ and that the Fourier coefficients of $\Sd_{q+2}v \cdot \Dd_q u$ are included in the ball $B(0, 2^q)$, we  have the following quasi-orthogonal properties:
\begin{equation*}
\begin{aligned}
	\Dd_j ( \Sd_{q-1} u \cdot \Dd_q v) &= 0\qquad \text{for}\quad |q-j|\geq 5,\\
	\Dd_j(\Sd_{q+2}v \cdot \Dd_q u) &= 0\qquad \text{for}\quad j> q+5.
\end{aligned}
\end{equation*}
We will hence often use the Bony decomposition
\begin{equation}\label{BONY}
\begin{aligned}
		\Dd_q (u\cdot v) &= \sum_{|q-j|\leq 5} \Dd_q(\Sd_{j-1} u \cdot \Dd_j v) + \sum_{j\geq q- 5} \Dd_q(\Sd_{j+2}v \cdot \Dd_j u) \\
		&=\sum_{|q-j|\leq 5} [\Dd_q,\Sd_{j-1} u]\cdot \Dd_j v +\sum_{|q-j|\leq 5}(\Sd_{j-1} u-\Sd_{q-1} u)\Dd_q \Dd_j  v + \Sd_{q-1}u \cdot \Dd_q v + \sum_{j\geq q- 5} \Dd_q(\Sd_{j+2}v \cdot \Dd_j u).
\end{aligned}
\end{equation}
We remark that usually it is the third term that is difficult to control. To control the first term of the above decomposition, the following classical commutator estimate can be applied.
\begin{lemma}[\!\!\cite{MR2038119}, Lemma~2.4]\label{l:commut}
 {Let $p\geq 1$, $r\geq 1$, $h\geq 1$ be such that} $1/p = 1/r + 1/h$. Then there exists a constant $C$ such that for any vector fields $u$ and $v$ with null average, 
	the following inequality holds true:
	\begin{equation*}
		\| [\Dd_q, u ]\cdot v \|_{L^p} \leq C2^{-q} \| \nabla u \|_{L^r} \| v \|_{L^h}.
	\end{equation*}
\end{lemma}
\noindent We conclude this section with the statement of a product law between homogeneous Sobolev space.
\begin{lemma}[\!\!\cite{MR3576270}, Theorem~A.1]\label{lemma_product_homogeneous_sobolev_spaces}
	Let $s$ and $t$ be two real numbers such that $ |s|$ and $|t|$ belong to $[0,d/2)$. Let us assume that $s+t$ is positive, then there exists a positive 
	constant $C$ such that
	\begin{equation*}
		\|a \,b\|_{\Hh^{s+t-d/2}(\TT^\dd)}\leq C\|a\|_{\Hh^{s}(\TT^d)}\|b\|_{\Hh^t(\TT^\dd)}
	\end{equation*}
	for every  $a\in \Hh^s(\TT^d)$ and for every $b\in \Hh^t(\TT^d)$.
\end{lemma}
\begin{remark}
    Lemma~\ref{lemma_product_homogeneous_sobolev_spaces} also holds true  when the product $a b$ is not homogeneous, namely it is not average free. Nevertheless, Lemma~\ref{lemma_product_homogeneous_sobolev_spaces} states that the homogeneous Sobolev seminorm $\Hh^{s+t-\dd/2}$ can still be bounded by the correponding Sobolev norms of the homogeneous functions $a$ and $b$.
\end{remark}

\section{Uniqueness of weak solutions}\label{sec:uniqueness}
\noindent
In this section we provide the proof of the uniqueness result for the weak solutions of system \eqref{main_equation}, that we state in the following theorem:
\begin{theorem}\label{thm:uniqueness}
    Let $(\uu_0,\,F_0,,M_0)\in L^2(\TT^2)\times L^2(\TT^2)\times H^1(\TT^2)$, then the solution determined in Theorem~\ref{struweSolutions} is unique within of the 
    functional space introduced in Definition~\ref{def:weak-sol}. 
\end{theorem}
\noindent 
The main idea is to evaluate the difference between two Struwe-like solutions with the same initial data in a function space which is less regular than $L^2(\TT^2)$, such as $\Hh^{-\frac{1}{2}}(\TT^2)$. For convenience, we introduce the following notation:
\begin{equation*}
	(\delta \uu,\, \delta F,\,\delta M) = (\uu_1 - \uu_2,\,F_1 -F_2,\, M_1 - M_2),
\end{equation*}
where $(\uu_1,\, F_1,\,M_1)$ and $(\uu_2,\,F_2,\,M_2)$ stand for the first and second solution, respectively. 
We aim to estimate the two differences within the function space given by
\begin{equation}\label{space-where-deltauanddeltad-are-defined}
	\begin{aligned}
		(\delta \uu^h,\,\delta F)\,\in\,L^\infty(0,T;\,\Hh^{-\frac{1}{2}}(\TT^2))\cap L^2(0,T;\Hh^{\frac{1}{2}}(\TT^2)),\quad
		\delta M\,\in\, L^\infty(0,T;\,H^{\frac{1}{2}}(\TT^2))\cap  L^2(0,T;\,H^{\frac{3}{2}}(\TT^2)),
	\end{aligned}
\end{equation} 
where $\delta \uu^h$ stands for the homogeneous component of the velocity field $\delta \uu$, which will be explicitly defined below in this section. Hence, under the assumption that the initial data of both solutions coincide, we claim that
\begin{equation*}
	(\delta \uu(t),\, \delta F(t),\,\delta M(t))  = (0, 0, 0)\qquad \text{for any time }t>0.
\end{equation*}
To this end we introduce the $\Hh^{-1/2}$-energy of the system and its related dissipation:
\begin{equation}    \label{energyDifference}
	\begin{aligned}
		\delta \EE (t) &:= \frac{1}{2}|\overline{\delta \uu(t)}|^2+ \frac{1}{2}\| \delta \uu^h(t) \|_{\Hh^{-\frac{1}{2}}(\TT^2)}^2 +\frac{1}{2} \| \delta F(t) \|_{\Hh^{-\frac{1}{2}}(\TT^2)}^2+ \frac{1}{2}\| \delta M(t) \|_{H^{\frac{1}{2}}(\TT^2)}^2,\\
		\delta \DD(t) &:=\nu \| \nabla \delta \uu(t) \|_{\Hh^{-\frac{1}{2}}(\TT^2)}^2+\kappa \| \nabla \delta F(t) \|_{\Hh^{-\frac{1}{2}}(\TT^2)}^2  + \| \Delta \delta M (t) \|_{\Hh^{-\frac{1}{2}}(\TT^2)}^2,
	\end{aligned}
\end{equation}
where $\overline{\delta \uu(t)}$ stands for the average of the velocity field in $\mathbb{T}^2$, as described below.

\noindent 
The functional setting of $\delta F$ is that of homogeneous Sobolev space $\Hh^{-\frac{1}{2}}(\TT^2)$, which is well motivated by the formal identity
\begin{equation}
	\int_{\TT^2} \delta F(t) \,=\, 0,
\end{equation} 
for any time $t\in\RR_+$. This relation can be obtained integrating in $[0, t)\times \TT^2$ the equations for $\delta F$ of \eqref{main_equation}, recalling that $\Div \,\delta F(t)^\top = 0$ for any time $t\in \RR_+$. We remark however that this null-average property is not satisfied by the velocity $\delta \uu$ and by the magnetization $\delta M$, hence we are forced to estimate these terms in a non-homogeneous environment. For this reason, we will repeatedly split these terms into their homogeneous and non-homogeneous components:
\begin{equation*}
\begin{alignedat}{8}
    \delta M(t) &= \overline{\delta M(t)} + \delta M^h(t),\qquad &&\text{with}\quad
    \overline{\delta M(t)} = \int_{\TT^2}\delta M(t,x)\dd x\quad&&&&\text{and}\quad 
    \delta M^h(t) = \sum_{n\in \ZZ^2 \setminus \{0\}} \delta M_n(t) e^{i n\cdot x},\\
    \delta \uu(t) &= \overline{\delta \uu(t)} + \delta \uu^h(t),\qquad &&\text{with}\quad
    \overline{\delta \uu(t)} = \int_{\TT^2}\delta \uu(t,x)\dd x\quad&&&&\text{and}\quad 
    \delta \uu^h(t) = \sum_{n\in \ZZ^2 \setminus \{0\}} \delta \uu_n(t) e^{i n\cdot x}.
\end{alignedat}
\end{equation*}

\noindent Furthermore, $\delta \uu$, $\delta F$ and $\delta M$ keep a higher regularity than the one presented in \eqref{space-where-deltauanddeltad-are-defined}. Indeed, since $\delta \uu$, $\delta F$  and $\delta M$ belong to the energy space given by \eqref{def:weak-sol-energy-space}, this automatically implies that the following time regularity holds true:
\begin{equation*}
	\overline{\delta\uu}\in  \mathcal{C}([T_i, \widetilde T]),\; 	\delta \uu^h \in \mathcal{C}([T_i, \widetilde T],\, \dot H^{-\frac{1}{2}}(\TT^2)),\;\delta F \in \mathcal{C}([T_i,\widetilde  T],\, \dot H^{-\frac{1}{2}}(\TT^2))\;\text{and}\;\delta M \in \mathcal{C}([T_i, \widetilde T],\,H^\frac{1}{2}(\TT^2)),
\end{equation*}
for any $\widetilde T\in [T_i, T_{i+1})$, where the family $(T_i)_{i\in\NN}$ stands for the discrete set given by any time in which a singularity for $(\uu_1,\,F_1,\, M_1)$ or for $(\uu_2,\,F_2,\,M_2)$ occurs (with the exception of $T_0$ which is defined by $T_0 = 0$). Furthermore, we have the following explicit formulations of the distributional time derivatives of the Sobolev norms:
\begin{equation}\label{distr-time-der-of-the-norm}
\begin{alignedat}{4}
		\frac{\dd}{\dd t}& \|\,\delta \uu^h(t)\|_{\Hh^{-\frac{1}{2}}}^2 \,&&=\, 2\big\langle \partial_t \delta \uu^h,\,\delta \uu^h\big\rangle_{\Hh^{-1}\times L^2 }
		= 2\sum_{q\in\ZZ} 2^{-q} \langle \partial_t \Dd_q \delta \uu,\, \Dd_q \delta \uu\rangle_{L^2(\TT^2)}\in L^1_{\rm loc}(T_i,T),
		\\
		\frac{\dd}{\dd t}& \|\,\delta M(t)\|_{H^{\frac{1}{2}}}^2
		\,&&=\,
		\frac{\dd}{\dd t} \|\,\delta M(t)\|_{L^2(\TT^2)}^2
		\,+\,
		\frac{\dd}{\dd t} \|\,\delta M(t)\|_{\Hh^{\frac{1}{2}}}^2
		\,\\ 
		&\, 
		&&=
		2\big\langle \partial_t \delta M,\,\delta M\big\rangle_{L^2\times L^2}
		\,+\,
		2\big\langle \partial_t \delta M,\,\delta M\big\rangle_{L^2\times\Hh^{1} }
		\in L^1_{\rm loc}(T_i,T),
\end{alignedat}
\end{equation}
where the above identities are satisfied in a distributional sense on $\mathcal{D}'((T_i,T_{i+1}))$. 

\noindent
The main step of the uniqueness proof is a standard Osgood type inequality which involves the above function spaces. We formulate this result in the following proposition. 
\begin{prop}\label{prop:estEE}
There exists a time-dependent function $f = f(t)$ which belongs to $L^1(T_i,\,\widetilde T)$, for any $i \in \NN$ and for any $T_i<\widetilde T<T_{i+1}$, such that
\begin{equation*}
	\delta \EE(t) + \int_{T_i}^t \delta \DD(s)\dd s \leq \delta \EE(T_i) + \int_{T_i}^t f(s) \mu(\delta \EE(s) )\dd s,
\end{equation*}
for any time $t\in [T_{i+1}, \widetilde T)$, where $\mu(\cdot)$ is the Osgood modulus of continuity given by
\begin{equation}\label{modulus_of_continuity}
    \mu(r) = r \Big(1+\ln \Big(1+ \frac{1}{r}\Big)\Big).
\end{equation}
\end{prop}

\noindent The uniqueness of our weak solutions are a straightforward application of the above proposition. Indeed, under the assumption that the mentioned statement holds true, the uniqueness follows by a standard continuation argument, that can be summarized by the following induction process:
\begin{enumerate}[1.]
	\item Prove that the two solutions are unique in the first interval $[0, T_1)$, where $T_1$ is the first value of time in which the singularity of one of the solutions occurs. 
	\item Show that if the two solutions coincide on an interval $[T_i,T_{i+1})$ then the same solutions coincide also at the time $T_{i+1}$. 
	\item Show that if two solutions coincide on the time interval $[T_i,T_{i+1}]$ then the same solutions coincide also on the interval $[T_{i+1}, T_{i+2})$.
\end{enumerate}
The first result turns out to be a direct application of Proposition~\ref{prop:estEE}. Indeed, if $f(t)$ is a locally integrable function in $[0, \widetilde T]$ for any $\widetilde T<T_1$, the Osgood inequality together with the initial condition $\delta \EE(0)=0$ imply that
\begin{equation*}
			\delta \EE(t) = 0 \quad\text{for any}\quad t\in [0, T_1) \quad \Rightarrow (\uu_1(t),\,F_1(t),M_1(t)) = (\uu_2(t),\,F_2(t),\,M_2(t)).
\end{equation*}
The second point of the above induction process is a straightforward consequence of the following weak continuity properties (which hold true by the energy estimate \eqref{energyInequality} and standard duality estimates for the time-derivatives $\partial_t (u,F,M)$): 
\begin{equation*}
	(\uu_1,\,\uu_2) \in \mathcal{C}_w([0,T],\, L^2(\TT^2)), \qquad (F_1,\,F_2) \in \mathcal{C}_w([0,T],\,L^2(\RR^2)), \qquad (M_1,\,M_2)\in \mathcal{C}_w([0,T], H^1(\TT^2)).
\end{equation*}
We hence deduce that $\uu_1(T_{i+1}) = \uu_2(T_{i+1})$, $F_1(T_{i+1}) = F_2(T_{i+1})$ and $M_1(T_{i+1}) = M_2(T_{i+1})$, whenever $(\delta \uu,\delta F,\delta M)$ is null in $[T_i,T_{i+1})$. 

\noindent
Finally, the last induction step  is achieved by remarking that $\delta \EE(T_{i+1}) = 0$. Indeed, we reapply the Osgood lemma together with Proposition~\ref{prop:estEE} on the interval $[T_{i+1}, \widetilde T]$ for $ \widetilde T<T_{i+2}$, to deduce that our solutions coincide also in $[T_{i+1}, T_{i+2})$. 

\noindent In summary, the uniqueness asserted in Theorem~\ref{thm:uniqueness} is satisfied thanks to Proposition~\ref{prop:estEE}, which we prove next.
\begin{proof}[Proof of Proposition~\ref{prop:estEE}]
We begin with considering the difference between the momentum equations of the two solutions, driving the evolution of $\delta \uu$:
\begin{align}\label{deltau-eq}
\begin{split}
	\partial_t \delta \uu + \uu_1\cdot \nabla \delta \uu + \delta \uu\cdot  \nabla \uu_2 - \nu \Delta \delta \uu + \nabla \delta \pre 
	= &\,
	\Div( \nabla M_1 \odot \nabla \delta M) + \Div(\nabla \delta M\odot \nabla M_2) \\
	& +\Div( \delta  W'(F) F_1^\top+ W'(F_2) \delta F^\top )+ \mu_0 \nabla H_{\rm ext} \delta M,
	\end{split}
\end{align}
where $\delta W'(F)= W'(F_1)-W'(F_2)$. Integrating over $(T_i,t)\times \mathbb{T}^2$, we have that
\begin{equation*}
	\overline{\delta \uu (t)} =\overline{\delta \uu (T_i)}+ \mu_0 \int_{T_i}^t \int_{\mathbb{T}^2} \nabla H_{\rm ext} \delta M
	\quad \Rightarrow\quad 
	|\overline{\delta \uu (t)}| 
	\leq 
	|\overline{\delta \uu (T_i)}| +
	|\mu_0| \int_{T_i}^t \| \nabla H_{\rm ext} \|_{L^2} \| \delta M \|_{L^2},
\end{equation*}
which implies 
\begin{equation}\label{average-ineq}
	\frac{1}{2}
	|\overline{\delta \uu(t)}|^2 
	\leq 
	|\overline{\delta \uu(T_i)}|^2 +
	\mu_0^2 \| \nabla H_{\rm ext} \|_{L^2(T_i,t; L^2)}^2\int_{T_i}^t \| \delta M(s) \|_{L^2}^2\dd s.
\end{equation}
The function $\delta \uu^h(t,x)$ belongs to the function space defined in \eqref{def:weak-sol-energy-space}, i.e., 
\begin{equation*}
	\delta \uu^h \in L^\infty(0,T;\,L^2(\TT^2))\cap L^2(0,T;\Hh^1(\TT^2)),\qquad   \partial_t \delta \uu^h \in L^2(T_i,\widetilde T; \Hh^{-1}(\TT^2)),
\end{equation*}
for any $\widetilde T\in [T_i, T_{i+1})$. Recalling that $\delta \uu^h$ is homogeneous, namely $\int_{\TT^2}\delta \uu^h=0$, we apply the operator $\sqrt{-\Delta}^{-1}$ which maps $\Hh^{-1}(\TT^2)$ into $L^2(\TT^2)$, by means of
\begin{equation*}
	\sqrt{-\Delta}^{-1}\delta \uu^h(t,x)  = \sum_{k\in \ZZ^2} \frac{\delta \uu_k^h(t)}{|k|}e^{ ik\cdot x}\in  L^\infty(0,T;\,\Hh^1(\TT^2))\cap L^2(0,\,T;\Hh^2(\TT^2)).
\end{equation*}
Additionally, $\partial_t \sqrt{-\Delta}^{-1}\delta \uu^h = 
 \sqrt{-\Delta}^{-1}\partial_t\delta \uu^h\in L^2(T_i, \widetilde T; L^2(\TT^2))$.
We hence consider a sequence of smooth functions $(\varphi^n)_{n\in\NN} \in \mathfrak{D}([T_i, t)\times\TT^2)$, with $t\in [T_i, T_{i+1})$ such that the following properties hold true
\begin{align}\label{test-fcts-deltau}
	\int_{\TT^2} \varphi^n (s) = 0\quad &\text{for any } s\in [T_i,T_{i+1}),\\
	\varphi^n \rightarrow \sqrt{-\Delta}^{-1}\delta \uu^h(t,x) \quad&\text{in } L^2(T_i,T_{i+1};H^1(\TT^2)),\\
	\partial_t\varphi^n \rightarrow \partial_t \sqrt{-\Delta}^{-1}\delta \uu^h(t,x) \quad&\text{in }L^2(T_i,\, t;L^2(\TT^2)).
\end{align}
Applying identity \eqref{deltau-eq} together with Definition~\ref{def:weak-sol} to the test function $\varphi^n$ and passing to the limit in $n$, we eventually gather that
\begin{align*}
	\int_{T_i}^t \langle \partial_t \delta \uu^h ,\, \delta \uu^h\rangle_{\Hh^{-1}\times L^2}& + \int_{T_i}^t \nu \|\,\nabla \delta \uu\,\|_{\Hh^{-\frac{1}{2}}}^2
	= 
	\int_{T_i}^t
	\Big\{-\langle	\,	\uu_1\cdot\nabla \delta \uu		,	\delta \uu^h	\,	\rangle_{\Hh^{-\frac{1}{2}}}-
		\langle	\,	\delta \uu\cdot\nabla  \uu_2 		,	\delta \uu^h	\,	\rangle_{\Hh^{-\frac{1}{2}}} \\ &-
		\langle	\,	 \nabla M_1 \odot \nabla \delta M		,	\nabla \delta \uu	\,	\rangle_{\Hh^{-\frac{1}{2}}}  -
		\langle	\,	 \nabla \delta M\odot \nabla M_2		,	\nabla \delta \uu	\,	\rangle_{\Hh^{-\frac{1}{2}}} \\ &-
		\langle	\,	  \delta  W'(F) F_1^\top+ W'(F_2) \delta F^\top	,	\nabla \delta \uu	\,	\rangle_{\Hh^{-\frac{1}{2}}}+
		\mu_0\langle	\,	\nabla H_{\rm ext} \delta M	,	\delta \uu^h	\,	\rangle_{\Hh^{-\frac{1}{2}}}
	\Big\},
\end{align*}
where we remark that $\nabla \delta \uu = \nabla \delta \uu^h$. We emphasize that the above identity is also consequence of the following cancellations:
\begin{align*}
	\int_{T_i}^t \int_{\mathbb{T}^2} \partial_t \overline{\delta \uu}\cdot  \varphi^n   = 0\qquad 
	\text{for any } n\in \mathbb{N},
\end{align*}
since $ \partial_t \overline{\delta \uu}$ and $\delta \uu^h$ are orthogonal. Indeed, one can easily remark that $\overline{\delta\uu}\in W^{1,2}(0,T)\hookrightarrow W^{1,2}(0,T; L^2(\mathbb{T}^2))$ from Definition~\ref{def:weak-sol}. Additionally, we recall that $\varphi^n(t,\cdot)$ has null average for any time $t\in [T_i,T)$.

\noindent
We can thus apply the identity \eqref{distr-time-der-of-the-norm} stating that $\langle \delta \uu^h ,\, \partial_t \delta \uu^h\rangle_{\Hh^{-1}\times L^2} =\frac12 \frac{ \dd}{\dd t }\| \delta \uu^h \|_{\Hh^{-1/2}}^2$ in distributional sense. We hence deduce that
\begin{align}\label{uniq-momeq}
\begin{split}
	\frac 12\| \delta \uu^h(t) \|_{\Hh^{-\frac{1}{2}}}^2 + \int_{T_i}^t \nu \|\,\nabla \delta \uu\,\|_{\Hh^{-\frac{1}{2}}}^2
	= &
	\frac 12\| \delta \uu^h(T_i) \|_{\Hh^{-\frac{1}{2}}}^2 +
	\int_{T_i}^t
	\Big\{-\langle	\uu_1\cdot\nabla \delta \uu		,	\delta \uu^h		\rangle_{\Hh^{-\frac{1}{2}}}-
		\langle	\delta \uu\cdot\nabla  \uu_2 		,	\delta \uu^h		\rangle_{\Hh^{-\frac{1}{2}}} \\ &-
		\langle		 \nabla M_1 \odot \nabla \delta M		,	\nabla \delta \uu		\rangle_{\Hh^{-\frac{1}{2}}}-
		\langle		 \nabla \delta M\odot \nabla M_2		,	\nabla \delta \uu	\rangle_{\Hh^{-\frac{1}{2}}}\\ &-
		\langle		  \delta W'( F )F_1^\top+ W'(F_2) \delta F^\top	,	\nabla \delta \uu		\rangle_{\Hh^{-\frac{1}{2}}}+
		\mu_0\langle	\,	\nabla H_{\rm ext} \delta M	,	\delta \uu^h	\rangle_{\Hh^{-\frac{1}{2}}}
	\Big\}.
	\end{split}
\end{align}

\noindent
	We then proceed to estimate each term on the right-hand side. To this end, we will repeatedly make use of the following product law between Sobolev spaces (cf.\ Lemma~\ref{lemma_product_homogeneous_sobolev_spaces}): 
	\begin{alignat*}{8}
		\| u\cdot v \|_{\Hh^{-\frac{1}{2}}(\TT^2)} &\leq C \| u \|_{\Hh^{\frac{1}{2}}(\TT^2)}\| v \|_{L^2(\TT^2)},\qquad&&\text{for any}\quad u\in \Hh^{-\frac{1}{2}}(\TT^2)\quad &&&&\text{and}\quad v\in L^2(\TT^2)\quad\text{with}\quad \int_{\TT^2}v = 0, \\
		\| u\cdot v \|_{\Hh^{-\frac{1}{2}}(\TT^2)} &\leq C \| u \|_{\Hh^{\frac{3}{4}}(\TT^2)}\| v \|_{\Hh^{-\frac{1}{4}}(\TT^2)},\qquad &&\text{for any}\quad u\in \Hh^{\frac{3}{4}}(\TT^2)\quad &&&&\text{and}\quad v\in \Hh^{-\frac{1}{4}}(\TT^2)
	\end{alignat*}
	where the constant $C$ does not depend on $v$ and $u$. Furthermore, in order to implement the estimates of the main convective terms, we need to introduce the following inequality  at the level of homogeneous Sobolev space $\Hh^{-\frac{1}{2}}(\TT^2)$:
	\begin{lemma}\label{comm-est} For any divergence-free vector field $\rm v$ in $\Hh^{1}(\TT^2)$ and any vector field $B$ in $\Hh^{\frac{1}{2}}(\TT^2)$
		\begin{equation*}
			\langle\,  {\rm v}\cdot \nabla B,\, B\,\rangle_{\Hh^{-\frac{1}{2}}(\TT^2)}\,\leq\, C \|\,\nabla {\rm v}\,\|_{L^2(\TT^2)}\|\,\nabla B\,\|_{\Hh^{-\frac{1}{2}}(\TT^2)}\|\,B\,\|_{\Hh^{-\frac{1}{2}}(\TT^2)},
		\end{equation*}
		for a suitable positive constant $C$.
	\end{lemma}
	\begin{proof}
	    This follows from Lemma~\ref{lemma6.2} in Section~\ref{sec5}.
	\end{proof}
	
	\noindent
	Coming back to our problem, we remark that $\Div\,\delta \uu\,=\,0$. 
	We first split the following term of \eqref{uniq-momeq} into its homogeneous and non-homogeneous components:
	\begin{equation*}
	    \langle	\,	\uu_1\cdot\nabla \delta \uu^h		,	\delta \uu^h	\,	\rangle_{\Hh^{-\frac{1}{2}}}
	    =
	    \langle	\,	\overline{\uu}_1\cdot\nabla \delta \uu^h		,	\delta \uu^h	\,	\rangle_{\Hh^{-\frac{1}{2}}}
	    +\langle	\,	\uu_1^h\cdot\nabla \delta \uu^h		,	\delta \uu^h	\,	\rangle_{\Hh^{-\frac{1}{2}}}.
	\end{equation*}
	The first is dealt with by
	\begin{equation*}
	    \big|
	    \langle	\,	\overline{\uu}_1\cdot\nabla \delta \uu^h		,	\delta \uu^h	\,	\rangle_{\Hh^{-\frac{1}{2}}}
	    \big|
	    \leq 
	    |\overline{\uu}_1|\| \nabla \delta \uu^h \|_{\Hh^{-\frac{1}{2}}}
	    \| \delta \uu^h \|_{\Hh^{-\frac{1}{2}}}
	\lesssim 
	   \| \uu_1 \|_{L^2}^2\| \delta \uu^h \|_{\Hh^{-\frac{1}{2}}}^2
	   +
	   \ee
	   \|\nabla \delta \uu \|_{\Hh^{-\frac{1}{2}}}^2.
	\end{equation*}
	Thus, in accordance with Lemma~\ref{comm-est} together with the free divergence condition on $\uu_1$, we gather that
	\begin{align*}	
	    \big|
		\langle	\,	\uu_1^h\cdot\nabla \delta \uu^h		,	\delta \uu^h	\,	\rangle_{\Hh^{-\frac{1}{2}}}
		\big|
		\,
		\lesssim
		\,
		\|\,\nabla			\uu_1^h 	\,\|_{L^2	}
		\|\,			\delta 	\uu^h		\,\|_{\Hh^{-\frac{1}{2}}	}
		\|\,	\nabla 	\delta 	\uu^h		\,\|_{\Hh^{-\frac{1}{2}}	}
		\,
		\lesssim	
		\,
		\|\,\nabla			\uu_1 	\,\|_{L^2	}^2
		\|\,			\delta 	\uu^h		\,\|_{\Hh^{-\frac{1}{2}}	}^2\,+\,
		\ee
		\|\,	\nabla	\delta 	\uu^h		\,\|_{\Hh^{-\frac{1}{2}}	}^2,
	\end{align*}
	for a suitable small constant $\ee$ to be fixed later.  Next, the second term on the right-hand side of \eqref{uniq-momeq} is estimated as follows: we first split it into
	\begin{equation*}
		\langle	\,	\delta \uu\cdot\nabla  \uu_2 		,\,	\delta \uu^h	\,	\rangle_{\Hh^{-\frac{1}{2}}}
		= 
		\langle	\,	\overline{\delta \uu}\cdot\nabla  \uu_2 		,\,	
				\delta \uu^h	\,	\rangle_{\Hh^{-\frac{1}{2}}}+
		\langle	\,	\delta \uu^h\cdot\nabla  \uu_2 		,\,	
				\delta \uu^h	\,	\rangle_{\Hh^{-\frac{1}{2}}},
	\end{equation*}
	hence we have first that
	\begin{align*}
		|\langle	\,	\overline{\delta \uu}\cdot\nabla  \uu_2 		,\,	
								&\delta \uu^h	\,	\rangle_{\Hh^{-\frac{1}{2}}}|
		\leq 
		|\overline{\delta \uu}| 
		\| \nabla  \uu_2 \|_{\Hh^{-\frac{1}{2}}}
		\| \delta \uu^h \|_{\Hh^{-\frac{1}{2}}}
		\lesssim
		\| \uu_2 \|_{L^2}^\frac{1}{2}\| \nabla \uu_2 \|_{L^2}^\frac{1}{2}
		\Big(
		 |\overline{\delta \uu}|^2+ \| \delta \uu^h \|_{\Hh^{-\frac{1}{2}}}^2
		\Big),
	\end{align*}
	while
	\begin{align*}	
		|\langle	\,	\delta \uu^h\cdot\nabla  \uu_2 		,\,	\delta \uu^h	\,	
															\rangle_{\Hh^{-\frac{1}{2}}}|
		=& \,
			|\langle	\,	\delta \uu^h\cdot\nabla  \uu_2^h 		,\,	\delta \uu^h	\,	
															\rangle_{\Hh^{-\frac{1}{2}}}|
		=
		|\langle	\,	\delta \uu^h\otimes  \uu_2^h 		,\,	\nabla\delta \uu	\,	
															\rangle_{\Hh^{-\frac{1}{2}}}|
		\,\lesssim\,
		\|\,			\delta 	\uu^h	\otimes \uu_2^h 	\,\|_{\Hh^{-\frac{1}{2}}	}
		\|\,			\nabla \delta 	\uu						\,\|_{\Hh^{-\frac{1}{2}}	}\\
		\lesssim &
		\|\,			\delta 	\uu^h		\,\|_{\Hh^{-\frac{1}{4}}}
		\|\, 		\uu_2^h 	\,\|_{\Hh^\frac{3}{4}}
		\|\,		\nabla	\delta 	\uu		\,\|_{\Hh^{-\frac{1}{2}}	}
		\,\lesssim\, 
		\|\,			\delta 	\uu^h		\,\|_{\Hh^{-\frac{1}{2}}}^\frac{3}{4}	
		\|\, 		\uu_2^h 	\,\|_{L^2}^\frac{1}{4}
		\|\, \nabla	\uu_2^h 	\,\|_{L^2}^\frac{3}{4}
		\|\,		\nabla	\delta 	\uu		\,\|_{\Hh^{-\frac{1}{2}}	}^\frac{5}{4}\\
		\lesssim &	
		\| \uu_2 \|_{L^2}^\frac{2}{3}
		\|\,\nabla			\uu_2 	\,\|_{L^2	}^2
		\|\,			\delta 	\uu^h		\,\|_{\Hh^{-\frac{1}{2}}	}^2\,+\,
		\ee
		\|\,	\nabla	\delta 	\uu		\,\|_{\Hh^{-\frac{1}{2}}	}^2.
	\end{align*}
	Now, we consider the contribution of the magnetization in \eqref{uniq-momeq} given by 
	\begin{equation*}
	\begin{aligned}
	    \big|
		\langle	\,	 \nabla M_1 \odot \nabla \delta M&	+\nabla \delta M\odot \nabla M_2	,	\nabla \delta \uu	\,	\rangle_{\Hh^{-\frac{1}{2}}}
		\big|
		\lesssim
		\|\,			\nabla M_1 \odot \nabla \delta M +\nabla \delta M\odot \nabla M_2 	\,\|_{\Hh^{-\frac{1}{2}}	}
		\|\,			\nabla \delta 	\uu						\,\|_{\Hh^{-\frac{1}{2}}	}\\
		&\lesssim
		\|\,			\nabla \delta M	\,\|_{\Hh^{-\frac{1}{4}}}
		\|\, 		\nabla (M_1,\,M_2) 	\,\|_{\Hh^\frac{3}{4}}
		\|\,		\nabla	\delta 	\uu		\,\|_{\Hh^{-\frac{1}{2}}	}\\
		&\lesssim\, 
		\|\,			\nabla \delta M		\,\|_{\Hh^{-\frac{1}{2}}}^\frac{3}{4}	
		\|\, 		\nabla (M_1,\,M_2)	\,\|_{L^2}^\frac{1}{4}
		\|\, \Delta(M_1,\,M_2)	\,\|_{L^2}^\frac{3}{4}
		\|\,		\nabla	\delta 	\uu		\,\|_{\Hh^{-\frac{1}{2}}	}
		\| \Delta \delta M \|_{\Hh^{-\frac{1}{2}}	}^\frac{1}{4}\\
		&\lesssim	
		\|\nabla (M_1,\,M_2) \|_{L^2}^\frac{2}{3}
		\|\,\Delta(M_1,\,M_2) 	\,\|_{L^2	}^2
		\|\,		\nabla	\delta 	M	\,\|_{\Hh^{-\frac{1}{2}}	}^2\,+\,
		\ee
		\|\,	\nabla	\delta 	\uu		\,\|_{\Hh^{-\frac{1}{2}}	}^2+
		\ee
		\|\,	\Delta	\delta 	M		\,\|_{\Hh^{-\frac{1}{2}}	}^2.
	\end{aligned}
	\end{equation*}
	The following notations have been introduced:
	$$
	    \|\, 		\nabla (M_1,\,M_2) 	\,\|_{\Hh^\frac{3}{4}} = 
	    \|\, 		\nabla M_1 	\,\|_{\Hh^\frac{3}{4}}+
	    \|\, 		\nabla M_2 	\,\|_{\Hh^\frac{3}{4}}
	,\quad  
	     \|\, 		\nabla (M_1,\,M_2) 	\,\|_{L^2} = 
	    \|\, 		\nabla M_1 	\,\|_{L^2}+
	    \|\, 		\nabla M_2 	\,\|_{L^2},
	$$
	with further $ \|\, 		\Delta (M_1,\,M_2) 	\,\|_{L^2} = 
	    \|\, 		\Delta M_1 	\,\|_{L^2}+
	    \|\, 		\Delta M_2 	\,\|_{L^2}$.
	We now state that the following inequality, whose proof is postponed to Lemma~\ref{lemma:log-ineq}:
	\begin{equation}\label{log-ineq}
	\begin{aligned}
	    	\Big|
		\langle	\,	 &\delta W'(F) F_1^\top+ W'(F_2) \delta F^\top	,	\nabla \delta \uu	\,	\rangle_{\Hh^{-\frac{1}{2}}}
		\Big|
		\lesssim
		(1+\| F_2 \|_{L^2})^6\| (F_1,\,F_2) \|_{H^1}^2
		\delta \EE(t)
		\ln\Big(1+\frac{1}{\delta \EE(t)}\Big)
		\\&+\,
		C
		\Big(
		\| (F_1,\,F_2,\,\uu_1,\,\uu_2)\,\|_{H^1}^\frac{3}{2}
		+
		(1+\| F_2 \|_{L^2})^6\| (F_1,\,F_2) \|_{H^1}^2
		\Big)
		\delta \EE(t)
		\,
		+\,
		\ee
		\|\,	\nabla	\delta 	\uu		\,\|_{H^{-\frac{1}{2}}	}^2+
		\ee
		\|\,	\nabla	\delta 	F		\,\|_{H^{-\frac{1}{2}}	}^2.
	\end{aligned}
	\end{equation}
	Finally, we consider the contribution of the external force, which is estimated as follows:
	\begin{equation*}
	\begin{aligned}
		\big|
		\mu_0\langle	\,	\nabla H_{\rm ext} \delta M	,	\delta \uu^h	\,	\rangle_{\Hh^{-\frac{1}{2}}}
		\big|
		&=
		\big|\mu_0\langle	\,	\nabla H_{\rm ext} \overline{\delta M}	,	\delta \uu^h	\,	\rangle_{\Hh^{-\frac{1}{2}}}
		+\mu_0\langle	\,	\nabla H_{\rm ext} \delta M^h	,	\delta \uu^h	\,	\rangle_{\Hh^{-\frac{1}{2}}}
		\big|\\
		&\lesssim
		\| \nabla H_{\rm ext} \|_{\Hh^{-\frac{1}{2}}}
		| \overline{\delta M}|
		\| 	\delta \uu^h \|_{\Hh^{-\frac{1}{2}}}
		+\| \nabla H_{\rm ext}\delta M^h \|_{\Hh^{-\frac{1}{2}}}
		\| 	\delta \uu^h \|_{\Hh^{-\frac{1}{2}}}
		\\
		&\lesssim
		\| H_{\rm ext} \|_{L^2}^\frac{1}{2}
		\|\nabla H_{\rm ext} \|_{L^2}^\frac{1}{2}
		\| \delta M\|_{L^2}
		\| 	\delta \uu^h \|_{\Hh^{-\frac{1}{2}}}
		+\| \nabla H_{\rm ext} \|_{L^2} \|\delta M^h \|_{\Hh^{\frac{1}{2}}}
		\| 	\delta \uu^h \|_{\Hh^{-\frac{1}{2}}}
		\\
		&\lesssim
		\big(
		\| H_{\rm ext} \|_{L^2}^\frac{1}{2}
		\|\nabla H_{\rm ext} \|_{L^2}^\frac{1}{2}
		+\| \nabla H_{\rm ext} \|_{L^2}
		\big)
		\big(\| \delta M\|_{H^\frac{1}{2}}^2+\| 	\delta \uu^h \|_{\Hh^{-\frac{1}{2}}}^2\big).
	\end{aligned}
	\end{equation*}
	Coupling the previous inequalities together with \eqref{average-ineq} and identity \eqref{uniq-momeq}, we deduce that there exist a constant $C_1>0$ and a function $g_1 \in L^1(T_i,\,\widetilde T)$, for any $i =1, \dots N$ and $\widetilde T<T_{i+1}$, such that the following inequality holds true
	\begin{equation}\label{ineq1}
	\begin{aligned}
		\frac 12 |\overline{\delta \uu}(t)|^2ds 
		&+ 
		\frac 12\|\,\delta \uu^h(t)\,\|_{\Hh^{-\frac{1}{2}}}^2+
		\nu \int_{T_i}^t \|\,\nabla \delta \uu(s)\,\|_{\Hh^{-\frac{1}{2}}}^2ds\leq 
		\\
		&\leq 
		\frac 12\|\,\delta \uu(T_i)\,\|_{\Hh^{-\frac{1}{2}}}^2 
		\,+\,
		\int_{T_i}^t
		g_1(s)
		\mu(\delta \EE(s))
		d s
		+C_1\ee \int_{T_i}\delta \DD(s)d s,
	\end{aligned}
	\end{equation}
	where the inequality is well defined  for any $i \in \NN$ and for any  $t \in (T_i,\,T_{i+1})$.
	
	\noindent
	We now take into account the difference $\delta F=F_1-F_2$. We begin with formulating the evolutionary equation satisfied by $\delta F$:
	\begin{equation*}
	\begin{aligned}
		\partial_t \delta F  + \delta \uu \cdot \nabla  F_1+  \uu_2 \cdot \nabla  \delta  F - \nabla \delta \uu F_1  -  \nabla \uu_2 \delta F= \kappa \Delta \delta F
	\end{aligned}
	\end{equation*}
	We then proceed similarly as for proving identity \eqref{uniq-momeq}, obtaining:
	\begin{align}\label{inertia}
	\begin{split}
		\frac 12\| \delta F(t)			\|_{\Hh^{-\frac{1}{2}}}^2+
		\kappa
		\int_{T_i}^t
		\|\nabla \delta F(s)\|_{\Hh^{-\frac{1}{2}}}^2d s
		=&
		\frac 12\| \delta F(T_i)			\|_{\Hh^{-\frac{1}{2}}}^2
		-
		\int_{T_i}^t
		\Big\{
		\langle	\delta	\uu\cdot \nabla F_1	,	\delta F		\,\rangle_{\Hh^{-\frac{1}{2}}}-
		\langle		\uu_2 \cdot \nabla	\delta F	,	\delta F		\rangle_{\Hh^{-\frac{1}{2}}}\\ &+
		\langle		\nabla \delta \uu F_1	,	\delta F		\rangle_{\Hh^{-\frac{1}{2}}}+
		\langle	\nabla \uu_2 \delta F	,	\delta F		\rangle_{\Hh^{-\frac{1}{2}}}
		\Big\}(s)\dd s.
	\end{split}
	\end{align}
	We then proceed by estimating each term on the right-hand side. First we have
	\begin{equation*}
		\langle	\delta	\uu\cdot \nabla F_1	,	\delta F		\,\rangle_{\Hh^{-\frac{1}{2}}} = 
		\langle	\overline{\delta	\uu}\cdot \nabla F_1	,	\delta F		\,\rangle_{\Hh^{-\frac{1}{2}}} + 
		\langle	\delta	\uu^h\cdot \nabla F_1	,	\delta F		\,\rangle_{\Hh^{-\frac{1}{2}}},
	\end{equation*}		
	thus
	\begin{equation*}
	\begin{aligned}
		|\langle	\overline{\delta	\uu}\cdot \nabla F_1	,	\delta F		\,\rangle_{\Hh^{-\frac{1}{2}}}|
		&\leq 
		| \overline{\delta	\uu} |
		\| \nabla F_1 \|_{\Hh^{-\frac{1}{2}}}
		\| \delta F \|_{\Hh^{-\frac{1}{2}}}\\
		&\lesssim 
		|\overline{\delta	\uu}|
		\| F_1 \|_{L^2}^\frac{1}{2}\| \nabla F_1 \|_{L^2}^\frac{1}{2}
		\| \delta F \|_{\Hh^{-\frac{1}{2}}}
		\lesssim 
		\| F_1 \|_{L^2}^\frac{1}{2}\| \nabla F_1 \|_{L^2}^\frac{1}{2}
		\Big(
		|\overline{\delta	\uu}|^2+
		\| \delta F \|_{\Hh^{-\frac{1}{2}}}^2
		\Big),
	\end{aligned}
	\end{equation*}
	as well as
	\begin{equation*}
	\begin{aligned}
	|\langle	\delta	\uu^h\cdot \nabla F_1	,	\delta F		\,\rangle_{\Hh^{-\frac{1}{2}}}|
	&=
	|\langle	\delta	\uu^h\otimes F_1	,	 \nabla\delta F		\,\rangle_{\Hh^{-\frac{1}{2}}}|
	\leq
	\| \delta	\uu^h\otimes F_1	 \|_{\Hh^{-\frac 12}} 
	\| \nabla \delta F \|_{\Hh^{-\frac{1}{2} }}\\
	&\lesssim
	\| \delta	\uu^h \|_{\Hh^{-\frac 14}}
	\| F_1	 \|_{\Hh^{\frac 34}} 
	\| \nabla \delta F \|_{\Hh^{-\frac{1}{2} }}
	\lesssim
	\| \delta	\uu^h \|_{\Hh^{-\frac 12}}^\frac{3}{4}
	\| F_1	 \|_{L^2}^\frac{1}{4} 
	\|\nabla F_1	 \|_{L^2}^\frac{3}{4}
	\| \nabla \delta F \|_{\Hh^{-\frac{1}{2} }}
	\|\nabla  \delta	\uu \|_{\Hh^{-\frac 12}}^\frac{1}{4}\\
	&\lesssim
	\| F_1	 \|_{L^2}^\frac{2}{3}
	\|\nabla F_1	 \|_{L^2}^2
	\| \delta	\uu^h \|_{\Hh^{-\frac 12}}^2
	+
	\ee
	\| \nabla \delta F \|_{\Hh^{-\frac{1}{2} }}^2
	+
	\ee
	\| \nabla \delta	\uu \|_{\Hh^{-\frac 12}}^2.
	\end{aligned}
	\end{equation*}
	Next, we split the second term of \eqref{inertia} into 
	$ \langle		\uu_2\cdot \nabla \delta  F	,	\delta F		 \rangle_{\Hh^{-\frac{1}{2}}} = 
	\langle		\overline{\uu}_2\cdot \nabla \delta  F	,	\delta F		 \rangle_{\Hh^{-\frac{1}{2}}} + 
	\langle		\uu_2^h\cdot \nabla \delta  F	,	\delta F		 \rangle_{\Hh^{-\frac{1}{2}}}$. Therefore
	\begin{equation*}
	    \big|\langle	\overline{\uu}_2\cdot \nabla \delta  F	,	\delta F		 \rangle_{\Hh^{-\frac{1}{2}}}\big|
	    \lesssim
	    \big| \overline{\uu}_2 \big| 
	    \| \nabla \delta  F \|_{\Hh^{-\frac{1}{2}}}
	    \| \delta F \|_{\Hh^{-\frac{1}{2}}}
	    \lesssim
	    \| \uu_2 \|_{L^2}^2\| \delta F \|_{\Hh^{-\frac{1}{2}}}^2 + \varepsilon 
	    \| \nabla \delta F \|_{\Hh^{-\frac{1}{2}}}^2,
	\end{equation*}
	as well as
	\begin{equation*}
	\begin{aligned}
	|\langle		\uu_2^h\cdot \nabla \delta  F	,	\delta F		\,\rangle_{\Hh^{-\frac{1}{2}}}|
	&=
	|\langle		\uu_2^h\otimes \delta F	,	 \nabla\delta F		\,\rangle_{\Hh^{-\frac{1}{2}}}|
	\leq
	\| \uu_2^h\otimes\delta F	 \|_{\Hh^{-\frac 12}} 
	\| \nabla \delta F \|_{\Hh^{-\frac{1}{2} }}\\
	&\lesssim
	\| \delta F \|_{\Hh^{-\frac 14}}
	\| \uu_2^h	 \|_{\Hh^{\frac 34}} 
	\| \nabla \delta F \|_{\Hh^{-\frac{1}{2} }}
	\lesssim
	\| \delta	F \|_{\Hh^{-\frac 12}}^\frac{3}{4}
	\| \uu_2^h	 \|_{L^2}^\frac{1}{4} 
	\|\nabla \uu_2^h	 \|_{L^2}^\frac{3}{4}
	\| \nabla \delta	F \|_{\Hh^{-\frac 12}}^\frac{5}{4}\\
	&\lesssim
	\| \uu_2	 \|_{L^2}^\frac{2}{3}
	\|\nabla \uu_2	 \|_{L^2}^2
	\| \delta	F \|_{\Hh^{-\frac 12}}^2
	+
	\ee
	\| \nabla \delta F \|_{\Hh^{-\frac{1}{2} }}^2,
	\end{aligned}
	\end{equation*}
	where we recall that $\| \uu_2^h \|_{L^2}\leq 	\| \uu_2	 \|_{L^2}$ and 
	$\nabla \uu_2^h = \nabla \uu_2	$. 
	Similarly, using the free divergence condition $\Div F_i^\top = 0$, for $i=1,2$, we remark that
	\begin{equation*}
	\begin{aligned}
	    \big|\langle		\nabla \delta \uu F_1	,	\delta F		\rangle_{\Hh^{-\frac{1}{2}}}\big| 
	    &= 
	    \big|\langle  \delta \uu^h \otimes F_1	, \nabla \delta F		\rangle_{\Hh^{-\frac{1}{2}}}\big|
	    \lesssim 
	    \| \delta \uu^h \otimes F_1 \|_{\Hh^{-\frac{1}{2}}}
	    \|  \nabla \delta F \|_{\Hh^{-\frac{1}{2}}}\\
	    &\lesssim 
	    \| \delta \uu^h \|_{\Hh^{-\frac{1}{4}}}
	    \| F_1 \|_{\Hh^{\frac{3}{4}}}
	     \|  \nabla \delta F \|_{\Hh^{-\frac{1}{2}}}
	    \lesssim 
	    \| F_1 \|_{L^2}^\frac{1}{4}
	    \| \nabla F_1 \|_{L^2}^\frac{3}{4}
	    \| \delta \uu^h \|_{\Hh^{-\frac{1}{2}}}^\frac{3}{4}
	    \|\nabla \delta \uu^h \|_{\Hh^{-\frac{1}{2}}}^\frac{1}{4}
	    \|  \nabla \delta F \|_{\Hh^{-\frac{1}{2}}}\\
	    &\lesssim 
	    \| F_1 \|_{L^2}^\frac{2}{3}
	    \| \nabla F_1 \|_{L^2}^2
	    \| \delta \uu^h \|_{\Hh^{-\frac{1}{2}}}^2+
	    \varepsilon
	    \Big(
	    \|\nabla \delta \uu \|_{\Hh^{-\frac{1}{2}}}^2+
	    \|  \nabla \delta F \|_{\Hh^{-\frac{1}{2}}}^2
	    \Big),
	\end{aligned}
	\end{equation*}
	with moreover
	\begin{equation*}
	\begin{aligned}
	|\langle	\nabla \uu_2 \delta F	,	\delta F		\rangle_{\Hh^{-\frac{1}{2}}}|
	&\lesssim 
	\| \nabla \uu_2 \delta F \|_{\Hh^{-\frac{1}{2}}}
	\| \delta F \|_{\Hh^{-\frac{1}{2}}}
	\lesssim 
	\| \nabla \uu_2 \|_{L^2}
	\| \delta F \|_{\Hh^{\frac{1}{2}}}
	\| \delta F \|_{\Hh^{-\frac{1}{2}}}
	\\
    &\lesssim
	\|\nabla \uu_2	 \|_{L^2}^2
	\| \delta	F \|_{\Hh^{-\frac 12}}^2
	+
	\ee
	\| \nabla \delta F \|_{\Hh^{-\frac{1}{2} }}^2.
	\end{aligned}
	\end{equation*}
	Summarizing again the previous inequalities together with \eqref{inertia}, we deduce that there exist a constant $C_2>0$ and a function $g_2 \in L^1(T_i,\,T_{i+1})$, for any $i =1, \dots N$, such that the following inequality holds true
	\begin{equation}\label{ineq2}
		\frac 12 \|\,\delta F(t)\,\|_{\Hh^{-\frac{1}{2}}}^2+
		\kappa \int_{T_i}^t\|\,\nabla \delta F(s)\,\|_{\Hh^{-\frac{1}{2}}}^2\dd s
		\leq 
		\frac 12 \|\,\delta F(T_i)\,\|_{\Hh^{-\frac{1}{2}}}^2+
		\int_{T_i}^t
		g_2(s)
		\delta \EE(s)
		\dd s
		+C_2\ee \int_{T_i}^t\delta \DD(s)\dd s,
	\end{equation}
	where the inequality is well defined  for any $i \in \NN$ and for any  $t \in (T_i,\,T_{i+1})$.
	
	\noindent	We now take into account the difference $\delta M$ between the two magnetization fields in \eqref{main_equation}. Denoting by $\delta \psi'(M) := \psi'(M_1)-\psi'(M_2)$, we first formulate the differential equation satisfied by $\delta M$:
	\begin{align}\label{eq:deltaM}
	\begin{split}
		\partial_t \delta M + &\uu_1 \cdot \nabla \delta M
		\!+\! \delta \uu \cdot \nabla M_2 \!-\! \Delta \delta M  = - \delta M \wedge \Delta M_1 - M_2 \wedge \Delta \delta M - \mu_0 \delta M \wedge H_{\rm ext} 
		 + (\nabla \delta M: \nabla M_1)M_1\\ &+
		(\nabla M_2 :\nabla \delta M)M_1 + |\nabla M_2|^2\delta M 
		-\mu_0(\delta M\cdot H_{\rm ext})M_1 -\mu_0(M_2\cdot H_{\rm ext})\delta M
		+ \delta M\wedge \psi'(M_1)\\
		&+ M_2\wedge \delta \psi'(M) +\delta M\wedge M_1 \wedge \psi'(M_1)
		+ M_2\wedge \delta M \wedge \psi'(M_1)
		+ M_2\wedge  M_2 \wedge \delta \psi'(M).
	\end{split}
	\end{align}
	As a-priori estimate, we take the $L^2$ scalar product of this equation with $\delta M$. Thus, recalling that $\uu_1$ has null divergence, we eventually gather that
	\begin{align}\label{deltaM-L2}
		\begin{split}
			\frac 12 \| \delta M&	(t)		\|_{L^2}^2 +
			\int_{T_i}^t\|\nabla \delta M(s)\|_{L^2}^2
			\dd s
			=
			\frac 12 \| \delta M	(T_i)		\|_{L^2}^2
			-
			\int_{[T_i,t)\times\TT^2}\hspace{-0.8cm}
			\delta \uu \cdot \nabla M_2\cdot \delta M
		-
		\int_{[T_i,t)\times \TT^2}\hspace{-0.8cm}
		M_2 \wedge \Delta \delta M\cdot \delta M \\
		&+
		\int_{[T_i,t)\times\TT^2}\hspace{-0.8cm}
		 (\nabla \delta M: \nabla M_1)M_1\cdot \delta M
		 +
		 \int_{[T_i,t)\times\TT^2}\hspace{-0.8cm}
		 (\nabla  M_2: \nabla \delta  M)M_1\cdot \delta M 
		 +
		 \int_{[T_i,t)\times\TT^2}\hspace{-0.8cm}
		 |\nabla M_2|^2|\delta M |^2\\
		 &- \mu_0\int_{[T_i,t)\times\TT^2}\hspace{-0.8cm}
		(\delta M\cdot H_{\rm ext})M_1\cdot \delta M 
		 - \mu_0\int_{[T_i,t)\times\TT^2}\hspace{-0.8cm}
		(M_2\cdot H_{\rm ext})|\delta M |^2
		 +\int_{[T_i,t)\times\TT^2}\hspace{-0.8cm}
		(M_2\wedge\delta \psi'(M))\cdot \delta M \\
		 &+\int_{[T_i,t)\times\TT^2}\hspace{-0.8cm}
		(M_2\wedge \delta M \wedge \psi'(M_1)) \cdot \delta M)
		 +\int_{[T_i,t)\times\TT^2}\hspace{-0.8cm}
		(M_2\wedge  M_2 \wedge \delta \psi'(M))  \cdot \delta M).
		\end{split}
	\end{align}
	We proceed estimating each term on the right-hand side of the previous inequality. The first inner product we deal with is
\begin{equation*}
	\int_{\TT^2}
				\delta \uu \cdot \nabla M_2\cdot \delta M
	= 
	\int_{\TT^2}
			\overline{\delta \uu} \cdot \nabla M_2\cdot \delta M +
			\int_{\TT^2}
				\delta \uu^h \cdot \nabla M_2\cdot \delta M,
\end{equation*}
so that
\begin{equation*}
	\Big|
	\int_{\TT^2}
			\overline{\delta \uu} \cdot \nabla M_2\cdot \delta M
    \Big|
	\leq 
	| \overline{\delta \uu} | \| \nabla M_2 \|_{L^2} \| \delta M \|_{L^2}
	\leq 
	 \| \nabla M_2 \|_{L^2} 
	 \Big( | \overline{\delta \uu} |^2 + \| \delta M \|_{L^2}^2
	 \Big)
\end{equation*}
and
\begin{equation*}
		\begin{aligned}
			\Big|\int_{\TT^2}
				\delta \uu^h \cdot \nabla M_2\cdot \delta M
			\Big|
			&\leq
			\| \delta \uu^h \|_{L^2}
			\| \nabla M_2 \|_{L^4}
			\| \delta M \|_{L^4}
			\lesssim
			\| \delta \uu^h \|_{\Hh^{-\frac{1}{2}}}^\frac{1}{2}
			\|\nabla \delta \uu \|_{\Hh^{-\frac{1}{2}}}^\frac{1}{2}
			\| \nabla M_2 \|_{L^2}^\frac{1}{2}
			\| \Delta M_2 \|_{L^2}^\frac{1}{2}
			\| \delta M \|_{L^2}^\frac{1}{2}
			\| \nabla \delta M \|_{L^2}^\frac{1}{2}\\
			&\lesssim
			\| \nabla M_2 \|_{L^2}^2
			\| \Delta M_2 \|_{L^2}^2
			\| \delta \uu^h \|_{\Hh^{-\frac{1}{2}}}^2
			+
			\| \delta M \|_{L^2}^2
			+
			\ee
			\|\nabla \delta \uu \|_{\Hh^{-\frac{1}{2}}}^2
			+
			\ee
			\| \nabla \delta M \|_{L^2}^2
			.
		\end{aligned}
	\end{equation*}
	Secondly, we take into account the precession contribution of the LLG equation, that is
	\begin{equation*}
		\begin{aligned}
			\Big|\int_{\TT^2}M_2 \wedge \Delta \delta M\cdot \delta M\Big|
			& = \Big| \sum_{i=1}^2 \int_{\TT^2}M_2 \wedge \partial_{ii}^2 \delta M\cdot \delta M\Big| \\
			&= \Big| -\sum_{i=1}^2 \int_{\TT^2}\partial_i M_2 \wedge \partial_{i} \delta M\cdot \delta M-\sum_{i=1}^2 \int_{\TT^2} \underbrace{M_2 \wedge \partial_{i} \delta M\cdot \partial_{i} \delta M}_{=0}\Big|\\
			&\lesssim
			\| \nabla M_2 \|_{L^4}\| \nabla \delta M \|_{L^2} \| \delta M \|_{L^4}
			\lesssim
			\| \nabla M_2 \|_{L^2}^\frac{1}{2}
			\| \Delta M_2 \|_{L^2}^\frac{1}{2}
			\| \nabla \delta M \|_{L^2}^\frac{3}{2}
			\| \delta M \|_{L^2}^\frac{1}{2}\\
			&\lesssim 
			\| \nabla M_2 \|_{L^2}^2
			\| \Delta M_2 \|_{L^2}^2
			\| \delta M   \|_{L^2}^2 +
			\varepsilon
			\| \nabla \delta M \|_{L^2}^2,
		\end{aligned}
	\end{equation*}
	while, taking into account the Lagrangian multiplier for the unitary constraint of $M$,
	\begin{equation*}
		\begin{aligned}
		\Big|
			\int_{\TT^2}
		 (\nabla \delta M: \nabla M_1)M_1\cdot \delta M 
		\Big|
		 &\leq 
		 \| \nabla \delta M \|_{L^2} \| \nabla M_1 \|_{L^4}  \| \delta M \|_{L^4}
		 \lesssim
		 \| \nabla M_1 \|_{L^2}^\frac{1}{2}\| \Delta M_1 \|_{L^2}^\frac{1}{2}  \| \delta M \|_{L^2}^\frac{1}{2}\| \nabla \delta M \|_{L^2} ^\frac{3}{2}\\
		 &\lesssim
		  \| \nabla M_1 \|_{L^2}^2\| \Delta M_1 \|_{L^2}^2  \| \delta M \|_{L^2}^2 +\ee \| \nabla \delta M \|_{L^2} ^2
		\end{aligned}
	\end{equation*}
	with
	\begin{equation*}
		\begin{aligned}
		\Big|
			\int_{\TT^2}
		 (\nabla  M_2: \nabla \delta M)M_1\cdot \delta M 
		 \Big|
		 \lesssim
		  \| \nabla M_2 \|_{L^2}^2\| \Delta M_2 \|_{L^2}^2  \| \delta M \|_{L^2}^2 +\ee \| \nabla \delta M \|_{L^2} ^2,
		\end{aligned}
	\end{equation*}
	and moreover
	\begin{equation*}
	\begin{aligned}
		  \int_{\TT^2}
		 |\nabla M_2|^2|\delta M |^2
		 &\leq 
		 \| \nabla M_2 \|_{L^4}^2
		 \| \delta M \|_{L^4}^2
		 \lesssim
		 \| \nabla M_2 \|_{L^2}
		 \| \Delta M_2 \|_{L^2}
		 \| \delta M \|_{L^2}
		 \| \nabla \delta M \|_{L^2}\\
		 &\lesssim
		  \| \nabla M_2 \|_{L^2}^2
		 \| \Delta M_2 \|_{L^2}^2
		 \| \delta M \|_{L^2}^2
		 +
		 \ee
		 \| \nabla \delta M \|_{L^2}^2.
	\end{aligned}
	\end{equation*}
	Next, we address the following nonlinear terms:
	\begin{equation*}
		\begin{aligned}
			 \Big| \int_{\TT^2}
		(\delta M\cdot H_{\rm ext})M_1\cdot \delta M 
		 + \int_{\TT^2}
		(M_2\cdot H_{\rm ext})\delta M \cdot \delta M
		\Big|
		\leq 
		\| \delta M \|_{L^4}^2 \| H_{\rm ext} \|_{L^2}
		\lesssim
		\| H_{\rm ext} \|_{L^2}^2 \|\delta M \|_{L^2}^2 + \ee \| \nabla \delta M \|_{L^2}^2.
		\end{aligned}
	\end{equation*}
	Furthermore, recalling that $\psi(M)$ is a $\mathcal{C}^2$ function in $\mathbb{R}^3$, then $\psi'(M)$ is Lipschitz in $\mathbb{R}^3$, i.e.\ there exists a constant $C>0$ such that 
	$|\psi'(M_1) - \psi'(M_2)|\leq C |M_1-M_2|$. Thus
	\begin{equation*}
	    \begin{aligned}
	        \Big|\int_{\TT^2} (M_2\cdot H_{\rm ext})|\delta M |^2\Big|
	        &\lesssim
	        \| H_{\rm ext} \|_{L^2}\| \delta M \|_{L^4}^2
	        \lesssim 
	        \| H_{\rm ext} \|_{L^2}^2 \|\delta M \|_{L^2}^2 + \ee \| \nabla \delta M \|_{L^2}^2,\\
	        \Big|\int_{\TT^2} (M_2\wedge\delta \psi'(M))\cdot \delta M \Big|
	        &\leq 
	        \| \delta \psi'(M) \|_{L^2}\| \delta M \|_{L^2}
	        \lesssim
	        \| \delta M \|_{L^2}^2,\\
	       \Big| \int_{\TT^2}
		    (M_2\wedge \delta M \wedge \psi'(M_1)) \cdot \delta M)\Big|
		    &
		    \lesssim \| \delta M \|_{L^2}^2\\
		   \Big| \int_{\TT^2} (M_2\wedge  M_2 \wedge \delta \psi'(M))  \cdot \delta M)\Big|
		    &\leq 
		    \| \delta \psi'(M) \|_{L^2}\| \delta M \|_{L^2}
		    \lesssim
		    \| \delta M \|_{L^2}^2.
	    \end{aligned}
	\end{equation*}
	We hence couple the previous inequalities together with identity \eqref{deltaM-L2}. We deduce that there exist a constant $C_3>0$ and a function $g_3 \in L^1(T_i,\,
	\widetilde T)$, for any $\widetilde T\in [T_i,T_{i+1})$, such that the following inequality holds true
	\begin{equation}\label{ineq3}
		\frac{1}{2}\|\,\delta M(t)\,\|_{L^2}^2+
		\int_{T_i}^t\|\,\nabla \delta M(s)\,\|_{L^2}^2d s
		\leq 
		\frac{1}{2}\|\,\delta M(T_i)\,\|_{L^2}^2
		+
		\int_{T_i}^t
		g_3(s)
		\delta \EE(s)
		d s
		+C_3\ee \int_{T_i}^t \delta \DD(s) d s,
	\end{equation}
	where the inequality is well defined for any singular time $T_i$ and for any  $t \in (T_i,\,T_{i+1})$.
	
	\noindent It remains to control the $\Hh^{-1/2}(\TT^2)$-norm of $\delta M$ to conclude our estimates. We take the inner product in such a space between $-\Delta \delta M$ and equation \eqref{eq:deltaM}, to gather that
	\begin{align}\label{H12norm_of_delta_M}
		\begin{split}
			\frac 12
			\frac{\dd}{\dd t}
			\| & \nabla \delta M			\|_{\Hh^{-\frac{1}{2}}}^2+
			\|\Delta \delta M\|_{\Hh^{-\frac{1}{2}}}^2
		=
		\langle	   \uu_1 \cdot \nabla \delta M	,		\Delta \delta M		\rangle_{\Hh^{-\frac{1}{2}}}
		+
		\langle		\delta  \uu    \cdot \nabla M_2			,		\Delta \delta M	\rangle_{\Hh^{-\frac{1}{2}}}
		+
		\langle		\delta M \wedge \Delta M_1			,		\Delta \delta M		\rangle_{\Hh^{-\frac{1}{2}}}\\
		&+
		\langle		M_2 \wedge\Delta \delta M			,		\Delta \delta M		\,\rangle_{\Hh^{-\frac{1}{2}}}
		+
		\mu_0\langle		\delta M \wedge H_{\rm ext}		,		\Delta \delta M		\,\rangle_{\Hh^{-\frac{1}{2}}}
		-
		\langle		 (\nabla \delta M: \nabla M_1)M_1	,		\Delta \delta M		\,\rangle_{\Hh^{-\frac{1}{2}}}\\
		&-
		\langle		 (\nabla M_2 :\nabla \delta M)M_1	,		\Delta \delta M		\,\rangle_{\Hh^{-\frac{1}{2}}}
		-
		\langle		|\nabla M_2|^2\delta M		,		\Delta \delta M		\,\rangle_{\Hh^{-\frac{1}{2}}}
		+\mu_0\langle		(\delta M \cdot H_{\rm ext})M_1		,		\Delta \delta M		\,\rangle_{\Hh^{-\frac{1}{2}}}\\
		&+\mu_0\langle		(M_2 \cdot H_{\rm ext})\delta M		,		\Delta \delta M		\,\rangle_{\Hh^{-\frac{1}{2}}}
		- \langle		\delta M\wedge \psi'(M_1)		,		\Delta \delta M		\,\rangle_{\Hh^{-\frac{1}{2}}}
		-\langle		M_2\wedge \delta \psi'(M)		,		\Delta \delta M		\,\rangle_{\Hh^{-\frac{1}{2}}}\\
		&-\langle		\delta M\wedge M_1 \wedge \psi'(M_1)		,		\Delta \delta M		\,\rangle_{\Hh^{-\frac{1}{2}}}
		-\langle		M_2\wedge \delta M \wedge \psi'(M_1)		,		\Delta \delta M		\,\rangle_{\Hh^{-\frac{1}{2}}}
		-\langle		M_2\wedge M_2 \wedge \delta \psi'(M)		,		\Delta \delta M		\,\rangle_{\Hh^{-\frac{1}{2}}}.
		\end{split}
	\end{align}	 
	The first term on the right-hand side is controlled by
	\begin{equation*}
		\langle	   \uu_1 \cdot \nabla \delta M	,		\Delta \delta M		\,\rangle_{\Hh^{-\frac{1}{2}}}
		= 
		\langle	   \overline{\uu}_1 \cdot \nabla \delta M	,		
						\Delta \delta M		\,\rangle_{\Hh^{-\frac{1}{2}}} + 
		\langle	   \uu_1^h \cdot \nabla \delta M	,		
						\Delta \delta M		\,\rangle_{\Hh^{-\frac{1}{2}}},			
	\end{equation*}
	thus
	\begin{align*}
		\big|\langle	   \overline{\uu}_1 \cdot \nabla \delta M	,		
						\Delta \delta M		\,\rangle_{\Hh^{-\frac{1}{2}}}\big|
		\lesssim
		|  \overline{\uu}_1 |\| \nabla \delta M \|_{\Hh^{-\frac{1}{2}}}
		\|\Delta \delta M	 \|_{\Hh^{-\frac{1}{2}}}
		\lesssim 
		\| \uu_1 \|_{L^2}^2 \| \nabla \delta M \|_{\Hh^{-\frac{1}{2}}}^2 +
		\varepsilon 
		\| \Delta \delta M \|_{\Hh^{-\frac{1}{2}}}^2
	\end{align*}
	and
	\begin{align*}
			\big|\langle	   \uu_1^h \cdot \nabla \delta M	,		\Delta \delta M		\,\rangle_{\Hh^{-\frac{1}{2}}}\big|
			\leq&
			\| \uu_1^h\cdot \nabla \delta M \|_{\Hh^{-\frac{1}{2}}}
			\| \Delta \delta M \|_{\Hh^{-\frac{1}{2}}}
			\lesssim
			\| \uu_1^h \|_{\Hh^\frac{3}{4}} \| \nabla \delta M \|_{\Hh^{-\frac{1}{4}}}
			\| \Delta \delta M \|_{\Hh^{-\frac{1}{2}}} \\
			\lesssim&
			\| \uu_1^h \|_{L^2}^\frac{1}{4} \| \nabla \uu_1 \|_{L^2}^\frac{3}{4} \| \nabla \delta M \|_{\Hh^{-\frac{1}{2}}}^\frac{3}{4}
			\| \Delta \delta M \|_{\Hh^{-\frac{1}{2}}}^\frac{5}{4}
			\lesssim
			\| \uu_1 \|_{L^2}^\frac{2}{3} \| \nabla \uu_1 \|_{L^2}^2
			 \| \nabla \delta M \|_{\Hh^{-\frac{1}{2}}}^2+
			 \ee
			\| \Delta \delta M \|_{\Hh^{-\frac{1}{2}}}^2.
	\end{align*}
	The second term is similarly handled through
	$ \langle		\delta  \uu    \cdot \nabla M_2			,		\Delta \delta M		\,\rangle_{\Hh^{-\frac{1}{2}}} = \langle	\overline{\delta  \uu}    \cdot \nabla M_2			,		\Delta \delta M		\,\rangle_{\Hh^{-\frac{1}{2}}} + 
	\langle		\delta  \uu^h    \cdot \nabla M_2			,		\Delta \delta M		\,\rangle_{\Hh^{-\frac{1}{2}}}$ and
	\begin{equation*}
		\begin{aligned}	
		\big|\langle	\overline{\delta  \uu}    \cdot \nabla M_2			
		    ,		\Delta \delta M		\,\rangle_{\Hh^{-\frac{1}{2}}} \big|
		&\leq
		| \overline{\delta \uu}| \| \nabla M_2 \|_{\Hh^{-\frac{1}{2}}} 
		 \| \Delta \delta M \|_{\Hh^{-\frac{1}{2}}}
		 \lesssim
		 | \overline{\delta \uu}| \| M_2 \|_{L^2}^\frac{1}{2}\| \nabla M_2 \|_{L^2}^\frac{1}{2} 
		 \| \Delta \delta M \|_{\Hh^{-\frac{1}{2}}}
		 \\
		  &\lesssim
		\| \nabla M_2 \|_{L^2}
		 | \overline{\delta \uu}|^2+
		 \varepsilon
		 \| \Delta \delta M \|_{\Hh^{-\frac{1}{2}}}^2
		\\
				\big|\langle		\delta  \uu^h    \cdot \nabla M_2			,		\Delta \delta M		\,\rangle_{\Hh^{-\frac{1}{2}}}\big|
		&\lesssim
		\| 		\delta  \uu^h \|_{\Hh^{-\frac{1}{4}}}
		\| \nabla M_2 \|_{\Hh^\frac{3}{4}}
		\| \Delta \delta M \|_{\Hh^{-\frac{1}{2}}}
		\lesssim
			\| 		\delta  \uu^h \|_{\Hh^{-\frac{1}{2}}}^\frac{3}{4}
			\| \nabla	\delta  \uu^h \|_{\Hh^{-\frac{1}{2}}}^\frac{1}{4}
			\| \nabla M_2 \|_{L^2}^\frac{1}{4}
			\| \Delta M_2 \|_{L^2}^\frac{3}{4}
			\| \Delta \delta M \|_{\Hh^{-\frac{1}{2}}}
		\\
		&
		\lesssim
			\| \nabla M_2 \|_{L^2}^\frac{2}{3} \| \Delta M_2 \|_{L^2}^2
			 \| \delta \uu^h \|_{\Hh^{-\frac{1}{2}}}^2+
			 \ee
			\| \Delta \delta M \|_{\Hh^{-\frac{1}{2}}}^2
			+\ee \|\nabla \delta \uu \|_{\Hh^{-\frac{1}{2}}}^2.
		\end{aligned}
	\end{equation*}
	Coming back to \eqref{H12norm_of_delta_M}, we now address $ \langle		M_2 \wedge \Delta \delta M			,		\Delta \delta M		\,\rangle_{\Hh^{-\frac{1}{2}}}$. We first split the magnetization into its homogeneous and non-homogeneous components $ M_2 = \overline{M}_2 + M_2^h$. Then 
	\begin{equation*}
			\langle		M_2 \wedge \Delta \delta M			,		\Delta \delta M		\,\rangle_{\Hh^{-\frac{1}{2}}}
			=
			\underbrace{\langle		\overline M_2 \wedge \Delta \delta M			,		\Delta \delta M		\,\rangle_{\Hh^{-\frac{1}{2}}}}_{=0}+
			\langle		M_2^h \wedge \Delta \delta M			,		\Delta \delta M		\,\rangle_{\Hh^{-\frac{1}{2}}}.
	\end{equation*}
	We hence state the following inequality, whose proof is postponed to Section~\ref{sec5} (see Proposition~\ref{prop:blabla}):
	\begin{equation}\label{ineq:toprove1}
		\big|\langle		M_2^h \wedge \Delta \delta M			,		\Delta \delta M		\,\rangle_{\Hh^{-\frac{1}{2}}}\big|
		\lesssim
    	(1+\|\nabla M_2 \|_{L^2}^2)\|  \Delta M_2 \|_{L^2}^2 \|\nabla \delta M\|_{\Hh^{-\frac{1}{2}}}^2
		+ \ee\| \Delta \delta M \|_{\Hh^{-\frac{1}{2}}}^2.
	\end{equation}
	Furthermore, we obtain that
	\begin{equation*}
		\begin{aligned}
			\big|\langle		\delta M \wedge H_{\rm ext}		,		\Delta \delta M		\,\rangle_{\Hh^{-\frac{1}{2}}}\big|
		&\leq
		\| \delta M \wedge H_{\rm ext} \|_{L^2} 
		\|  \Delta \delta M \|_{\Hh^{-1}} 
		\lesssim
		\| \delta M \|_{L^4} 
		\| H_{\rm \ext} \|_{L^4}
		\| \nabla \delta M \|_{L^2}\\
		&\lesssim 
		\| \delta M \|_{L^2}^\frac{1}{2}
		\| H_{\rm \ext} \|_{L^2}^\frac{1}{2}
		\| \nabla H_{\rm \ext} \|_{L^2}^\frac{1}{2}
		\| \nabla \delta M \|_{L^2}^\frac{3}{2}
		\lesssim 
		\| H_{\rm \ext} \|_{L^2}^2
		\| \nabla H_{\rm \ext} \|_{L^2}^2
		\| \delta M \|_{L^2}^2
		+\varepsilon
		\| \nabla \delta M \|_{L^2}^2.
		\end{aligned}
	\end{equation*}
	
	\noindent We now  focus on the term $\langle		 (\nabla \delta M: \nabla M_1)M_1	,		\Delta \delta M		\,\rangle_{\Hh^{-\frac{1}{2}}}$ and we claim that a similar inequality holds also for 
	$\langle		 (\nabla  M_2: \nabla \delta M)M_1	,		\Delta \delta M		\,\rangle_{\Hh^{-\frac{1}{2}}}$. We first separate the homogeneous component of $M_1$ in the inner product by:
	\begin{equation*}
		\langle		 (\nabla \delta M: \nabla M_1)M_1	,		\Delta \delta M		\,\rangle_{\Hh^{-\frac{1}{2}}} = 
		\langle		 (\nabla \delta M: \nabla M_1)\overline{M}_1	,		\Delta \delta M		\,\rangle_{\Hh^{-\frac{1}{2}}}
		+
		\langle		 (\nabla \delta M: \nabla M_1)M_1^h	,		\Delta \delta M		\,\rangle_{\Hh^{-\frac{1}{2}}}.
	\end{equation*}
	We hence remark that
	\begin{equation*}
	\begin{aligned}
		\big|\langle		 (\nabla \delta M: \nabla M_1)\overline{M}_1	,		\Delta \delta M		\,\rangle_{\Hh^{-\frac{1}{2}}}\big|
		&\leq 
		| \overline{M}_1 | \| \nabla \delta M: \nabla M_1 \|_{\Hh^{-\frac{1}{2}}} \| \Delta \delta M \|_{\Hh^{-\frac{1}{2}}}\\
		&\lesssim
	 \|\nabla \delta M \|_{\Hh^{-\frac{1}{4}}}\|\nabla  M_1\|_{\Hh^{\frac{3}{4}}}\| \Delta \delta M \|_{\Hh^{-\frac{1}{2}}}\\
		&\lesssim
		 \|\nabla  M_1\|_{L^2}^\frac{1}{4}\|\Delta  M_1\|_{L^2}^\frac{3}{4}\|\nabla \delta M \|_{\Hh^{-\frac{1}{2}}}^\frac{3}{4}\| \Delta \delta M \|_{\Hh^{-\frac{1}{2}}}^\frac{5}{4}\\
		&\lesssim
		 \|\nabla  M_1\|_{L^2}^\frac{2}{3}\|\Delta  M_1\|_{L^2}^2\| \nabla \delta M \|_{\Hh^{-\frac{1}{2}}}^2 +\ee\| \Delta \delta M \|_{\Hh^{-\frac{1}{2}}}^2.
	\end{aligned}	
	\end{equation*}
	Furthermore,
	\begin{equation*}
	\begin{aligned}
		\big|\langle		 (\nabla \delta M: \nabla M_1)M_1^h	,		\Delta \delta M		\,\rangle_{\Hh^{-\frac{1}{2}}}\big|
		&\leq 
		 \| (\nabla \delta M: \nabla M_1)M_1^h \|_{\Hh^{-\frac{1}{2}}} \| \Delta \delta M \|_{\Hh^{-\frac{1}{2}}}\\
		&\lesssim
		 \|\nabla \delta M \|_{\Hh^{-\frac{1}{4}}}\|\nabla  M_1\otimes M_1^h\|_{\Hh^{\frac{3}{4}}}\| \Delta \delta M \|_{\Hh^{-\frac{1}{2}}}\\
		&\lesssim
		 \|\nabla  M_1\otimes M_1^h\|_{L^2}^\frac{1}{4}\|\nabla (\nabla  M_1\otimes M_1^h)\|_{L^2}^\frac{3}{4}\|\nabla \delta M \|_{\Hh^{-\frac{1}{2}}}^\frac{3}{4}\| \Delta \delta M \|_{\Hh^{-\frac{1}{2}}}^\frac{5}{4}\\
		&\lesssim
		 \|\nabla  M_1\|_{L^2}^\frac{1}{4}\| M_1^h \|_{L^\infty}^\frac{1}{4}\Big( \|\Delta  M_1 \|_{L^2} \| M_1^h\|_{L^\infty} + \| \nabla M_1 \|_{L^4}^2\Big)^\frac{3}{4}\|\nabla \delta M \|_{\Hh^{-\frac{1}{2}}}^\frac{3}{4}\| \Delta \delta M \|_{\Hh^{-\frac{1}{2}}}^\frac{5}{4}\\
		 &\lesssim
		 \|\nabla  M_1\|_{L^2}^\frac{1}{4}\Big( \|\Delta  M_1 \|_{L^2}  + \| \nabla M_1 \|_{L^2}\| \Delta M_1 \|_{L^2}\Big)^\frac{3}{4}\|\nabla \delta M \|_{\Hh^{-\frac{1}{2}}}^\frac{3}{4}\| \Delta \delta M \|_{\Hh^{-\frac{1}{2}}}^\frac{5}{4}\\
		 &\lesssim
		 \|\nabla  M_1\|_{L^2}^\frac{2}{3}\Big(1  + \| \nabla M_1 \|_{L^2}\Big)^2  \|\Delta  M_1 \|_{L^2}^2\|\nabla \delta M \|_{\Hh^{-\frac{1}{2}}}^2 +\ee \| \Delta \delta M \|_{\Hh^{-\frac{1}{2}}}^2.
	\end{aligned}	
	\end{equation*}
	
	\noindent 
	It hence remains to bound the last terms of identity \eqref{H12norm_of_delta_M}. First, we remark that
	\begin{equation*}
	\begin{aligned}
			\big|\langle		|\nabla M_2|^2\delta M		,		\Delta \delta M		\,\rangle_{\Hh^{-\frac{1}{2}}}\big|
			&=
			\big|\langle		|\nabla M_2|^2\overline{\delta M}		,		\Delta \delta M		\,\rangle_{\Hh^{-\frac{1}{2}}}
			+
			\langle		|\nabla M_2|^2\delta M^h		,		\Delta \delta M		\,\rangle_{\Hh^{-\frac{1}{2}}}\big|\\
			&\leq 
			| \overline{\delta M}|
			\| |\nabla M_2|^2 \|_{L^2}\| \Delta \delta  M \|_{\Hh^{-1}} +
			\| |\nabla M_2|^2 \|_{L^2}\| \delta M^h \|_{\Hh^{\frac{1}{2}}}\| \Delta \delta M \|_{\Hh^{-\frac{1}{2}}}\\
			&\lesssim 
			\| \delta M \|_{L^2}
			\| \nabla M_2 \|_{L^4}^2\| \nabla \delta  M \|_{L^2} +
			\| \nabla M_2 \|_{L^4}^2\|\nabla \delta M^h \|_{\Hh^{-\frac{1}{2}}}\| \Delta \delta M \|_{\Hh^{-\frac{1}{2}}}\\
			&\lesssim 
			\| \nabla M_2 \|_{L^2}^2
			\| \Delta M_2 \|_{L^2}^2
			\Big(
			\| \delta M \|_{L^2}^2
			+
			\|\nabla \delta M^h \|_{\Hh^{-\frac{1}{2}}}^2
			\Big)
			+
			\ee\| \nabla \delta  M \|_{L^2}^2 +
			\ee\| \Delta \delta M \|_{\Hh^{-\frac{1}{2}}}^2.
	\end{aligned}
	\end{equation*}
	Furthermore, we remark that
	\begin{equation*}
		\begin{aligned}
			\big|\langle		(\delta M  \cdot H_{\rm ext})M_1		,		\Delta \delta M		\,\rangle_{\Hh^{-\frac{1}{2}}}\big|
			&\lesssim
		    \|	(\delta M  \cdot H_{\rm ext})M_1 \|_{L^2} \| \Delta\delta M \|_{\Hh^{-1}}\\
			&\lesssim
			\|  \delta M \|_{L^4} \| H_{\rm ext} \|_{L^4} \| \nabla\delta M \|_{L^2}\\
			&\lesssim
			\|       H_{\rm ext}  \|_{L^2}^\frac{1}{2}
			\|\nabla H_{\rm ext}  \|_{L^2}^\frac{1}{2}
			\|        \delta M    \|_{L^2}^\frac{1}{2} 
			\| \nabla \delta M    \|_{L^2}^\frac{3}{2}\\
			&\lesssim
			\|       H_{\rm ext}  \|_{L^2}^2
			\|\nabla H_{\rm ext}  \|_{L^2}^2
			\|        \delta M    \|_{L^2}^2
			+\varepsilon
			\| \nabla \delta M    \|_{L^2}^2.
		\end{aligned}
	\end{equation*}
	Similarly,
	\begin{equation*}
			\big|\langle		( M_2 \cdot H_{\rm ext})\delta M		,		\Delta \delta M		\,\rangle_{\Hh^{-\frac{1}{2}}}\big|
			\lesssim
			\|       H_{\rm ext}  \|_{L^2}^2
			\|\nabla H_{\rm ext}  \|_{L^2}^2
			\|        \delta M    \|_{L^2}^2
			+\varepsilon
			\| \nabla \delta M    \|_{L^2}^2.
	\end{equation*}
	Moreover,	
	\begin{equation*}
		\begin{aligned}
			\big|\langle		\delta M\wedge \psi'(M_1)		,		\Delta \delta M		\,\rangle_{\Hh^{-\frac{1}{2}}}\big|
			&\leq 
			\| \delta M \wedge \psi'(M_1) \|_{L^2     }
			\| \Delta \delta M            \|_{\Hh^{-1}}\\
			&\lesssim 
			\| \delta M \|_{L^2}\| \nabla \delta M \|_{L^2}
			\lesssim
			\| \delta M \|_{L^2}^2
			+\varepsilon
			\| \nabla \delta M \|_{L^2}^2.
		\end{aligned}	
	\end{equation*}
	With similar procedures, one can easily control the remaining terms of \eqref{H12norm_of_delta_M}, i.e.
	\begin{align*}
	    \big|\langle		\delta M\wedge M_1 \wedge \psi'(M_1)	
		+		M_2\wedge \delta M \wedge \psi'(M_1)	+		M_2\wedge M_2 \wedge \delta \psi'(M)		,		\Delta \delta M		\,\rangle_{\Hh^{-\frac{1}{2}}}
		\big|
		\lesssim
		\| \delta M \|_{L^2}^2 + \varepsilon \|\nabla \delta M \|_{L^2}^2.
	\end{align*}
	Summarizing the previous inequalities and recalling the identity \eqref{H12norm_of_delta_M}, we hence deduce that there exist a constant $C_4>0$ and a function $g_4 \in L^1(T_i,\,\widetilde T)$, for any $i =1, \dots N$ and $\widetilde T<T_{i+1}$, such that the following inequality holds true
	\begin{equation}\label{ineq4}
		\frac 12
		\| \nabla \delta M(t)			\|_{\Hh^{-\frac{1}{2}}}^2+
		\int_{T_i}^t
		\|\Delta \delta M(s)\|_{\Hh^{-\frac{1}{2}}}^2\dd s
		\leq 
		\int_{T_i}^t
		g_4(s)
		\delta \EE(s)
		\dd s
		+C_4\ee 
		\int_{T_i}^t\delta \DD(s)\dd s,
	\end{equation}
	where the inequality is well defined for any $i \in \NN$ and for any $t \in (T_i,\,T_{i+1})$.
	
	\noindent The proof of Proposition~\ref{prop:estEE} is then accomplished by coupling the inequalities \eqref{ineq1}, \eqref{ineq2}, \eqref{ineq3} and \eqref{ineq4}, and assuming the parameter $\ee$ small enough to absorb the dissipative term on the right-hand side to the one of the left-hand side.
	\end{proof}

\section{Technical estimates} \label{sec5}
\noindent We are now in the position to prove inequality \eqref{ineq:toprove1}, which we recall in the following proposition.
\begin{prop}\label{prop:blabla}
	Let $M_1,M_2, \delta M$ and $M_2^h$ be as in Section~\ref{sec:uniqueness}. Then for any $\ee>0$, one has the following inequality
		\begin{equation*}
		\begin{aligned}
			\big|\langle		M_2^h \wedge \Delta \delta M			,		\Delta \delta M		\,\rangle_{\Hh^{-\frac{1}{2}}}\big|
		\lesssim
		(1+\|\nabla M_2 \|_{L^2}^2)\|  \Delta M_2 \|_{L^2}^2 \|\nabla \delta M\|_{\Hh^{-\frac{1}{2}}}^2
		+ \ee\| \Delta \delta M \|_{\Hh^{-\frac{1}{2}}}^2.
		\end{aligned}
		\end{equation*}
	\end{prop}
\begin{proof}
	We first recast the inner product with the formula given by \eqref{inner_product}, hence we apply the Bony decomposition \eqref{BONY}:
	\begin{align*}
		\big|\langle		M_2^h \wedge \Delta \delta M			,		\Delta \delta M		\,\rangle_{\Hh^{-\frac{1}{2}}}\big|
		\leq &
		\bigg|\sum_{q\in\ZZ}
		\sum_{|j-q|\leq 5}
		2^{-q}\int_{\TT^2} [\Dd_q,\Sd_{j-1} M_2^h]\wedge\Dd_j \Delta  \delta M \cdot \Dd_q \Delta \delta M
		\bigg|\\
		&+\bigg|
		\sum_{q\in\ZZ}
		\sum_{|j-q|\leq 5}
		2^{-q}\int_{\TT^2} (\Sd_{j-1} M_2^h-\Sd_{q-1} M_2^h)\wedge\Dd_q\Dd_j \Delta  \delta M \cdot \Dd_q \Delta \delta M
		\bigg|\\
		&+\bigg|
		\sum_{q\in\ZZ}
		\sum_{j\geq q-5}2^{-q}\int_{\TT^2}\Dd_q (\Dd_j M_2^h\wedge \Sd_{j+2}\Delta  \delta M) \cdot \Dd_q \Delta \delta M\bigg| =: \I + \I\I + \I\I\I,
	\end{align*}
	where we have also used the fact that
	\begin{equation*}
		\sum_{q\in\ZZ} 2^{-q}\int_{\TT^2} (\Sd_{q-1}M_2^h\wedge \Dd_q \Delta  \delta M) \cdot \Dd_q \Delta \delta M = 0.
	\end{equation*}
	We begin with the estimate of $\I$ and we remark that by Lemma~\ref{l:commut}
	\begin{equation*}
	\begin{aligned}
		\big|\sum_{q\in\ZZ}
		&\sum_{|j-q|\leq 5}
		2^{-q}\int_{\TT^2} [\Dd_q,\Sd_{j-1} M_2^h]\wedge\Dd_j \Delta  \delta M \cdot \Dd_q \Delta \delta M
		\big|\\
		&\lesssim
		\sum_{q\in\ZZ}
		\sum_{|j-q|\leq 5}
		2^{-q}
		\| \Sd_{j-1} \nabla M_2 \|_{L^4} \| \Dd_j \Delta \delta M \|_{L^2}
		2^{-q}\| \Dd_q \Delta \delta M \|_{L^4}\\
		&\lesssim
		\sum_{q\in\ZZ}
		\sum_{|j-q|\leq 5}
		2^{-q}
		\| \nabla M_2 \|_{L^4} \| \Dd_j \Delta \delta M \|_{L^2}
		\| \Dd_q \nabla \delta M \|_{L^2}^\frac{1}{2}
		\| \Dd_q \Delta \delta M \|_{L^2}^\frac{1}{2}\\
		&\lesssim
		\| \nabla M_2 \|_{L^2}^\frac{1}{2}
		\| \Delta M_2 \|_{L^2}^\frac{1}{2}
		\| \nabla \delta M \|_{\Hh^{-\frac{1}{2}}}^\frac{1}{2}
		\| \Delta \delta M \|_{\Hh^{-\frac{1}{2}}}^\frac{3}{2}.
	\end{aligned}
	\end{equation*}
	The above estimate holds true, since $\| \Dd_q \Delta \delta M \|_{L^4} \lesssim 
	2^q\| \Dd_q \nabla \delta M \|_{L^4}\lesssim 
	2^q\| \Dd_q \nabla \delta M \|_{L^2}^\frac{1}{2}\| \Dd_q \Delta \delta M \|_{L^2}^\frac{1}{2}$, thanks to Lemma~\ref{l:bern}. Hence, the following estimate holds true
	\begin{equation*}
		\I \lesssim  \| \nabla M_2 \|_{L^2}^2
		\| \Delta M_2 \|_{L^2}^2
		\| \nabla \delta M \|_{\Hh^{-\frac{1}{2}}}^2
		+\ee
		\| \Delta \delta M \|_{\Hh^{-\frac{1}{2}}}^2.
	\end{equation*}
	With a similar procedure, one can show that the above inequality is true also when replacing $\I$ by $\I \I$. Hence, it just remains to estimate $\I\I\I$. Using Lemma~\ref{l:bern}, we obtain that
	\begin{equation*}
	\begin{aligned}
		\I\I\I
		&\lesssim 
		\sum_{q\in\ZZ}
			\sum_{j\geq q-5}2^{-q}\|\Dd_q (\Dd_j M_2^h\wedge \Sd_{j+2}\Delta  \delta M)\|_{L^2} \|\Dd_q \Delta \delta M \|_{L^2} \\
			&\lesssim
			\sum_{q\in\ZZ}
			\sum_{j\geq q-5}2^{-q}2^q\|\Dd_q (\Dd_j M_2^h\wedge\Sd_{j+2}\Delta  \delta M)\|_{L^1} \|\Dd_q \Delta \delta M \|_{L^2} 
			\\
			&\lesssim
			\sum_{q\in\ZZ}
			\sum_{j\geq q-5}\|\Dd_j M_2^h \|_{L^2} \|\Sd_{j+2}\Delta  \delta M\|_{L^2} \|\Dd_q \Delta \delta M \|_{L^2} 
			\\
			&\lesssim
			\sum_{q\in\ZZ}
			\sum_{j\geq q-5}2^{-2j}\|\Dd_j \Delta M_2^h \|_{L^2}2^{j+2} \|\Sd_{j+2}\nabla  \delta M\|_{L^2} \|\Dd_q \Delta \delta M \|_{L^2} 
			\\
			&\lesssim
			\sum_{q\in\ZZ}
			\sum_{j\in\ZZ}{\bf 1}_{(-\infty, 5]}(q-j)2^{\frac{q-j}{2}}\|\Dd_j \Delta M_2^h \|_{L^2} 2^{-\frac{j+2}{2}}\|\Sd_{j+2}\nabla  \delta M\|_{L^2} 2^{-\frac{q}{2}}\|\Dd_q \Delta \delta M \|_{L^2}\\
			&\lesssim
			\|\Delta M_2 \|_{L^2}
			\left(
			\sum_{q\in\ZZ}
			\Big|
			\sum_{j\in\ZZ}{\bf 1}_{(-\infty, 5]}(q-j)2^{\frac{q-j}{2}} 2^{-\frac{j+2}{2}}\|\Sd_{j+2}\nabla  \delta M\|_{L^2}\Big|^2
			\right)^\frac{1}{2}
			\left(
			\sum_{q\in\ZZ}2^{-q}\|\Dd_q \Delta \delta M \|_{L^2}^2
			\right)^\frac{1}{2}.
	\end{aligned}
	\end{equation*}
	Applying the Young inequality to the convolution in $(q,j)\in \mathbb{Z}^2$ and combining the result with Proposition~\ref{prop:s_negative}, we finally deduce that
	\begin{equation*}
			\I\I\I 
			\lesssim
			\|\Delta M_2 \|_{L^2}
			\| \nabla \delta M \|_{\Hh^{-\frac{1}{2}}}
			\|  \Delta \delta M \|_{\Hh^{-\frac{1}{2}}}
			\lesssim
			\|\Delta M_2 \|_{L^2}^2
			\| \nabla \delta M \|_{\Hh^{-\frac{1}{2}}}^2
			+\ee
			\|  \Delta \delta M \|_{\Hh^{-\frac{1}{2}}}^2
	\end{equation*}
	and this concludes the proof of the proposition.
\end{proof}
\noindent Next, we address the proof of Lemma~\ref{comm-est}, 
which we recall in the following statement.
	\begin{lemma}\label{lemma6.2} There exists a $C>0$ such that for any divergence-free vector field $\rm v$ in $\Hh^{1}(\TT^2)$ and any vector field $B$ in $\Hh^{\frac{1}{2}}(\TT^2)$, it holds
		\begin{equation*}
			\left|\langle\,  {\rm v}\cdot \nabla B,\, B\,\rangle_{\Hh^{-\frac{1}{2}}(\TT^2)}\right|\,\leq\, C \|\,\nabla {\rm v}\,\|_{L^2}\|\,\nabla B\,\|_{\Hh^{-\frac{1}{2}}}\|\,B\,\|_{\Hh^{-\frac{1}{2}}}.
		\end{equation*}
	\end{lemma}
\begin{proof}
	The proof is similar to the one of Proposition~\ref{prop:blabla}, we first recast the inner product with the formula given by \eqref{inner_product}, then we apply the Bony decomposition \eqref{BONY}:
	\begin{align*}
		\langle\,  {\rm v}\cdot \nabla B,\, B\,\rangle_{\Hh^{-\frac{1}{2}}(\TT^2)} =&
		\sum_{q\in \ZZ}2^{-q} \int_{\TT^2} \Dd_q  ({\rm v}\cdot \nabla B) \cdot \Dd_q B
		=
		\sum_{q\in \ZZ}\sum_{|j-q|\leq 5}2^{-q} \int_{\TT^2} [\Dd_q, \Sd_{j-1} {\rm v}] \Dd_j\nabla B \cdot \Dd_q B\\ &+
		\sum_{q\in \ZZ}\sum_{|j-q|\leq 5}2^{-q} \int_{\TT^2} (\Sd_{j-1} {\rm v}-\Sd_{q-1} {\rm v})\cdot \Dd_q\Dd_j\nabla B \cdot \Dd_q B+
		\sum_{q\in \ZZ}2^{-q} \int_{\TT^2} \Sd_{q-1} {\rm v}\cdot \Dd_q\nabla B \cdot \Dd_q B\\ &+
		\sum_{q\in \ZZ}\sum_{j\geq q- 5}2^{-q} \int_{\TT^2} \Dd_q(\Dd_j {\rm v}\cdot \Sd_{j+2}\nabla B) \cdot \Dd_q B =: \I_1+\I_2+\I_3+\I_4.
	\end{align*}
	Using Lemma~\ref{l:commut} and recalling that the gradient $\nabla$ and the operator $\Sd_{j-1}$ commute, the first term $\I_1$ is bounded by
	\begin{equation*}
		\begin{aligned}
			\big|\I_1\big|
			&\lesssim 
			\sum_{q\in \ZZ}2^{-q}  \sum_{|j-q|\leq 5}2^{-q} \| \Sd_{j-1} \nabla {\rm v} \|_{L^\infty} \| \Dd_j \nabla B \|_{L^2} \| \Dd_q B \|_{L^2}\\
			&\lesssim 
			\sum_{q\in \ZZ}2^{-q}  \sum_{|j-q|\leq 5} \| \Sd_{j-1} \nabla {\rm v} \|_{L^2} \| \Dd_j \nabla B \|_{L^2} \| \Dd_q B \|_{L^2}\\
			&\lesssim 
			\|\,\nabla {\rm v}\,\|_{L^2}\|\,\nabla B\,\|_{\Hh^{-\frac{1}{2}}}\|\,B\,\|_{\Hh^{-\frac{1}{2}}}.
		\end{aligned}	
	\end{equation*}
	Similarly, one can show that
	\begin{equation*}
		\big|\I_2\big| \lesssim \sum_{q\in \ZZ}2^{-q}  \sum_{|j-q|\leq 5} \| (\Sd_{j-1} -\Sd_{q-1})\nabla {\rm v} \|_{L^2} \| \Dd_q\Dd_j \nabla B \|_{L^2} \| \Dd_q B \|_{L^2}\lesssim 
			\|\,\nabla {\rm v}\,\|_{L^2(\TT^2)}\|\,\nabla B\,\|_{\Hh^{-\frac{1}{2}}(\TT^2)}\|\,B\,\|_{\Hh^{-\frac{1}{2}}(\TT^2)}.
	\end{equation*}
	The third term $\I_3$ is identically null, because of the free divergence condition on ${\rm v}$:
	\begin{equation*}
		\I_3=\sum_{q\in \ZZ}2^{-q} \int_{\TT^2} \Sd_{q-1} {\rm v}\cdot \Dd_q\nabla B \cdot \Dd_q B=0.
	\end{equation*}
	Finally, the last term $\I_4$ is estimated as follows:
	\begin{equation*}
		\begin{aligned}
			\big|\I_4\big|
			&\leq 
			\sum_{q\in\ZZ}
			\sum_{j\geq q-5}2^{-q}\|\Dd_q (\Dd_j {\rm v}\cdot \Sd_{j+2}\nabla B)\|_{L^2} \|\Dd_q B \|_{L^2} \\
			&\lesssim
			\sum_{q\in\ZZ}
			\sum_{j\geq q-5}2^{-q}2^q\|\Dd_q (\Dd_j {\rm v}\cdot \Sd_{j+2}\nabla B)\|_{L^1} \|\Dd_q B \|_{L^2} 
			\\
			&\lesssim
			\sum_{q\in\ZZ}
			\sum_{j\geq q-5}\|\Dd_j {\rm v} \|_{L^2} \|\Sd_{j+2}\nabla B\|_{L^2} \|\Dd_q B \|_{L^2} 
			\\
			&\lesssim
			\sum_{q\in\ZZ}
			\sum_{j\in\ZZ}{\bf 1}_{(-\infty, 5]}(q-j)2^{\frac{q-j}{2}}\|\Dd_j \nabla {\rm v} \|_{L^2} 2^{-\frac{j+2}{2}}\|\Sd_{j+2}\nabla B\|_{L^2} 2^{-\frac{q}{2}}\|\Dd_q B \|_{L^2}.
	\end{aligned}
	\end{equation*}
	Applying the Young inequality to the last convolution and combining the result with Proposition~\ref{prop:s_negative}, we finally deduce that
	\begin{equation*}
			\big|\I_4\big|
			\lesssim
			\|\,\nabla {\rm v}\,\|_{L^2}\|\,\nabla B\,\|_{\Hh^{-\frac{1}{2}}}\|\,B\,\|_{\Hh^{-\frac{1}{2}}}
	\end{equation*}
	and this concludes the proof of the lemma.
\end{proof}
\noindent It remains to investigate inequality \eqref{log-ineq}, that we recall in the following lemma.
\begin{lemma}\label{lemma:log-ineq}
    Let $\delta W'(F), F_1, F_2, \delta F$ and $\delta u$ be as in Section~\ref{sec:uniqueness}. The following estimate holds true:
    \begin{equation}\label{log-ineq2}
	\begin{aligned}
		\Big|
		\langle	\,	 \delta W'(F)& F_1^\top+ W'(F_2) \delta F^\top	,	\nabla \delta \uu	\,	\rangle_{\Hh^{-\frac{1}{2}}}
		\Big|
		\lesssim
		(1+\| F_2 \|_{L^2})^6\| (F_1,\,F_2) \|_{H^1}^2
		\delta \EE(t)
		\ln\Big(1+\frac{1}{\delta \EE(t)}\Big)
		\\&+\,
		C
		\Big(
		\| (F_1,\,F_2,\,\uu_1,\,\uu_2)\,\|_{H^1}^\frac{3}{2}
		+
		(1+\| F_2 \|_{L^2})^6\| (F_1,\,F_2) \|_{H^1}^2
		\Big)
		\delta \EE(t)
		\,
		+\,
		\ee
		\|\,	\nabla	\delta 	\uu		\,\|_{H^{-\frac{1}{2}}	}^2+
		\ee
		\|\,	\nabla	\delta 	F		\,\|_{H^{-\frac{1}{2}}	}^2,
	\end{aligned}
	\end{equation}
	for a small positive parameter $\ee>0$.
\end{lemma}
\begin{proof}
    We begin with controlling the second term of the left-hand side, which is
    \begin{equation*}
    \begin{aligned}
        |\langle W'(F_2) \delta F^\top	,	\nabla \delta \uu	\,	\rangle_{\Hh^{-\frac{1}{2}}}|
        &\leq
        \| W'(F_2) \delta F^\top \|_{\Hh^{-\frac 12}}
        \| \nabla \delta \uu\|_{\Hh^{-\frac 12}}
        \lesssim
        \| W'(F_2) \|_{\Hh^\frac{1}{4}}
        \| \delta F^\top \|_{\Hh^\frac{1}{4}}
        \| \nabla \delta \uu\|_{\Hh^{-\frac 12}}\\
        &\lesssim
        \| W'(F_2) \|_{L^2}^\frac{3}{4}
        \| \nabla W'(F_2) \|_{L^2}^\frac{1}{4}
        \| \delta F \|_{\Hh^{-\frac 12}}^\frac{1}{4}
        \| \nabla \delta F \|_{\Hh^{-\frac 12}}^\frac{3}{4}
        \| \nabla \delta \uu\|_{\Hh^{-\frac 12}}\\
        &\lesssim
        (1+\| F_2 \|_{L^2})^\frac{3}{4}
        \| W''(F_2)\nabla F_2 \|_{L^2}^\frac{1}{4}
        \| \delta F \|_{\Hh^{-\frac 12}}^\frac{1}{4}
        \| \nabla \delta F \|_{\Hh^{-\frac 12}}^\frac{3}{4}
        \| \nabla \delta \uu\|_{\Hh^{-\frac 12}}\\
        &\lesssim
        (1+\| F_2 \|_{L^2})^\frac{3}{4}
        \| \nabla F_2 \|_{L^2}^\frac{1}{4}
        \| \delta F \|_{\Hh^{-\frac 12}}^\frac{1}{4}
        \| \nabla \delta F \|_{\Hh^{-\frac 12}}^\frac{3}{4}
        \| \nabla \delta \uu\|_{\Hh^{-\frac 12}}\\
        &\lesssim
        (1+\| F_2 \|_{L^2})^6
        \| \nabla F_2 \|_{L^2}^2
        \| \delta F \|_{\Hh^{-\frac 12}}^2
        +\ee
        \| \nabla \delta F \|_{\Hh^{-\frac 12}}^2
        +
        \ee
        \| \nabla \delta \uu\|_{\Hh^{-\frac 12}}^2.
    \end{aligned}
    \end{equation*}
    We now deal with the first term on the left-hand side of inequality \eqref{log-ineq2}. First we split the matrix $F_1^\top$ into two main terms, addressing the low and high frequencies of $F_1$: $F_1^\top = \Sd_N F_1^\top + (\Id - \Sd_N)F_1^\top$, for a suitable radius $N$, to be fixed later. Hence, we remark that
    \begin{equation*}
        |\langle \delta W'(F) \Sd_N F_1^\top, \nabla \delta \uu \rangle_{\Hh^{-\frac{1}{2}}}|
        \leq 
        \| \delta W'(F)\Sd_N F_1^\top \|_{L^2} \|\nabla \delta \uu \|_{\Hh^{-1}}
        \leq 
        \| \delta W'(F) \|_{L^2}\| \Sd_N F_1^\top \|_{L^\infty} \|\delta \uu^h \|_{L^2}.
    \end{equation*}
    Since $\|f \|_{L^\infty} \lesssim \| f \|_{H^1}(1+\sqrt{N})$ for any function $f$ whose Fourier transform is supported in the ball  $B_N(0)$, we deduce that
    \begin{equation*}
    \begin{aligned}
        |\langle \delta W'(F) \Sd_N F_1^\top, \nabla \delta \uu \rangle_{\Hh^{-\frac{1}{2}}}|
        &\lesssim 
        \| \delta W'(F) \|_{L^2}\| F_1\|_{H^1}(1+\sqrt{N})
        \|\delta \uu^h \|_{\Hh^{-\frac{1}{2}}}^\frac{1}{2}
        \|\nabla \delta \uu \|_{\Hh^{-\frac{1}{2}}}^\frac{1}{2}\\
        &\lesssim 
        \| \delta F \|_{L^2}\| F_1\|_{H^1}(1+\sqrt{N})
        \|\delta \uu^h \|_{\Hh^{-\frac{1}{2}}}^\frac{1}{2}
        \|\nabla \delta \uu \|_{\Hh^{-\frac{1}{2}}}^\frac{1}{2}\\
        &\lesssim 
        \| \delta F \|_{\Hh^{-\frac{1}{2}}}^\frac{1}{2}\| \nabla \delta F \|_{\Hh^{-\frac{1}{2}}}^\frac{1}{2}\| F_1\|_{H^1}(1+\sqrt{N})
        \|\delta \uu^h \|_{\Hh^{-\frac{1}{2}}}^\frac{1}{2}
        \|\nabla \delta \uu \|_{\Hh^{-\frac{1}{2}}}^\frac{1}{2}\\
        &\lesssim 
        \| F_1\|_{H^1}^2
        (\| \delta F \|_{\Hh^{-\frac{1}{2}}}^2+\|\delta \uu^h \|_{\Hh^{-\frac{1}{2}}}^2)(1+N)
        +
        \ee
        \| \nabla \delta F \|_{\Hh^{-\frac{1}{2}}}^2+
        \ee
        \|\nabla \delta \uu \|_{\Hh^{-\frac{1}{2}}}^2.
    \end{aligned}
    \end{equation*}
    Next, we consider the high frequencies of $F_2^\top$, to gather the following inequality:
    \begin{equation*}
        \begin{aligned}
            |\langle \delta W'(F) (\Id -\Sd_N )F_1^\top, \nabla \delta \uu \rangle_{\Hh^{-\frac{1}{2}}}|
            &\leq
            \|  \delta W'(F) (\Id -\Sd_N )F_1^\top \|_{L^2}
            \| \delta \uu \|_{L^2}\\
            &\lesssim
            \| \delta W'(F)\|_{L^4}
            \| (\Id -\Sd_N )F_1^\top\|_{L^4}
            \| \delta \uu \|_{L^2}\\
            &\lesssim
            \| \delta F\|_{L^4}
            \Big( \sum_{q \geq N}\| \Dd_q F_1\|_{L^4}\Big)
           \| \delta \uu \|_{L^2}\\
            &\lesssim
            \| \delta F\|_{L^2}^\frac{1}{2}
            \|\nabla \delta F\|_{L^2}^\frac{1}{2}
            \Big( \sum_{q \geq N}\| \Dd_q F_1\|_{L^4}\Big)
            \| \delta \uu \|_{L^2}\\
            &\lesssim
            \| \delta F\|_{L^2}^\frac{1}{2}
            \|\nabla \delta F\|_{L^2}^\frac{1}{2}
            \Big( \sum_{q \geq N}2^{\frac{q}{2}}\| \Dd_q F_1\|_{L^2}\Big)
            \| \delta \uu \|_{L^2}\\
             &\lesssim
            \| \delta F\|_{L^2}^\frac{1}{2}
            \|\nabla \delta F\|_{L^2}^\frac{1}{2}
            \Big( \sum_{q \geq N}2^{-\frac{q}{2}}\| \Dd_q \nabla F_1\|_{L^2}\Big)
            \| \delta \uu \|_{L^2}\\
            &\lesssim
            \| \delta F\|_{L^2}^\frac{1}{2}
            \|\nabla \delta F\|_{L^2}^\frac{1}{2}
            \| \nabla F_1 \|_{L^2}^\frac{1}{2}
            \| \delta \uu \|_{L^2}
            2^{-\frac{N}{2}}.
        \end{aligned}
    \end{equation*}
    We hence set up the radius $N>0$ to satisfy
    \begin{equation*}
        N := 2\log_2 e \ln\Big( 1+\frac{1}{\delta \EE(t)}\Big),
    \end{equation*}
    from which we deduce that
    \begin{equation*}
        \begin{aligned}
            |\langle \delta W'(F) (\Id -\Sd_N )F_1^\top, \nabla \delta u \rangle_{\Hh^{-\frac{1}{2}}}|
            &\lesssim
            \| \delta F\|_{L^2}^\frac{1}{2}
            \|\nabla \delta F\|_{L^2}^\frac{1}{2}
            \| \nabla F_2 \|_{L^2}^\frac{1}{2}
            \| \delta \uu \|_{L^2}
            \delta \EE(t).
        \end{aligned}
    \end{equation*}
    Summarizing the previous inequalities and recalling that 
    \begin{equation*}
        \| (\uu_1,\uu_2,F_1,F_2) \|_{H^1} = \| \uu_1 \|_{H^1} + \| \uu_2 \|_{H^1} +
        \| F_1 \|_{H^1} + \| F_2 \|_{H^1},
    \end{equation*}
    we finally achieve the estimate \eqref{log-ineq2}, which concludes the proof of the lemma.
\end{proof}

\vspace{1em}
\noindent{\textbf{Acknowledgment}.} This work is partially funded by the Deutsche Forschungsgemeinschaft (DFG, German Research Foundation), grant SCHL 1706/4-2, project number 391682204.

\end{document}